\documentclass[11pt]{amsart}
\input xy
\xyoption{all}

\usepackage{amsmath,amsfonts,amsthm,amssymb,latexsym,amscd,amsopn,graphics,color,enumerate,lscape,mathtools,mathrsfs,placeins,stmaryrd}
\usepackage{hyperref}
\usepackage{tikz-cd}
\usepackage{todonotes}
\usepackage{comment}
\usepackage{enumitem}
\usepackage{rotating}
\usepackage[nocompress,noadjust]{cite}
\usepackage{stmaryrd}
\usepackage[margin=1in]{geometry}

\allowdisplaybreaks

\theoremstyle{plain}

\newtheorem{theorem}{Theorem}

  \newtheorem{prop}[theorem]{Proposition}
  
  \newtheorem{lemma}[theorem]{Lemma}
  
  \newtheorem{question}[theorem]{Question}
\theoremstyle{definition}

\theoremstyle{remark}
 \newtheorem*{remark}{Remark}

\newtheoremstyle{TheoremNum}
        {8.0pt plus 2.0pt minus 4.0pt}
        {8.0pt plus 2.0pt minus 4.0pt}
        {\itshape}
        {}
        {\bfseries}
        {.}
        {.5em}
        {\thmname{#1}\thmnote{ \bfseries #3}}
\theoremstyle{TheoremNum}

\makeatletter
\def\@tocline#1#2#3#4#5#6#7{\relax
  \ifnum #1>\c@tocdepth 
  \else
    \par \addpenalty\@secpenalty\addvspace{#2}%
    \begingroup \hyphenpenalty\@M
    \@ifempty{#4}{%
      \@tempdima\csname r@tocindent\number#1\endcsname\relax
    }{%
      \@tempdima#4\relax
    }%
    \parindent\z@ \leftskip#3\relax \advance\leftskip\@tempdima\relax
    \rightskip\@pnumwidth plus4em \parfillskip-\@pnumwidth
    #5\leavevmode\hskip-\@tempdima
      \ifcase #1
       \or\or \hskip 30pt \or \hskip 2em \else \hskip 3em \fi%
      #6\nobreak\relax
    \hfill\hbox to\@pnumwidth{\@tocpagenum{#7}}\par
    \nobreak
    \endgroup
  \fi}
\makeatother

\DeclareMathOperator{\Sym}{Sym}

\DeclareMathOperator{\Stab}{Stab}
\def\Z{{\mathbb Z}}

\def\B{{\mathcal B}}

\def\Z{{\mathbb Z}}

\def\GL{{\rm GL}}

\def\bs{{\rm bs}}
\def\Sym{{\rm Sym}}

\def\Stab{{\rm Stab}}
\def\diag{{\rm diag}}
\def\inv{{\rm inv}}

\def\P{{\mathbb P}}

\def\red{{\rm red}}
\def\dist{{\rm red}}
\def\Vol{{\rm Vol}}

\def\R{{\mathbb R}}
\def\C{{\mathbb C}}
\def\F{{\mathbb F}}
\def\bR{{\mathbb R}}

\def\FF{{\mathcal F}}
\def\RR{{\mathcal R}}

\def\Q{{\mathbb Q}}

\def\Z{{\mathbb Z}}
\def\P{{\mathbb P}}
\def\F{{\mathbb F}}
\def\Q{{\mathbb Q}}

\def\cJ{{\mathcal J}}

\def\cB{{\mathcal B}}

\def\cF{{\mathcal F}}

\def\GG{G}
\def\cR{{\mathcal R}}

\def\Z{{\mathbb Z}}

\newcommand{\PGL}{\mathrm{PGL}}

\newcommand{\SO}{\mathrm{SO}}

\newcommand{\nc}{\newcommand}
\nc{\on}{\operatorname}
\nc{\renc}{\renewcommand}
\nc{\wt}{\widetilde}
\nc{\defeq}{\vcentcolon=}
\nc{\eqdef}{=\vcentcolon}
\nc{\Spec}{\on{Spec}}
\nc{\ol}{\overline}
\renc{\d}{\partial}
\nc{\mc}{\mathcal}

\def\AA{{\mathcal A}}
\def\stabsize{{\theta}}
\def\fundset{{\mathcal R}}

\makeatletter
\@namedef{subjclassname@2020}{%
  \textup{2020} Mathematics Subject Classification}
\makeatother

\newcommand\rddots
   {\mathinner{\mkern1mu
       \raise 1pt\hbox{.}\mkern2mu
       \raise 4pt\hbox{.}\mkern2mu
       \raise 7pt\hbox{.}\mkern1mu}}

\title[Geometry-of-Numbers Methods in the Cusp]{Geometry-of-Numbers Methods in the Cusp
\vspace*{-0.1in}
}
\author{Arul Shankar, Artane Siad, Ashvin A.~Swaminathan, Ila Varma}


\AtEndDocument{\medskip{\footnotesize
  \noindent\textsc{Department of Mathematics, University of Toronto, \mbox{Toronto, ON M5S 2A3}} \par
  \textit{E-mail address}, Arul Shankar: \texttt{arul.shnkr@gmail.com}
 \\[0.1cm]
  \noindent\textsc{Department of Mathematics, Princeton University, \mbox{Princeton, NJ 08544, \&}} \par
  \noindent\textsc{School of Mathematics, Institute for Advanced Study, \mbox{Princeton, NJ 08540}} \par
  \textit{E-mail address}, Artane Siad: \texttt{as4426@princeton.edu} \\[0.1cm]
  \noindent\textsc{Department of Mathematics, Harvard University, \mbox{Cambridge, MA 02138}} \par
  \textit{E-mail address}, Ashvin A. Swaminathan: \texttt{swaminathan@math.harvard.edu}\\[0.1cm]
  \noindent\textsc{Department of Mathematics, University of Toronto, \mbox{Toronto, ON M5S 2A3}} \par\vspace*{-0.1cm}
  \textit{E-mail address}, Ila Varma: \texttt{ila@math.toronto.edu}
}}
\date{\today}

\begin{document}

\vspace*{-0.2in}

\maketitle

\vspace*{-0.4in}
\begin{abstract}
%
In this article, we develop new methods for counting integral orbits having bounded invariants that lie inside the cusps of fundamental domains for coregular representations. 
We illustrate these methods for a representation of cardinal interest in number theory, namely that of the split orthogonal group acting on the space of quadratic forms.
\end{abstract}

\tableofcontents

\vspace{-30pt}
\section{Introduction} \label{sec-intro}

A \emph{coregular representation} $(G,W)$ consists of a reductive algebraic group $G$, defined over $\Z$, and a finite-dimensional representation $W$ of $G$, also defined over $\Z$, such that the ring of polynomial invariants for the action of the semisimplification of $G$ on $W$ is freely generated (say by the elements $p_1,\ldots,p_k$). We then say that $U={\rm Aff}^k$ is the {\it space of relative invariants}, and note that for any ring $R$, we have the natural map $W(R)\to U(R)$ given by $w\mapsto (p_1(w),\ldots,p_k(w))$.\footnote{More algebro-geometrically, the ring of relative invariants is $\Z[W]^G$ while the space  of relative invariants is $U \defeq W \sslash G = \Spec\left(\Z[W]^G \right)$.} is freely and finitely generated over $\Z$. Many objects of interest in number theory and arithmetic geometry admit natural parametrizations in terms of integral orbits of coregular representations.
Typically, these parametrizations impose the following three pieces of structure on the pair $(G,W)$ and the associated \mbox{invariant space $U$:}
\begin{enumerate}[leftmargin=20pt]
    \item an algebraic notion of {\it nondegeneracy} for the orbits of $G$ on $W$;
    \item a natural notion of {\it height} on $U(\R)$, which then lifts to a $G(\R)$-invariant notion of height on $W(\R)$;
    \item an arithmetic notion of {\it irreducibility} for the orbits of $G(\Z)$ on $W(\Z)$.
\end{enumerate}
We note that {\it prehomogeneous} representations, i.e., representations whose rings of invariants are generated by a single element, are all coregular. In that case, the single generating invariant (usually termed the discriminant) is the height that is used.

A landmark result of Borel and Harish-Chandra \cite{MR147566} implies that the number of nondegenerate $G(\Z)$-orbits on $W(\Z)$ having bounded height is finite. In light of this result, it is natural to ask the following fundamental question: {\it what are the asymptotics for the number of nondegenerate $G(\Z)$-orbits on $W(\Z)$ having bounded height?} Answering this question for just those integral orbits that are \emph{irreducible} has proven to be a problem with significant applications in arithmetic statistics, the study of distributions of arithmetic objects. Indeed, counting integral irreducible orbits of coregular representations played a key role in the proofs of many breakthrough results, from determining asymptotics for $S_n$-number fields of small fixed degree ordered by discriminant (see~\cite{MR491593,MR2183288,MR2745272,MR4035297}), to computing average sizes of $p$-class groups in certain families of number fields for small primes $p$ (see~\cite{MR491593, MR2183288,MR3782066,BSHpreprint,Siadthesis1,Siadthesis2,Swpreprint,BSSpreprint}), to calculating average $n$-Selmer ranks of elliptic curves and hyperelliptic Jacobians for small $n$ \mbox{(see~\cite{MR3272925, MR3275847, 1312.7333, 1312.7859,MR3156850,MR3719247,BSSpreprint, MR3968769, Laga1, Laga2}), and many more.}

More recently, counting {\it reducible} integral orbits of coregular representations has emerged as a problem of significant interest in its own right, yielding applications toward determining the sizes of the $3$-torsion in the class groups of quadratic orders \cite{MR289428,MR3471250}, counting octic $D_4$-number fields ordered by Artin conductor \cite{D4preprint} and by discriminant \cite{MOpreprint}, studying families of elliptic curves ordered by conductor \cite{MR4277110}, and carrying out squarefree sieves on families of polynomials \cite{sqfrval} and binary $n$-ic forms \cite{sqfrval2}. However, in all of these applications, the counts of reducible orbits were obtained using \emph{ad hoc} methods. Furthermore, the latter three results cited above do not actually prove asymptotics, and merely obtain upper bounds on the number of reducible orbits of bounded height. This is in stark contrast with the case of irreducible orbits, for which systematic methods have been developed to obtain precise asymptotics with power-saving error terms.

The purpose of this article is to devise new systematic techniques for counting reducible orbits. We illustrate our techniques for a representation that features prominently in the literature on arithmetic statistics, namely the action of the orthogonal group of the split $n$-ary quadratic form on the space of quadratic forms in $n$ variables, where $n \geq 3$ is an arbitrary fixed integer. In particular, we provide a complete answer to the fundamental question stated above for each representation in this infinite family indexed by $n$; see Section~\ref{sec-states} for precise statements of our main theorems. Before proving these general theorems, we provide the reader with a gentle introduction to our new method by illustrating how it applies in the context of a low-dimensional example, namely the action of $\on{PGL}_2$ on the space $\Sym_4(2)$ of binary quartic forms; see \S\ref{sec-binaryquartics}.

\subsection{Background on orbit-counting, and relation to earlier work} \label{sec-thebackground}

In this section, we summarize the orbit-counting methods that feature in the literature on arithmetic statistics leading up to the present article, and we describe the context for our new methods \mbox{of counting reducible orbits.}

\subsubsection*{The critical dichotomy} \label{sec-genapp}

Let $(G,W)$ be a coregular representation with space of relative invariants $U$, and suppose that a family of arithmetic objects can be parametrized in terms of certain $G(\Z)$-orbits on $W(\Z)$. We note that a natural source of examples of such coregular representations are {\it prehomogeneous} representations, i.e., the ring of relative invariants is generated by a single element, usually called the {\it discriminant}.

Further suppose that the space $U(\R)$ admits a natural notion of height, and define the height of (the $G(\R)$-orbit of) an element $w \in W(\R)$ to be the height of the image of $w$ under the natural map $W(\R) \to U(\R)$. In the prehomogeneous case, this height is usually taken to be the absolute value of the discriminant.
Under the above assumptions, the problem of counting arithmetic objects in the family with bounded height can be translated into the problem of counting lattice points of bounded height in a fundamental domain $\mc{D}$ for the action of $G(\Z)$ on $W(\R)$. This latter problem is complicated by the fact that the fundamental domains $\mc{D}$ is usually not compact, even for the subset $\mc{D}_X$ of points in $\mc{D}$ with height less than $X$. Indeed, the height-bounded fundamental domain $\mc{D}_X$ typically consists of a bounded region known as the \emph{main body}, along with one or more \emph{cusps}, which are long tentacle-like regions going off to infinity.

Surprisingly, in most of the coregular representations considered in the literature thus far, the following dichotomy holds:
\begin{enumerate}[leftmargin=20pt]
\item A proportion of $100\%$ of irreducible orbits lie in the main body; and a proportion of $100\%$ of points in the main body are irreducible.
\item A proportion of $100\%$ of reducible orbits lie in the cusp; and a proportion of $100\%$ of points in the cusp are reducible.
\end{enumerate}
If Properties (1) and (2) above hold, then the count of points in the main body of $\mc{D}_X$ gives an asymptotic for the number of irreducible $G(\Z)$-orbits on $W(\Z)$ having height less than $X$. This main body count can be determined in a systematic way: using geometry-of-numbers methods, the count can be expressed in terms of the volume of the main body, and this volume can then be computed by performing a suitable change-of-variables. The difficulty in counting the reducible orbits, which predominate in the cusp(s) of $\mc{D}_X$, is that geometry-of-numbers methods do not directly apply in such regions. In fact, the asymptotic count of reducible orbits is {\it not} given in terms of the volume of the corresponding cuspidal regions. In the examples treated in this paper, we find that the volume of the cuspidal region is an underestimate, and that the actual answer is given in terms of a certain weighted volume.

\subsubsection*{Historical context} \label{sec-history}

The development and use of the orbit-counting strategy summarized above goes back centuries. Mertens \cite{MR1579608} and Siegel \cite{MR12642} studied the action of $\on{GL}_2$ on the space $\on{Sym}_2(2)$ of binary quadratic forms, which has a unique polynomial invariant, namely the discriminant. They developed geometry-of-numbers methods to count the number of $\GL_2(\Z)$-orbits on $\on{Sym}_2 \Z^2$ with bounded discriminant, resolving conjectures of Gauss on the average sizes of class numbers of quadratic orders. In the series of papers~\cite{MR43822,MR1574296,MR43823}, Davenport considered the action of $\on{GL}_2$ on the space $(\on{Sym}^3(2))^*$ of binary cubic forms, which also has the discriminant as its unique polynomial invariant. Using the strategy described above, he obtained asymptotics for the number of irreducible $\GL_2(\Z)$-orbits on $(\on{Sym}^3 \Z^2)^*$ having bounded discriminant, and in collaboration with Heilbronn, he combined these asymptotics with results from class field theory and sieve methods to determine the density of discriminants of cubic fields~\cite{MR491593}.

The main obstruction to generalizing the work of Mertens, Siegel, and Davenport--Heilbronn to other representations was proving that the number of irreducible points in the cusps is negligible. 
A pioneering advance was made by Bhargava (see~\cite{MR2745272}), who introduced an ``averaging'' method that has the effect of thickening the cusps, making them a bit more amenable to doing geometry-of-numbers. Bhargava's averaging method, which applies to general coregular representations, gives a systematic way to bound the number of points in cuspidal regions and thus opened the door to study the distributions of a wide variety of arithmetic objects. However, the averaging method has, until the present paper, only been used to obtain upper bounds in the cusp and new methods are needed to obtain precise asymptotics.

\subsection{Notation and setup}
We first study a concrete representation in \S\ref{sec-binaryquartics}, before proceeding to our main results on families of representations. The concrete representation we consider is that of $\PGL_2$ acting on binary quartic forms. The notation in this case is quite manageable, and largely follows that of~\cite{MR3272925}.
In this section, we introduce the other (family of) representations studied in this paper, and we set up the notation necessary to state our main results.
Let $\mc{A}$ be the anti-diagonal $n \times n$ matrix with all anti-diagonal entries equal to $1$; if viewed as a quadratic form, we note that that $\mc{A}$ is split (i.e., it has a maximal isotropic space defined over $\Q$) and unimodular.

\subsubsection*{The representation}

We first define the group $\GG$, which is slightly different for $n$ odd and $n$ even:
\begin{itemize}[leftmargin=20pt]
\item When $n$ is odd, we take $G \defeq \on{SO}_{\mc{A}}$ to be the split special orthogonal group scheme over $\Z$ corresponding to $\mc{A}$. That is, we have \mbox{$G(R) = \{g \in \on{SL}_n(R) : \,  g^t\mc{A}\,g = \mc{A}\}$} for any $\Z$-algebra $R$.
\item When $n$ is even, we take $G \defeq \on{O}_{\mc{A}}/\mu_2$ to be the split projective orthogonal group scheme over $\Z$ corresponding to $\mc{A}$. That is, we take $G$ to be the cokernel of the inclusion $\mu_2 \hookrightarrow \on{O}_{\mc{A}}$ of group schemes over $\Z$, where \mbox{$\on{O}_{\mc{A}} \defeq \{g \in \on{GL}_n(R) : \,  g^t\mc{A}\,g = \mc{A}\}$ for any $\Z$-algebra $R$.}
\end{itemize}
We now define the representation $W$ of $G$. Let $W$ denote the affine $\Z$-scheme whose $R$-points consist of the set of $n \times n$ symmetric matrices with entries in $R$ (i.e., classically integral quadratic forms over $R$). Then $W$ has a natural structure as an $\GG$-representation, where the (left) action is given by $g \cdot B = (g^{-1})^t B g^{-1}$ for $g\in\GG$ and $B\in W$. We note that $W$ can also be interpreted as the space of self-adjoint operators for $\mc{A}$ over $R$ by identifying a self-adjoint operator $T$ with $B = -\mc{A}T \in W$.

The orbits of the representation of $\GG$ on $W$ have been studied extensively in the literature. For example, Bhargava and Gross~\cite{MR3156850} (resp.\ Shankar and Wang~\cite{MR3719247}) obtained asymptotics for the number of \emph{irreducible} orbits of $\GG(\Z)$ on $W(\Z)$ having bounded height when $n$ is odd (resp., when $n$ is even). These asymptotics have yielded a striking array of applications: they were utilized by the aforementioned authors to bound the average sizes of the $2$-Selmer groups of monic hyperelliptic Jacobians of any given dimension; by Swaminathan to prove that most odd-degree binary forms fail to primitively represent a square~\cite{super}; and by Siad~\cite{Siadthesis1,Siadthesis2} to bound the average size of the $2$-torsion subgroup of the class groups of monogenic fields of any given degree. Furthermore, by proving an upper bound on the \emph{reducible} orbits of this representation, Bhargava, Shankar, and Wang \cite{sqfrval} determined the probability that a monic integral polynomial has squarefree discriminant.

\subsubsection*{The invariants} Let $U$ denote the affine $\Z$-scheme whose $R$-points consist of monic degree-$n$ polynomials with coefficients in $R$. For any element $B \in W(R)$, the monic degree-$n$ polynomial
\begin{equation} \label{eq-definvtform}
\on{inv}(B) \defeq (-1)^{\lfloor \frac{n}{2} \rfloor} \det (x \mc{A} + B) \in U(R)
\end{equation}
is invariant under the action of $\GG(R)$; in fact, by~\cite[Section 8.3, part (VI) of Section 13.2]{MR0453824}, its coefficients freely generate the ring of polynomial invariants for the action of $\GG$ on $W$. Given $f \in U(R)$, we write $$\on{inv}^{-1}(f) \defeq \{B \in W(R) : \on{inv}(B) = f\}.$$

\subsubsection*{The notion of nondegeneracy}
Let $R$ be an integral domain. Then we say that a monic polynomial with coefficients in $R$ is {\it nondegenerate} if it has nonzero discriminant. We say that an element $B \in W(R)$ is {\it nondegenerate} if the monic polynomial $\on{inv}(B)$ has nonzero discriminant.

When $R = \Z$ or $\R$, it is convenient to partition the set of nondegenerate elements in $U(R)$ according to the number of real roots, and to lift this partition to $W(R)$. To this end, let $0 < r \leq n$ be odd if $n$ is odd, and let $0 \leq r \leq n$ be even if $n$ is even. For $R = \Z$ or $\R$, we define:
\begin{align*}
U(R)^{(r)} & \defeq \{f \in U(R) : f  \text{ is nondegenerate and exactly $r$ real roots} \}; \\
W(R)^{(r)} & \defeq \{B \in W(R) : \on{inv}(B) \in U(R)^{(r)}\}.
\end{align*}

\subsubsection*{The height}
We order elements of $U(\R)$ and $W(\R)$ by \emph{height}. Given a polynomial $f(x) = x^n + \sum_{i = 1}^n f_i x^{n-i} \in U(\R)$, we define its \emph{height} $\on{H}(f)$ by
$$\on{H}(f) \defeq \max_{1 \leq i \leq n}\{\vert f_i \vert^{1/i}\},$$
and given $B \in W(\R)$, we define its height by $\on{H}(B) \defeq \on{H}(\on{inv}(B))$. For any $X > 0$, we write $$N^{(r)}(X) \defeq \#\{f \in  U(\Z)^{(r)} : \on{H}(f) < X\}.$$ Notice that we have
$$N^{(r)}(X)\sim\mathcal{V}^{(r)}(X)\defeq\Vol(\{f \in  U(\R)^{(r)} : \on{H}(f) < X\}) \asymp X^{\dim W} = X^{\frac{n^2+n}{2}},$$
where the volume is computed with respect to the Euclidean measure on $U(\R)$, normalized so that $U(\Z)$ has covolume $1$.

\subsubsection*{The notion of reducibility}  When $R$ is a principal ideal domain with field of fractions $K$, we can partition the nondegenerate elements in $W(R)$ into two natural $\GG(R)$-invariant subsets. We call (the $\GG(R)$-orbit of) a nondegenerate element $B \in W(R)$ \mbox{\emph{reducible} if}:
\begin{itemize}[leftmargin=20pt]
    \item $\mc{A}$ and $B$, viewed as quadratic forms, share a maximal isotropic subspace over $K$ if $n$ is odd; or
    \item $B$ has an isotropic subspace of dimension $\frac{n-2}{2}$ over $K$ contained within a maximal isotropic subspace for $\mc{A}$ over $K$ if $n$ is even.
\end{itemize}
We say that $B \in W(R)$ is \emph{irreducible} if it is nondegenerate and not reducible. There are a few good arithmetic reasons to think of such elements as reducible. For example, pairs $(\mc{A},B)$ correspond to triples $(R',I,\delta)$, where $R'$ is a rank $n$ ring over $R$, $I$ is an ideal in $R'$, and $\delta\in L^\times/L^{\times2}$, where $L=R'\otimes K$, such that $I^2\subset (\delta)$ and $N(I)^2 = N(\delta)$. Then it was proven in \cite{Siadthesis1,Siadthesis2} that $(\mc{A},B)$ is reducible if and only if $\delta \equiv 1$. As another example, pairs $(\mc{A},B)$ also correspond to classes $\sigma \in H^1(K,J[2])$, where $J$ is the Jacobian of the monic hyperelliptic curve $y^2=\pm\det(\mc{A}+B)$, and $(\mc{A},B)$ is reducible if and only if $\sigma$ is trivial.

Write the coordinates of $W$ as $[b_{ij}]_{1\leq i\leq j\leq n}$. Let $W_0$ be the subscheme of $W$ obtained by setting $b_{ij}=0$ for all $(i,j)$ with $i+j<n$. Observe that, if $R$ is a principal ideal domain, then every element of $W_0(R)$ is reducible; for this reason, we call $W_0$ the \emph{reducible hyperplane}. The reducible hyperplane is sent to itself under the action of the lower-triangular subgroup $\mc{P} \subset \GG$, and so we obtain a well-defined representation of $\mc{P}$ on $W_0$. In what follows, for a $\Z$-algebra $R$ and any subset $S \subset  W(R)$, we sometimes write $S_0 \subset S$ to denote the subset $S\cap W_0(R)$.

\subsection{Statements of main theorems} \label{sec-states}

Having set up the notation, we are now in position to state our first main result. Define the constants $C_n^{{\rm fin}}$ and $C_{n,r}^{{\rm inf}}$ as follows:
\begin{equation*}
C_n^{{\rm fin}}\defeq\begin{cases}
\displaystyle\prod_{i = 1}^{\lfloor \frac{n}{2} \rfloor} \zeta(2i), & \text{ if $2\nmid n$;}
\\\displaystyle
\zeta(\tfrac{n}{2}) \prod_{i = 1}^{\frac{n-2}{2}} \zeta(2i), & \text{ if $2\mid n$;}
\end{cases}
\qquad\quad
C_{n,r}^{{\rm inf}}\defeq
\begin{cases}
\mathcal{V}^{(r)}(1), & \text{ if $2\nmid n$;}
\\
2^{-\frac{n}{2}}
\mathcal{V}^{(r)}(1), & \text{ if $2\mid n$;}
\end{cases}
\end{equation*}
Then we have the following result, which gives the total count of reducible $\GG(\Z)$-orbits on $W(\Z)^{(r)}$:
\begin{theorem} \label{thm-maincount2}
The number of reducible $\GG(\Z)$-orbits on $W(\Z)^{(r)}$ having height up to $X$ is given by
$$C_n^{{\rm fin}} \cdot \big(C_{n,r}^{{\rm inf}} \cdot X^{\frac{n^2+n}{2}}\big) + o\big(X^{\frac{n^2+n}{2}}\big).$$
\end{theorem}
\begin{remark}
The factor $C_{n,r}^{{\rm inf}} \cdot X^{\frac{n^2+n}{2}}$ occurring in Theorem~\ref{thm-maincount2} is an asymptotic for the number of invariant polynomials that arise from orbits of height up to $X$. Notice that $C_{n,r}^{{\rm inf}} \cdot X^{\frac{n^2+n}{2}} \sim N^{(r)}(X)$ when $n$ is odd and that $C_{n,r}^{{\rm inf}} \cdot X^{\frac{n^2+n}{2}} \sim 2^{-\frac{n}{2}}N^{(r)}(X)$ when $n$ is even. The extra factor of $2^{-\frac{n}{2}}$ occurs when $n$ is even because for any $B \in W(\Z)$, the $x^{i}$-coefficient of $\on{inv}(B)$ is divisible by $2$ for every odd number $i \in \{1, 3, \dots, n-1\}$ (see~\cite[Theorem~80]{Siadthesis2}).
\end{remark}
As mentioned above, the corresponding asymptotics for the irreducible $\GG(\Z)$-orbits on $W(\Z)^{(r)}$ have been obtained previously in~\cite[Theorem 10.1 and Equation (10.27)]{MR3156850} when $n$ is odd and in~\cite[Theorem 20 and Equation (38)]{MR3719247} when $n$ is even. The main term exponents for the irreducible and reducible cases are the same, and remarkably, the asymptotics for both cases have the same leading constant $C_n^{\on{fin}}$, up to a rational factor! Even more surprisingly, the manner in which $C_n^{\on{fin}}$ arises is completely different for the irreducible and reducible counts. In the irreducible case, this constant arises from the fundamental volume of the group $\GG$---i.e., the volume with respect to a suitably normalized Haar measure of $\GG(\Z) \backslash \GG(\R)$. However, in the reducible case considered in this paper, it arises from a summation of volumes of lower-dimensional slices of the cuspidal regions of a fundamental domain for $\GG(\Z)$ on $W(\R)$. These slices have negligible volume unless they are very high up in the cusp. Thus, they do not detect most of the main body volume, and this appears to be a genuinely different way of obtaining the constant $C_n^{\on{fin}}$.

In light of the above discussion, we are led to pose the following natural question, which we have answered in the affirmative for the action of $\GG$ on $W$:
\begin{question} \label{quest-two}
Let $(G,W)$ be a coregular representation with natural notions of nondegeneracy, height, and reducibility. Let $N^{{\rm irr}}(X)$ $($resp., $N^{{\rm red}}(X)${}$)$ denote the number of irreducible $($resp., reducible$)$ $G(\Z)$-orbits on $W(\Z)$ having height less than $X$.
\begin{enumerate}[leftmargin=20pt]
\item[{\rm (a)}] Is it true that as $X$ tends to infinity, we have $N^{{\rm irr}}(X)\asymp N^{{\rm red}}(X)$?
\item[{\rm (b)}] If so, then is it true that $\lim_{X\to \infty} \big(N^{{\rm irr}}(X)/N^{{\rm red}}(X)\big)$ is a rational constant?
\end{enumerate}
\end{question}
We note that the answers to both parts of Question~\ref{quest-two} are also ``yes'' for the action of $\on{GL}_2$ on $\on{Sym}_3(2)$ by work of Shintani~\cite[Chapter 2, Section 7, Remark 2]{MR289428} and Bhargava--Varma~\cite[Section 4.1.2]{MR3471250}. It follows by combining our results in \S\ref{sec-binaryquartics} with the irreducible count obtained by Bhargava and Shankar in their paper on binary quartic forms~\cite{MR3272925} that the answers are also ``yes'' for the action of $\on{GL}_2$ on $\on{Sym}_4(2)$. On the other hand, the answer to the first part of Question~\ref{quest-two} is ``no'' for the action of $\on{SL}_n$ on $2 \otimes \on{Sym}^2(n)$, where $n \geq 4$ is an arbitrary even integer; indeed, it was shown in work of Bhargava~\cite{thesource} that $N^{{\rm red}}(X)=o(N^{{\rm irr}}(X))$ for this representation. The situation is more complicated for the action of $\on{GL}_2 \times \on{SL}_3$ on $2 \otimes \on{Sym}_2(3)$. Bhargava proves in~\cite[Theorem~7]{MR2183288} that $N^{{\rm irr}}(X)\asymp X$, and it is well-known that
$N^{{\rm red}}(X)\asymp X\log X$. However, the reducible orbits of this representation can be partitioned into several natural subsets, and as was shown by Swaminathan in~\cite{swpreprint2} using the methods developed in the present paper, the answers to both parts of Question \ref{quest-two} are ``yes'' if we restrict to one of these subsets.

\medskip

Our next main results concern families defined by certain infinite sets of congruence conditions. We call a subset $S \subset W(\Z)$ a \emph{big family} if $S = W(\Z)^{(r)} \cap \bigcap_p S_p$, where the sets $S_p \subset W(\Z_p)$ satisfy the following properties:
\begin{enumerate}[leftmargin=2em]
    \item $S_p$ is $ \GG(\Z_p)$-invariant and is the preimage under reduction modulo $p^j$ of a nonempty subset of $W(\Z/p^j\Z)$ for some $j > 0$ for each $p$; and
    \item $S_p$ contains all ($\GG(\Z_p)$-orbits of) elements $B \in W_0(\Z_p)$ such that, for all $p \gg 1$, we have that $b_{i(n-i)}(B)$ is a $p$-adic unit for some $i$.
\end{enumerate}
Note that a big family $S$ is necessarily $\GG(\Z)$-invariant.

Our first result on reducible orbits in big families is stated in terms of a polynomial function $\lambda$ on the reducible hyperplane $W_0$, defined explicitly by
\begin{equation} \label{eq-defoflambda}
\lambda(B)\defeq\begin{cases} \displaystyle \prod_{i = 1}^{\lfloor \frac{n}{2}\rfloor} b_{i(n-i)}(B)^{2i-1},\hspace*{-4pt} & \text{if $2 \nmid n$,} \\ \displaystyle b_{\frac{n}{2}\frac{n}{2}}(B)^{\frac{n-2}{2}} \prod_{i = 1}^{\frac{n-2}{2}} b_{i(n-i)}(B)^{2i-1},\hspace*{-4pt} & \text{if $2 \mid n$.}
\end{cases}
\end{equation}
Given this definition of $\lambda$, we have the following asymptotic formula:
\begin{theorem} \label{thm-big}
Let $S$ be a big family. Then the number of reducible $\GG(\Z)$-orbits on $S$ of height up to $X$ is given by
$$C_{n,r}^{{\rm inf}} \cdot \Bigl(\prod_p\Big(1-\frac{1}{p}\Big)^{-\lfloor\frac{n}{2}\rfloor}\int_{B\in (S_p)_0}
\vert\lambda(B)\vert_pdB\Bigr)\cdot X^{\frac{n^2+n}{2}}+ o\big(X^{\frac{n^2+n}{2}}\big),
$$
where $dB$ denotes the Euclidean measure on $W_0(\Z_p)$, normalized so that $W_0(\Z_p)$ has volume $1$, and where $\vert-\vert_p$ denotes the usual $p$-adic absolute value.
 \end{theorem}
 Our second result on reducible orbits in big families, which is equivalent to Theorem~\ref{thm-big}, is stated in terms of a product of local orbit counts for the action of $\mc{P}$ on $W_0$:

\begin{theorem} \label{thm-acceptcubic}
Let $S$ be a big family. Then the number of reducible $\GG(\Z)$-orbits on $S$ of height up to $X$ is given by
\begin{equation} \label{eq-midacceptcubic}
 \Bigl(\prod_p \int_{f \in U(\Z_p) } \#\left(\frac{\on{inv}^{-1}(f)\cap (S_p)_0}{\mc{P}(\Z_p)}\right) df\Bigr)\cdot N^{(r)}(X) + o\big(X^{\frac{n^2+n}{2}}\big),
\end{equation}
 where $df$ denotes the Euclidean measure on $U(\Z_p)$, normalized so that $U(\Z_p)$ has volume $1$.
\end{theorem}

The local product in \eqref{eq-midacceptcubic} looks similar to many other mass formulas in related works with one major difference: the group in question, $\mc{P}$, is not reductive. To prove Theorem \ref{thm-acceptcubic}, it is therefore necessary for us to get some control over orbits of this non-reductive group, \mbox{which we do in, e.g., \S\ref{sec-sieve}.}

\subsection{Methods of proof} \label{sec-whatsnew}

To prove our main results, we introduce two new methods of determining asymptotics for reducible orbits, and we describe them both as follows.

\subsubsection*{Method I} Our first method proceeds by directly counting points in the cusp(s) of an ``averaged'' fundamental domain $\mc{D}$ for the action of $\GG(\Z)$ on $W(\R)$. This suffices to get the count of reducible orbits because, by the results in \cite{MR3156850} and \cite{MR3719247}, Properties (1) and (2) in Section~\ref{sec-genapp} are satisfied for the action of $\GG$ on $W$.
This method requires us to construct fundamental domains $\mc{D}$ for the action of $\GG(\Z)$ on $W(\R)$, which in turn requires us to construct a fundamental domain $\mc{F}$ \mbox{for the action of $\GG(\Z)$ on $\GG(\R)$.} Unlike in most previous situations, it is simply not enough for us to invoke the work of Borel and Harish-Chandra (see~\cite{MR0148666,MR147566}), who constructed fundamental domains for general semisimple groups. Indeed, our argument relies on $\mc{F}$ being \emph{box-shaped at infinity}, meaning that $\mc{F}$ looks like a Siegel domain in a neighborhood of the cusp. We prove the existence of such fundamental domains for our groups~$\GG$.

The region $\mc{D}$ is too skewed for a direct geometry-of-numbers argument to give anything better than an upper bound for the number of points it contains. To resolve this issue, we cut up the region $\mc{D}$ into a countable collection of nicer-looking slices. Within each slice, we prove that the count of the integral points is asymptotic to the volume of the slice. Summing up over all slices yields the desired total asymptotic. Our slicing method constitutes the first higher-dimensional generalization of an argument developed in~\cite{MR3090184}, which treated the simpler case of the cusps arising from the group $\on{GL}_2$.

Summing up the volumes of the slices gives us the desired asymptotic in terms of weighted volumes of certain sets in the reducible hyperplane $W_0(\R) \subset W(\R)$, where the volumes are computed with respect to the weight $\lambda$ defined in~\eqref{eq-defoflambda}.
We evaluate these weighted volumes by proving a Jacobian change-of-variables formula that transforms the measure $\lambda$ on $W_0$ into the product of the Euclidean measure on $U(\R)$ with the Haar measure on the lower-triangular subgroup $\mc{P}(\R) \subset \GG(\R)$. Such change-of-variable results have previously been proven when the group under consideration is unimodular (see, e.g.,~\cite[Section 3.4]{MR3272925}), but the fact that the group $\mc{P}$ fails to be unimodular presents significant new challenges.

Method I, as applied to the representation of $G$ on $W$, requires very complicated indexing and notation. To make Method I more readily comprehensible for the reader, we begin in \S\ref{sec-binaryquartics} by illustrating the method in the case of $\on{PGL}_2$ acting on $\on{Sym}_4(2)$, before proceeding with the parallel, but much more complicated, case of $G$ acting on $W$ starting in \S \ref{sec-cubicredux}. We note that \S\ref{sec-binaryquartics} can be read more or less independently of the rest of the paper.

\subsubsection*{Method II} Our second method proceeds by means of the following four steps. First, we claim that the asymptotics for the number of reducible $\GG(\Z)$-orbits on $W(\Z)$ are the same as the asymptotics for the number of $\mc{P}(\Z)$-orbits on $W_0(\Z)$. This claim is an immediate corollary \mbox{of the following two facts:}
\begin{enumerate}[leftmargin=20pt]
\item[(a)] If $B_1, B_2 \in W_0(\Z)$ are equivalent under $\GG(\Z)$ but not under $\mc{P}(\Z)$, then $B_1$ and $B_2$ have nontrivial stabilizer in $\GG(\Q)$, but as we establish in Proposition~\ref{prop-bs}, all but negligibly many $\GG(\Z)$-equivalence classes on $W_0(\Z)$ have trivial stabilizer in $\GG(\Q)$.
\item[(b)] As shown in \cite[Proposition~10.7]{MR3156850} and \cite[Proposition~23]{MR3719247}, all but negligibly many reducible $\GG(\Z)$-orbits on $W(\Z)$ have a representative lying on the reducible hyperplane $W_0(\Z)$.
\end{enumerate}

The second step is to develop the reduction theory for the action of $\mc{P}$ on $W_0$. Specifically, we prove that for any field $K$, the group $\mc{P}(K)$ acts \emph{simply transitively} on the set of elements in $W_0(K)$ having any fixed nondegenerate invariant polynomial. Using the fact that the group $\mc{P}$ has class number $1$ over $\Q$, we then deduce that the orbits of $\mc{P}(\Z)$ on $W_0(\Z)$ satisfy the following strong local-to-global principle:
\begin{theorem} \label{thm-stronglocglob}
Let $f \in U(\Z)$ be a nondegenerate monic degree-$n$ polynomial, and when $n$ is even, suppose that the $x^i$-coefficient of $f$ is divisible by $2$ for each odd $i$. For each prime $p$, choose $B_p \in W_0(\Z_p)$ such that $\on{inv}(B) = f$. Then there exists $B \in W_0(\Z)$, unique up to the action of $\mc{P}(\Z)$, such that $B$ is $\mc{P}(\Z_p)$-equivalent to $B_p$ for each prime $p$.
\end{theorem}


    In fact, we prove a more general version of Theorem \ref{thm-stronglocglob}, namely Theorem \ref{thm-genlocglob}, applying to representations of groups having class number $1$ over $\Q$.

    By Theorem~\ref{thm-stronglocglob}, the number of $\mc{P}(\Z)$-orbits on $W_0(\Z)$ with nondegenerate invariant polynomial $f$ is simply the product over all primes $p$ of the number of $\mc{P}(\Z_p)$-orbits on $W_0(\Z_p)$ lying above $f$. The third step is to sum this local product formula over all invariant polynomials of bounded height. A key ingredient for evaluating the sum is to verify that there are not too many orbits with large value of $\lambda$. This was verified in~\cite{sqfrval}, where an upper bound of roughly the correct order of magnitude was obtained for the number of $\mc{P}(\Z)$-orbits on $W_0(\Z)$ with large value of $\lambda$.\footnote{An even stronger upper bound for the number of orbits with large value of $\lambda$ can be obtained using Method I, by simply summing over those slices with large value of $\lambda$; \mbox{see Theorem~\ref{thm-bsw} (to follow).}} We thus arrive at the surprising conclusion that, at least for the representation of $\GG$ on $W$, an upper bound on the number of reducible orbits can be indirectly used to deduce a precise asymptotic!

The resulting asymptotic for the total count of reducible orbits is now expressed in terms of a product of local orbit counts, as in Theorem~\ref{thm-acceptcubic}. The fourth and final step is to evaluate this product. We do this by using the previously mentioned Jacobian change-of-variables formula in reverse! This expresses the local orbit count at a prime $p$ as a certain $p$-adic integral, thus yielding Theorem~\ref{thm-big}; evaluating each of these integrals and multiplying them all together yields the total asymptotic in Theorem~\ref{thm-maincount2}.

\begin{remark}
For the coregular representations considered in this paper, both of the above methods are sufficient to obtain asymptotics for the number of reducible elements. Method I gives a direct proof of Theorem \ref{thm-big}, while Method II gives a direct proof of Theorem \ref{thm-acceptcubic}. Moreover, the main terms in these two theorems can be related to each other using a Jacobian change-of-variables. However, for certain representations such as those considered in \cite{swpreprint2} and \cite{ollerpreprint}, Method II can be directly applied while applying Method I seems more complicated. Moreover, for certain representations such as those considered in \cite{D4preprint}, Method II is inapplicable, while Method I can be used.
\end{remark}

\subsection{Other applications of our methods} \label{sec-open}

We expect that both of the methods that we introduce in this paper can be used to count reducible orbits for other representations.

Our methods have already been used to derive arithmetic applications beyond the results of this paper. Swaminathan~\cite{swpreprint2} counts reducible orbits for the action of $\on{SL}_n$ on $2 \otimes \on{Sym}^2(n)$, where $n$ is odd, and uses this to prove asymptotics for counts of $2$-torsion in the ideal class groups of cubic orders in the case $n = 3$. (A similar application could be pursued for counting quintic rings, which also correspond to integral orbits of a coregular representation via a parametrization of Bhargava~\cite{MR2373152}.) Also, Oller \cite{ollerpreprint} counts reducible orbits in all the Vinberg representations in Thorne's thesis~\cite{MR3054927}, and also carries out squarefree sieves in all these cases.

One possible line of inquiry is to study the action of $\GL_2$ on the space of binary $n$-ic forms for $n\geq 5$. We carry out the case $n=4$ in \S\ref{sec-binaryquartics}; the $n=3$ case is similar and would give a simpler proof of the results of Shintani and Bhargava--Varma mentioned previously. The spaces are not coregular for $n\geq 5$, but when integral orbits are ordered by Julia invariant, Bhargava and Yang~\cite{BYpreprint} determined asymptotics for the number of irreducible orbits. We expect the methods introduced in this paper to have applications towards counting the reducible orbits, ordered by Julia invariant, for these spaces. Other arithmetically interesting families of coregular representations for which the question of counting reducible orbits remains open may be found in the thesis of Ho~\cite{MR2713823}.


\section{Counting reducible integral orbits on binary quartic forms} \label{sec-binaryquartics}

In this section, we determine asymptotics for the number of reducible orbits of bounded height for the action of $\PGL_2$ on the space of integral binary quartic forms. Our purpose is to illustrate Method I (see \S\ref{sec-whatsnew}) in the context of a low-dimensional example, with the view of making the higher-dimensional application treated in \S\S\ref{sec-cubicredux}--\ref{sec-locglob} more readily comprehensible. 

This section can be read more or less independently of the rest of the paper. We remark that some of the notation used within this section is recycled in subsequent sections to denote different but analogous objects. As none of the objects introduced in this section are used in subsequent sections, we do not expect this to cause ambiguity.

\medskip

\noindent\emph{Setup}. Let $V$ denote the affine $\Z$-scheme whose $R$-points consist of binary quartic forms with coefficients in $R$; i.e., we have $$V(R)\defeq \bigl\{f(x,y)=ax^4+bx^3y+cx^2y^2+dxy^3+ey^4:a,b,c,d,e\in R\bigr\}.$$
The group $\PGL_2$ acts on $V$ via
$(g\cdot f)(x,y) \defeq (\det g)^{-2} \times f((x,y)\cdot g)$. The ring of invariants for the action of $\PGL_2$ on $V$ is freely generated by two invariants, denoted $I$ and $J$. For the form $f(x,y)$ written as above, these invariants are given explicitly by
\begin{equation*}
I(f)=12ae-3bd+c^2;
\quad
J(f)=72ace+9bcd-27ad^2-27eb^2-2c^3.
\end{equation*}
For convenience, we define $\on{inv} \colon V \to \mathbb{A}^2$ to be the map that sends $f \mapsto (I(f), J(f))$. The image of this map over $\Z$ is not all of $\mathbb{A}^2(\Z)$, but is defined by congruence conditions modulo $27$; see~\cite[Theorem~1.7]{MR3272925}.

We define a $\PGL_2(\R)$-invariant \emph{height} $H$ on $V(\R)$ via $H(f)\defeq\max\{|I(f)|^3,J(f)^2/4\}$. We say that a binary quartic form $f$ is \emph{nondegenerate} if its \emph{discriminant} $\Delta(f) \defeq (4I(f)^3 - J(f)^2)/27$ is nonzero. More generally, we say that a pair $(I,J)$ is nondegenerate if the quantity $\Delta(I,J) \defeq (4I^3-J^2)/27$ is nonzero (so a binary quartic form is nondegenerate if and only if its invariants are nondegenerate). For $i\in\{0,1,2\}$ and $R = \Z$ or $\R$, we let $V(R)^{(i)}$ denote the set of nondegenerate elements in $V(R)$ having $i$ pairs of complex conjugate roots and $4-2i$ real roots in $\P^1(\C)$. 

A nondegenerate binary quartic form $f \in V(R)$ is said to be \emph{reducible} if it factors over $R$. 
It follows from \cite[Lemma 2.3]{MR3272925} that the number of reducible orbits on $V(\Z)$ having height bounded by $X$ and factoring into the product of two irreducible quadratic forms, is $O(X^{2/3+\epsilon})$, and thus negligible. That is, $100\%$ of reducible orbits have at least one rational (and thus at least one real) linear factor. Therefore, for the purposes of this section, it suffices to restrict our attention to counting orbits of binary quartic forms that possess a rational linear factor, which amounts to counting reducible integral orbits in the sets $V(\R)^{(0)}$ and $V(\R)^{(1)}$. 

Let $V_0(R)\subset V(R)$ be the ``reducible hyperplace'' consisting of those forms $f$ with $a=0$ (i.e., those forms $f$ that are divisible by $y$). The reducible hyperplane is sent to itself under the action of the lower-triangular subgroup $P \subset \PGL_2$,\footnote{More precisely, for a ring $R$, we define $P(R)$ to be the image of the subgroup of lower triangular matrices in $\GL_2(R)$ under the map $\GL_2(R)\to\PGL_2(R)$.} and so we obtain a well-defined representation of $P$ on $V_0$. In what follows, for a $\Z$-algebra $R$ and any subset $S \subset V(R)$, we sometimes write $S_0 \subset S$ to denote the subset $S \cap V_0(R)$.

\medskip

\noindent\emph{Main results}. Given the above setup, the main result of this section is as follows, by analogy with Theorem~\ref{thm-maincount2}:

\begin{theorem} \label{thm-quarticmaincount}
The number of reducible $\PGL_2(\Z)$-orbits on $V(\Z)^{(i)}$ having height less than $X$ is
\begin{equation*}
\zeta(2)\cdot \big(\tfrac{8+24i}{135} \cdot X^{\frac{5}{6}}\big)+O_\epsilon\big(X^{\frac{3}{4}+\epsilon}\big).
\end{equation*}
\end{theorem}
\begin{remark}
    The factor $\frac{8+24i}{135} \cdot X^{\frac{5}{6}}$ occurring in Theorem~\ref{thm-quarticmaincount} is an asymptotic for the number of pairs $(I,J)$ that arise as invariants of orbits of height up to $X$ (see~\cite[Proposition~2.10]{MR3272925}). The factor of $\frac{8+24i}{135}$ comprises two parts: the first is a factor of $\frac{8+24i}{5}$, which is the volume of the space of invariants of height at most $1$ in $V(\R)^{(i)}$, and the second is a factor of $\frac{1}{27}$, which occurs because not every pair $(I,J) \in R^2$ arises as the set of invariants of a binary quartic form in $V(R)$.
\end{remark}

Next, we consider subsets of $V(\Z)$ cut out by certain (possibly) infinite sets of congruence conditions. We call $S\subset V(\Z)$ a {\it big family} if $S=V(\Z)^{(i)}\cap\bigcap_p S_p$, for $i\in\{0,1\}$, where the sets $S_p\subset V(\Z_p)$ satisfy the following properties:
\begin{enumerate}[leftmargin=2em]
    \item $S_p$ is $ \PGL_2(\Z_p)$-invariant and is the preimage under reduction modulo $p^j$ of a nonempty subset of $V(\Z/p^j\Z)$ for some $j > 0$ for each $p$; and
    \item For all $p\gg 1$, the set $S_p$ contains (all $\PGL_2(\Z_p)$-orbits of) all elements $f(x,y) \in V(\Z_p)$ such that $a(f)=0$ and $b(f)$ is a $p$-adic unit, where $a(f)$ and $b(f)$ denote the $x^4$- and $x^3y$-coefficients of $f(x,y)$, respectively.
\end{enumerate}
We then have the following result, by analogy with Theorem~\ref{thm-big}:
\begin{theorem} \label{thm-bigquartic}
Let $S\subset V(\Z)^{(i)}$ be a big family. Then the number of reducible $\PGL_2(\Z)$-orbits on $S$ of height less than $X$ is
\begin{equation*} 
\big(\tfrac{8+24i}{135} \cdot X^{\frac{5}{6}}\big) \cdot \prod_p (1-p^{-1})^{-1}\int_{f \in (S_p)_0} |b(f)|_p df + O_\epsilon(X^{\frac{3}{4}+\epsilon}).
\end{equation*}
\end{theorem}
Finally, by analogy with Theorem~\ref{thm-acceptcubic}, we have the following result, which is equivalent to Theorem \ref{thm-bigquartic}:
\begin{theorem}\label{thm-needplocalfactsl21}
Let $S\subset V(\Z)^{(i)}$ be a big family. Then the number of reducible $\PGL_2(\Z)$-orbits on $S$ of height less than $X$ is
\begin{equation*}
 \bigg(\prod_p \int_{(I,J) \in \Z_p^2} \#\bigg(\frac{\on{inv}^{-1}(I,J) \cap (S_p)_0}{P(\Z_p)}\bigg) dI dJ \bigg) \cdot (\tfrac{8+24i}{5} \cdot X^{\frac{5}{6}}) +O_\epsilon\big(X^{\frac{3}{4}+\epsilon}\big).
\end{equation*}
\end{theorem}

\subsection{Reduction theory for the action of $\PGL_2(\Z)$ on $V(\R)$} \label{sec-sl2redux}

To count orbits of $\on{PGL}_2(\Z)$ on $V(\Z)$, we realize these orbits as lattice points in fundamental sets for the action of \mbox{$\on{PGL}_2(\Z)$ on $V(\R)$.} In this section, we construct such fundamental sets by means of a two-step process: first, in \S\ref{sec-pgl}, we describe a fundamental domain $\mc{F}$ for the action of $\on{PGL}_2(\Z)$ on $\on{PGL}_2(\R)$; subsequently, in \S\ref{sec-set}, we combine $\mc{F}$ with fundamental sets for the action of $\on{PGL}_2(\R)$ on $V(\R)$.

\subsubsection{A box-shaped fundamental domain for $\on{PGL}_2(\Z) \curvearrowright \on{PGL}_2(\R)$} \label{sec-pgl}

We start by recalling Gauss' fundamental domain for the action of $\on{PGL}_2(\Z)$ on $\on{PGL}_2(\R)$, rephrased in terms of the Iwasawa decomposition of $\on{PGL}_2(\R)$, which we now recall. 
 Let $N$ be the (algebraic) subgroup of $\on{PGL}_2$ consisting of lower triangular unipotent matrices $\left[\begin{smallmatrix} 1 & 0 \\ u & 1 \end{smallmatrix}\right]$, let $T$ be the maximal torus defined by \mbox{$T = \big\{ \left[\begin{smallmatrix} t^{-1} & 0 \\ 0 & t \end{smallmatrix}\right]: t \in \R_{> 0}\big\}$,} and let $K=\mathrm{SO}_2(\R)/\{\pm \on{id}\}$. We often abuse notation by writing $u$ and $t$ for the corresponding elements of $N(\R)$ and $A$. If we let $T' = \big\{t \in T : t \geq \frac{\sqrt[4]{3}}{\sqrt{2}}\big\}$, then for each $t \in T'$, there exists a compact subset $N'(t) \subset [-\frac{1}{2}, \frac{1}{2}]$ such that the set
 \begin{equation} \label{eq-thissethastwo}
 \{ut : u \in N'(t), t \in T'\} \cdot K
 \end{equation}
 is a fundamental domain $\mc{F}$ for the action of $\on{PGL}_2(\Z)$ on $\on{PGL}_2(\R)$. It is well-known that $N'(t) = [-\frac{1}{2},\frac{1}{2}]$ for all $t\geq 1$; consequently, we say that this fundamental domain is \emph{box-shaped at infinity}. This property of the fundamental domain is essential for our proof.

 We denote elements of $K$ by $\theta$. With respect to the coordinates $u$ on $N(\R)$, $t$ on $T$, and $\theta$ on $K$, the Haar measure $dg$ on $\on{PGL}_2(\R)$ is given by
 \begin{equation} \label{eq-haarpgl2}
     dg = d\theta du(t^{-2}d^\times t),
 \end{equation}
 where $d^\times t = dt/t$. Above, $d\theta$ is normalized so that $\int_{\theta \in K}d\theta = 1$, and $du$ is normalized so that $N(\Z)$ has covolume $1$ in $N(\R)$.

\subsubsection{Fundamental sets for $\on{PGL}_2(\R) \curvearrowright V(\R)$ and $\on{PGL}_2(\Z) \curvearrowright V(\R)$} \label{sec-set}

The action of $\on{PGL}_2(\R)$ on $V(\R)^{(0)}\sqcup V(\R)^{(1)}$ has one orbit per set of {\it nondegenerate invariants} --- i.e., a pair of invariants $(I,J)\in\R^2$ such that $\Delta(I,J)\defeq4I^3-J^2\neq 0$. This orbit belongs to $V(\R)^{(0)}$ when $\Delta(I,J)>0$ and to $V(\R)^{(1)}$ when $\Delta(I,J)<0$.
Consider the function $\sigma_0$ given by
\begin{equation}\label{eq:sigmabqf}
\begin{array}{rcl}
\sigma_0\colon \R^2\smallsetminus\{\Delta=0\}&\to& V(\R)
\\[.1in]
(I,J)&\mapsto&x^3y-\frac{I}3xy^3-\frac{J}{27}y^4.
\end{array}
\end{equation}
It is easy to check that $\sigma_0$ is a section of the map $\on{inv}$, meaning that the invariants of $\sigma_0(I,J)$ are $I$ and $J$. 
For $i\in\{0,1\}$, we take our fundamental sets for the action of $\PGL_2(\R)$ on $V(\R)$ to be
\begin{equation*}
\cR^{(i)}\defeq\R_{>0}\cdot \bigl\{\sigma_0(I,J):
(-1)^i\Delta(I,J)>0,\,H(I,J)=1
\bigr\}.
\end{equation*}
Finally, we note that the stabilizer $\on{Stab}_{\on{PGL}_2(\R)}(F)$ is independent of the choice of $f \in V(\R)^{(i)}$; letting $n_i \defeq \#\on{Stab}_{\on{PGL}_2(\R)}(f)$ for any $f \in V(\R)^{(i)}$, one readily verifies that $n_0 = 4$ and $n_{1} = 2$.

We conclude that the multiset $\mc{F} \cdot \mc{R}^{(i)}$ is a cover for a fundamental domain for the action of $\on{PGL}_2(\Z)$ on $V(\R)^{(i)}$. More precisely, every $\PGL_2(\Z)$-orbit of $f \in V(\R)^{(i)}$ is represented exactly $n_i/\#\on{Stab}_{\PGL_2(\Z)}(f)$ times in $\mc{F} \cdot \mc{R}^{(i)}$.

\subsection{The action of the subgroup $P$ on the reducible hyperplane $V_0$}

In this section, we examine the action of the lower-triangular (parabolic) subgroup $P$ on the reducible hyperplane $V_0$. Specifically, in \S\ref{sec-pv04}, we show that over many interesting base rings $R$, the action of $P(R)$ on the set of forms in $V_0(R)$ lying over a given nondegenerate pair of invariants is simply transitive. Then, in \S\ref{sec-jack4}, we prove a Jacobian change-of-variables formula relating the Euclidean measure on $V_0(R)$ to the product of the Haar measure on $P(R)$ with the Euclidean measure on $R^2$. This formula will be applied in \S\S\ref{sec-compconst4}--\ref{sec-sl2conditions}.

\subsubsection{Reduction theory for the action of $P$ on $V_0$} \label{sec-pv04}

Let $R$ be a field or $\Z_p$ for some prime $p$. Then we have the following result, which classifies the orbits and stabilizers of $P(R)$ on $V_0(R)$:

\begin{prop} \label{prop-cantranslate4}
Let $R$ be as above, and let $(I,J)\in R^2$ be such that $\Delta(I,J)$ is a unit. Then the set of binary quartic forms in $V_0(R)$ with invariants $I$ and $J$ is either empty or consists of a single $P(R)$-orbit, and the stabilizer of any element in this orbit is trivial.
\end{prop}
\begin{proof}
Given a form $f(x,y)=bx^3y+cx^2y^2+dxy^3+ey^4\in V_0(R)$ having invariants $I$ and $J$, we first note that $b^2\mid\Delta(f)=\Delta(I,J)$. Thus, if $\Delta(I,J)$ is a unit, then so is $b$. As a consequence, by replacing $f(x,y)$ with a $P(R)$-translate, we can arrange that $b=1$. When $R\neq\Z_3$, we have that $3 \in R^\times$, and hence by replacing $f$ with the $P(R)$-translate $f(x-c/3y,y)$, we may assume that $c=0$. When $R=\Z_3$, we may similarly replace $f$ with a $P(R)$-translate to arrange that $c\in\{0,1,2\}$ (depending on the residue classes of $I$ modulo $9$ and $J$ modulo $27$). Once this has been done, the values of $d$ and $e$ are respectively determined by $I$ and $J$ (since $a=0$ implies linear relations between $(I,J)$ and $(d,e)$). This proves that the set of elements in $V_0(R)$ with invariants $I$ and $J$ (if nonempty) form a single $P(R)$-orbit.

Next, we prove the claim regarding the stabilizer in $P(R)$ of $f\in V_0(R)$. Suppose that an element $g\in P(R)$, represented by a matrix with coefficients $1$ and $u$ on the diagonal and $n$ in the lower left coordinate, fixes $f$. First note that since $b$ is a unit, the $x^2y^2$-coefficient of $g\cdot f(x,y)$ will change unless $n=0$. Assume thus that $n=0$. Next note that the $x^3y$-coefficient of $g\cdot f(x,y)$ is $u^{-1}b$, implying that $u=1$, as needed.
\end{proof}

The above result has the following immediate consequence by specializing to the case $R = \R$.

\begin{lemma} \label{lem-cantranslate44}
    Let $(I,J) \in \R^2$ be nondegenerate. Then the set $\{f \in \on{inv}^{-1}(f)_0 : b(f) > 0\}$ consists of a single $N(\R)T$-orbit.
\end{lemma}

\subsubsection{A Jacobian change-of-variables formula} \label{sec-jack4}

Proposition~\ref{prop-cantranslate4} implies that when $R = \R$ or $\Z_p$ for a prime $p$, the space $V_0(R)$ is a fibration over $R^2$, where the generic fiber can be idenitfied with $P(R)$, so long as it is nonempty. Thus, the Euclidean measure on $V_0(R)$ should be related to the product of the Haar measure on $P(R)$ with the Euclidean measure on $R^2$. The following proposition gives a Jacobian change-of-variables formula relating these measures:
\begin{prop} \label{prop-jac4}
Let $R=\R$ or $\Z_p$ for some prime $p$, and let $\phi\colon V_0(R)\to\R$ be a measurable function. Then we have
\begin{equation*}
\int_{f\in V_0(R)}\phi(f){\vert}b(f){\vert}df=
\frac{2}{27}\int_{\substack{(I,J)\in R^2\\ \Delta(I,J)\neq 0}}
\Bigl(\sum_{f\in\frac{\inv^{-1}(I,J)_0}{P(R)}}\int_{h\in P(R)} \phi(h \cdot f)dh\Bigr)dIdJ.
\end{equation*}
where $df$, $dI$, and $dJ$ are Euclidean measures, where $dh$ is the right Haar measure on $P$ given by $dh = tdudt$, and where $|-|$ denotes the usual absolute value on $R$.
\end{prop}
\begin{proof}
First note that the result for $R=\R$ implies the result for $R=\Z_p$ by the principle of permanence of identities, so it suffices to treat the case $R = \R$.  
Recall the construction of the section $\sigma_0$ in \eqref{eq:sigmabqf}. Proposition~\ref{prop-cantranslate4} implies that we have 
\begin{equation*}
V_0(\R)=P(\R)\cdot\sigma_0(\{(I,J)\in\R^2:\Delta(I,J)\neq 0\}).
\end{equation*}
Thus, the theorem follows from the equality
\begin{equation}\label{eq-jac1bqf}
\int_{f \in P(\R) \cdot \sigma_0(\R^2\smallsetminus\{\Delta=0\})} \phi(f){\vert}b(f){\vert} df  = \frac{2}{27}\int_{(I,J) \in \R^2} \int_{h\in P(\R)} \phi(h \cdot \sigma_0(I,J)) dhdIdJ,
\end{equation}
which in turn is a consequence of the following Jacobian change-of-variables computation. First note that a typical element of the region of integration on the left-hand side of~\eqref{eq-jac1bqf} is given by
\begin{equation*}
(tu) \cdot \sigma_0(I,J)=
\frac{1}{t^2}x^3y+3u x^2y^2+\Bigl(3u^2t^2-\frac{It^2}{3}\Bigr)xy^3+\Bigl(
u^3t^4-\frac{Iut^4}{3}-\frac{Jt^4}{27}
\Bigr)y^4.
\end{equation*}
By taking partial derivatives with respect to $t$, $u$, $I$, and $J$ of the coefficients of the binary quartic form on the right-hand side above and arranging these partials into matrix form, we find that the Jacobian determinant relating the measures $df$ and $dtdudIdJ$ is given as follows:
\begin{equation*}
\displaystyle 
\begin{vmatrix}
-2t^{-3}& 0 &
6u^2t-2It/3&
4u^3t^3-4Iut^3/3-4Jt^3/27 \\
&3&6ut^2 &3u^2t^4-It^4/3\\
&&-t^2/3&-ut^4/3\\
&&&-t^4/27
\end{vmatrix}
=\displaystyle
-\frac{2t^3}{27}.
\end{equation*}
Therefore, we have
\begin{equation*}
    df=\frac{2}{27}t^3dtdudIdJ.
\end{equation*}
Equation \eqref{eq-jac1bqf} then follows, since $b((tu) \cdot \sigma_0(I,J))=t^{-2}$.
\end{proof}

\subsection{Counting reducible $\on{PGL}_2(\Z)$-orbits on $V(\Z)$} \label{sec-countpgl}

Let $i \in \{0,1\}$. In this section, we obtain asymptotics for the number of reducible orbits of $\on{PGL}_2(\Z)$ on $V_0(\Z)^{(i)}$ of bounded height, thus proving Theorems~\ref{thm-quarticmaincount}--\ref{thm-needplocalfactsl21}. To simplify the exposition in the rest of this section, we introduce the following notation:
\begin{itemize}[leftmargin=13pt]
\item For any set $S \subset V(\Z)$, let $S_{\on{red}} \subset S$ be the subset of forms in $S$ having a rational linear factor; for $X > 0$, let $S_X \defeq \{B \in S : H(B) < X\}$; and as before, let $S_0 \defeq S\cap V_0(\Z)$ be the set of elements of $S$ that lie on the reducible hyperplane.
\item Let $G_0 \subset \PGL_2(\bR)$ be a fixed nonempty open bounded set such that  $G_0^{-1}=G_0$ and $G_0$ is left- and right-invariant under the group $K'$ generated by $K$ together with the diagonal matrix having diagonal entries $1$ and $-1$. Such a set can be constructed by starting with a nonempty open bounded set $G_0'$ and taking $G_0=K'(G_0'\cup G_0'^{-1}) K'$.
\item Define the multiset $\cB_\infty$ by
\begin{equation*}
\cB_\infty \defeq G_0\cdot \fundset^{(i)}\cap V_0(\R).
\end{equation*}
Set $\cB \defeq (\cB_\infty)_1$, and 
note that by the construction of $\fundset$, we have \mbox{$(\cB_\infty)_X=X^{\frac{1}{6}}\cB$.}
\item We define the quantity $C(\cB)$ by
\begin{equation}\label{eq:CBbqf}
    C(\cB)\defeq\frac{1}{\wt{n}_i\on{Vol}(G_0)} \cdot
\int_{f\in \cB}{\vert}b(f){\vert}df,
\end{equation}
where the volume of $G_0$ is computed using the Haar measure $dg$, and where $\wt{n}_i = 2n_i$, with $n_i$ as defined in \S\ref{sec-set}.
\end{itemize}
\begin{itemize}[
leftmargin=13pt]
\item For a finite set $\Sigma$ of $\PGL_2(\Z)$-orbits on $V(\Z)$, let $\#'\Sigma$ be the number of elements of $\Sigma$, where each $f \in \Sigma$ is counted with weight $1/\#\on{Stab}_{\PGL_2(\Z)}(f)$.
\end{itemize}

\subsubsection{Averaging over fundamental domains} \label{sec-sl2avg}

We begin by applying Bhargava's averaging technique, developed in \cite{MR2183288,MR2745272}. 
By an argument identical to \cite[Theorem 2.5]{MR3272925}, which involves averaging over translates of the fundamental domain $\mc{F}$ by elements of $G_0$ and performing a suitable change-of-variables, \mbox{we have that}
\begin{equation} \label{eq-sl2pmmult}
\begin{array}{rcl}
\displaystyle\#'\left(\frac{(V^{(i)}(\Z)_{\on{red}})_X}{\on{PGL}_2(\Z)}\right) &=& 
\displaystyle\frac{1}{n_i} \cdot \#\big(\cF h \cdot \cR^{(i)} \cap (V(\Z)_{\on{red}})_X\big)
\\[.2in]&=&\displaystyle
\frac{1}{n_i\Vol(G_0)}\int_{g\in\FF}\#\big(
gG_0 \cdot \cR_X^{(i)}\cap V(\Z)_\red\big)dg,
\end{array}
\end{equation}
Then we have the following result:
\begin{prop} \label{prop-afterave4}
We have
\begin{equation} \label{eq-vredx}
\displaystyle\#\left(\frac{(V^{(i)}(\Z)_{\on{red}})_X}{\on{PGL}_2(\Z)}\right)
=\frac{1}{n_i\Vol(G_0)}\int_{t=1}^\infty\int_{u=-\frac{1}{2}}^{\frac12}
\#\big(ut X^{\frac{1}{6}}\mc{B}\cap V_0(\Z)\big)t^{-2}dud^\times t+O_\epsilon(X^{\frac{3}{4}+\epsilon}),
\end{equation}
where $\cB$ is the multiset $\cB \defeq (\cB_\infty)_1 \defeq G_0 \cdot \cR_1^{(i)} \cap V_0(\R)$.
\end{prop}
\begin{proof}
The number of reducible (and irreducible) $\on{PGL}_2(\Z)$-orbits on $V(\Z)$ that have nontrivial integral stabilizer was proven in \cite[Lemma~2.4]{MR3272925} to be bounded $O(X^{\frac34 +\epsilon})$. 
Thus, we may replace the $\#'$ in \eqref{eq-sl2pmmult} with $\#$ at the cost of an error of $O(X^{\frac34 +\epsilon})$. We split the fundamental domain $\FF$ as $\FF = \FF'\sqcup\{utk:u\in[-\frac12, \frac12],\,t\geq 1,\,k\in K\}$, where $\FF'$ is absolutely bounded.  When the integral over $\FF$ is restricted to the compact region $\FF'$, it is clear that the number of forms with vanishing $x^4$-coefficient is bounded by $O(X^{\frac23})$.
By \cite[Lemma 2.3]{MR3272925}, the number of reducible forms with nonzero $x^4$-coefficient is bounded by $O_\epsilon(X^{\frac23 +\epsilon})$. These error terms are sufficiently small.
The proposition then follows upon noting that, by the left $K$-invariance of $G_0$, the set $gG_0\cR^{(i)}$ is independent of $\theta$ when $g$ is written as $g=ut\theta$ in Iwasawa coordinates. 
\end{proof}

\subsubsection{Slicing} \label{sec-sl2pmslicing}
Throughout this subsection we set $Y\defeq X^{\frac16}$. The integrand of the right-hand side of~\eqref{eq-vredx} is the number of integral points in the region $utY\cB\cap V_0(\R)$. This region is typically quite skewed: indeed, whenever $t$ is high up in the cusp, the $x^3y$-coefficient is small, so the volumes of the projections of the set $utY\cB$ away from this coefficient has the same order of magnitude as the volume of $utY\cB$ itself. Furthermore, the region where $t$ is high up in the cusp contributes most of the lattice points that we are interested in counting! We resolve this issue in this section by fibering the region $utY\cB \cap V_0(\R)$ by the $x^3y$-coefficient and using
a result of Davenport to estimate the number of lattice points on each fiber.

We now partition the region $utY\mc{B}$ into slices, one for each possible value of the $x^3y$-coefficient. For any $b \in \R \smallsetminus \{0\}$, and any $\mc{S} \subset V(\R)$, let $\mc{S}|_b$ denote the \emph{slice of $\mc{S}$ at $b$}, i.e., the subset of forms in $\mc{S}$ with $x^4$-coefficient equal to $0$ and $x^3y$-coefficient equal to $b$. Then we can express the integrand of the right-hand side of~\eqref{eq-vredx} as follows:
\begin{equation} \label{eq-easyslice}
\#\big(utY\mc{B} \cap V_0(\Z)\big) = \sum_{\substack{b \in \Z \\ b \neq 0}} \#\big((utY\mc{B})|_b \cap V(\Z)\big)
=\sum_{\substack{b \in \Z \\ b \neq 0}}\Vol\big((utY\mc{B})|_b\big)(1+O(Y^{-1})),
\end{equation}
where the final estimate is a consequence of the following proposition, due to Davenport:

\begin{prop}[\protect{\cite{MR43821}}] \label{prop-davenport}
Let $\cR$ be a bounded, semi-algebraic multiset in $\bR^n$ having maximum multiplicity $m$ that is defined by at most $k$ polynomial inequalities, each having degree at most $\ell$. Let $\cR'$ denote the image of $\cR$ under any $($upper or lower$)$ triangular, unipotent transformation of $\bR^n$. Then the number of integer lattice points $($counted with multiplicity$)$ contained in the region $\cR'$ is given by
$$\on{Vol}(\cR) + O (\max\{\on{Vol}(\ol{\cR}),1\}),$$
where $\on{Vol}(\ol{\cR})$ denotes the greatest $d$-dimensional volume of any projection of $\cR$ onto a coordinate subspace obtained by equating $n-d$ coordinates to zero, where $d$ ranges over all values in $\{1, \dots, n-1\}$. The implied constant in the second summand depends only on $n$, $m$, $k$, and $\ell$.
\end{prop}

Now, since unipotent transformations preserve both the value of $b$ and the volume, we have $\on{Vol}\big((utY\cB)|_b\big) = \on{Vol}\big((tY\cB)|_b\big)$. Combining~\eqref{eq-easyslice} with Proposition~\ref{prop-afterave4} yields the following:
\begin{equation}\label{eq-halfwaysl2}
\begin{array}{rcl}
\displaystyle\#\left(\frac{(V^{(i)}(\Z)_{\on{red}})_X}{\on{PGL}_2(\Z)}\right)  &=&  
\displaystyle\frac{1}{n_i\Vol(G_0)}\sum_{\substack{b \in \Z \smallsetminus \{0\}}}\int_{1\leq t\ll Y^{\frac12}/|b|^{\frac12}}\Vol\big((tY\mc{B})|_b\big)(1+O(Y^{-1}))t^{-2}dud^\times t
\\[.3in]
&&\displaystyle
\qquad\qquad\qquad\qquad\qquad\qquad\qquad\qquad\qquad\qquad\qquad\qquad +O_\epsilon(X^{\frac{3}{4}+\epsilon}).
\end{array}
\end{equation}
The bound in the region of integration above is obtained by noting that the value of $|b|$ is bounded above by $O(Y)$, and for each fixed $b\neq 0$, the range of $t$ goes up to $Y^{\frac12}/\sqrt{|b|}$, since the $x^3y$-coefficients of elements in $utY\mc{B}$ are $\ll t^{-2}Y$.
We bound the error term in the right-hand side of \eqref{eq-halfwaysl2} to be
\begin{equation} \label{eq-ridoferror}
    \ll Y^{-1}\on{Vol}((tY\cB)|_b) \ll \sum_{b = 1}^{O(Y)} \int_{t = 1}^{O(Y^{\frac{1}{2}}  /\sqrt{b})} Y^2t^6  \frac{d^\times t}{t^2} \ll Y^{4}\sum_{b = 1}^{O(Y)} b^{-2} \ll X^{\frac23}.
\end{equation}
Substituting the estimate~\eqref{eq-ridoferror} into~\eqref{eq-halfwaysl2} yields
\begin{equation} \label{eq-toreferlater}
\#\left(\frac{(V^{(i)}(\Z)_{\on{red}})_X}{\on{PGL}_2(\Z)}\right)
= \frac{1}{n_i\Vol(G_0)}\sum_{\substack{b \in \Z \smallsetminus \{0\}}}\int_{1\leq t\ll Y^{\frac12}/|b|^{\frac12}} \Vol\big((tY\mc{B})|_b\big)\frac{d^\times t}{t^2}
 + O_\epsilon(X^{\frac34+\epsilon}).
\end{equation}

We now manipulate the integrand in~\eqref{eq-toreferlater} to extract its dependence on the slicing index $b$. Because unipotent transformations leave volumes unchanged, and the action of $t$ on $V(\R)$ scales the $x^iy^j$-coefficient by $t^{i-j}$, we have
\begin{equation} \label{eq-unisl2}
\Vol\big((tY\mc{B})|_b\big) = Y^3\Vol\big((t\mc{B})|_{b/Y}\big) = 
t^6X^{1/2}\on{Vol}\big(\mc{B}|_{t^{2}b/Y}\big).
\end{equation}
Since $G_0$ is left-$K'$-invariant and since the diagonal matrix with entries $1$ and $-1$ belongs to $K'$, it follows that $\on{Vol}(\mc{B}|_\beta) =  \on{Vol}(\mc{B}|_{-\beta})$ for any $\beta \in \R \smallsetminus \{0\}$.
Hence, setting $\beta=t^2b/Y$ (which gives $d^\times\beta=2d^\times t$), we obtain
\begin{equation} \label{eq-firstimebeta}
\begin{array}{rcl}
\displaystyle\#\left(\frac{(V^{(i)}(\Z)_{\on{red}})_X}{\on{PGL}_2(\Z)}\right) &=& 
\displaystyle\frac{X^{1/2}}{2n_i\Vol(G_0)}  \sum_{\substack{b \in \Z \smallsetminus \{0\}}}b^{-2} \int_{|b|/Y\leq\beta\ll 1} Y^2\beta^2  \on{Vol}\big(\mc{B}|_\beta\big) d^\times\beta + O_\epsilon(X^{\frac34+\epsilon})
\\[.25in]
&=&\displaystyle\frac{X^{5/6}}{n_i\Vol(G_0)}  \sum_{b = 1}^{\infty}b^{-2} \int_{\beta \geq 0} \beta  \on{Vol}\big(\mc{B}|_\beta\big) d\beta + O_\epsilon(X^{\frac34+\epsilon}),
\end{array}
\end{equation}
where the second line above follows since $\mc{B}|_\beta$ is a bounded set and $\on{Vol}\big(\mc{B}|_\beta\big) \ll 1$. It therefore follows that
\begin{equation} \label{eq-nomorebeta}
\#\left(\frac{(V^{(i)}(\Z)_{\on{red}})_X}{\on{PGL}_2(\Z)}\right)  = \zeta(2) \cdot C(\cB)\cdot X^{5/6} + O_\epsilon(X^{\frac34 +\epsilon}),
\end{equation}
where $C(\cB)$ was defined in \eqref{eq:CBbqf}.
Note that while going from \eqref{eq-toreferlater} to \eqref{eq-nomorebeta}, we pick up a factor of $1/2$ from $d^\times\beta=2d^\times t$, a factor of $2$ from restricting the sum over $b\in\Z$ to the sum over $b>0$, and another factor of $1/2$ by replacing the integral over $\beta\geq 0$ to the integral over all $\beta$ in the definition of $C(\cB)$ (which is equivalent to the integral over all $f\in\cB$).

\subsubsection{Computing the constant} \label{sec-compconst4}

In this section, we compute the value of $C(\cB)$. Recall that we defined $\cB$ to be the multiset $\cB \defeq G_0 \cdot \mc{R}_1^{(i)} \cap V_0(\R)$. Since $G_0$ is right $K'$-invariant, we may write $G_0 = \mc{S}K'$ for some $\mc{S} \subset N(\R)T$. Hence, we have that $\cB = \mc{S}K' \cdot \mc{R}_1^{(i)} \cap V_0(\R) = \mc{S} \cdot (K'\mc{R}_1^{(i)})_0$. We then have the following lemma concerning the multiplicity of the fiber of \mbox{$(K'\mc{R}^{(i)})_0$ over $\on{inv}(V(\R)^{(i)}) \subset \R^2$:}

\begin{lemma} \label{lem-whatsthefibersize}
    The map $\on{inv} \colon (K' \mc{R}^{(i)})_0 \to \{\on{inv}(V(\R)^{(i)})$ is $\wt{n}_i$ to $1$.
\end{lemma}
\begin{proof}
    We begin by noting that the map is certainly surjective since the invariants $I$ and $J$ of $x^3y+dxy^2+ey^3$ are linear in $d$ and $e$, respectively.
    It thus suffices to prove that the map $\on{inv} \colon \{f \in (K'\mc{R}^{(i)})_0 \colon b(f) > 0\} \to \on{inv}(V(\R)^{(i)})$ is $n_i$ to $1$. 
    This is a consequence of the following two facts: first, the stabilizer in $\on{PGL}_2(\R)$ of any element in the image has size $n_i$, and second, the group $N(\R)T$ acts simply transitively on $\on{inv}^{-1}(I,J) \cap V_0(\R)$ for any $(I,J) \in \on{inv}(V(\R)^{(i)})$ by Lemma~\ref{lem-cantranslate44}. Indeed, given $f \in \mc{R}^{(i)}$ having invariants $(I,J)$, and $p\theta \in \on{Stab}_{\on{PGL}_2(\R)}(f)$ with $p \in N(\R)T$ and $\theta \in K'$, the element $\theta f = p^{-1}f$ belongs to $(K'\mc{R}^{(i)})_0$ and has invariants $(I,J)$. This association yields the result. 
\end{proof}

We are now in position to compute the constant $C(\cB)$:

\begin{prop} \label{prop-cbi}
    We have that $$C(\cB) = \frac{1}{27}\Vol\bigl\{
(I,J)\in\R^2:(-1)^i\Delta(I,J)>0,\,H(I,J)<1
\bigr\}.$$
\end{prop}
\begin{proof}
We have
\begin{equation} \label{eq-jacsimplify4}
\begin{array}{rcl}
\displaystyle C(\cB)=\frac{1}{\wt{n}_i\Vol(G_0)}\int_{f \in \mc{B}} |b(f)|df&=&
\displaystyle\frac{1}{\wt{n}_i\Vol(G_0)}
\int_{f\in \mc{S}\cdot (K'\cR^{(i)}_1)_0}|b(f)|df
\\[.2in]&=&\displaystyle
\frac{2}{27\Vol(G_0)}
\int_{\substack{(I,J)\in\R^2\\(-1)^i\Delta(I,J)>0\\H(I,J)<1}}\int_{h\in \mc{S}}dhdIdJ
\\[.2in]&=&\displaystyle
\frac{2\on{Vol}_{\on{right}}(\mc{S})\Vol\bigl\{
(I,J)\in\R^2:(-1)^i\Delta(I,J)>0,\,H(I,J)<1\bigr\}}{27\on{Vol}(\mc{S}K')}
,
\end{array}
\end{equation}
where the second line follows by applying the Jacobian change-of-variables established in Proposition~\ref{prop-jac4} along with Lemma \ref{lem-whatsthefibersize}. In the third line, $\on{Vol}_{\on{right}}(\mc{S})$ denotes the volume of $\mc{S}$ with respect to the right Haar measure on $P(\R)$. But since $\on{Vol}(K)$ is normalized to be equal to $1$, we have that $\on{Vol}_{\on{right}}(\mc{S}) = \frac{1}{2}\on{Vol}(K'\mc{S})$, where the volume is computed with respect to the Haar measure on $G(\R)$. The next lemma demonstrates that $\on{Vol}(K'\mc{S}) = \on{Vol}(\mc{S}K')$:
\begin{lemma} \label{lem-volswitch}
We have $K'\mc{S} = \mc{S}K'$, and in particular $\on{Vol}(K'\mc{S}) = \on{Vol}(\mc{S}K')$.
\end{lemma}
\begin{proof}[Proof of Lemma~\ref{lem-volswitch}]
Since $G_0 = \mc{S}K'$ is left-$K'$-invariant and inversion-invariant, it follows that
\begin{equation} \label{eq-chainsl2}
K'\mc{S} \subset K'\mc{S}K' = \mc{S}K' = G_0 = G_0^{-1} = K'\mc{S}^{-1}.
\end{equation}
Since the Iwasawa decomposition of $\on{PGL}_2(\R)$ is unique,~\eqref{eq-chainsl2} implies that $\mc{S} \subset \mc{S}^{-1}$, and hence also that $\mc{S} = \mc{S}^{-1}$. Thus, $\mc{S}K' = K'\mc{S}^{-1} = K'\mc{S}$, as desired.
\end{proof}
Combining~\eqref{eq-jacsimplify4} with the result of Lemma~\ref{lem-volswitch} completes the proof of Proposition~\ref{prop-cbi}.
\end{proof}

\subsubsection{Congruence Conditions} \label{sec-sl2conditions}

We now prove Theorem~\ref{thm-bigquartic}. Let $S \subset V(\Z)^{(i)}$ be a big family, and suppose for now that $S$ is defined by congruence conditions at finitely many places (i.e., suppose that $S_p = V(\Z_p)$ for all primes $p \gg 1$). For each $b \in \Z_p \smallsetminus \{0\}$, let
$$(S_p)_0|_b \defeq \{f \in (S_p)_0 : b(f) = b\},$$
and for each $b \in \Z \smallsetminus \{0\}$, let $\nu(S_0|_b) \defeq \prod_p \on{Vol}((S_p)_0|_b)$ denote the density of the slice $S_0|_b$ in $V_0(\Z)|_b$; here, each $p$-adic volume $\on{Vol}((S_p)_0|_b)$ is computed with respect to the Euclidean measure on $V_0(\Z_p)|_b$, normalized so that $V_0(\Z_p)|_b$ has volume $1$. Then an argument identical to the one used to obtain~\eqref{eq-toreferlater} yields the following asymptotic formula:
\begin{equation}
    \begin{array}{rcl}
\displaystyle\#\left(\frac{(S_{\on{red}})_X}{\PGL_2(\Z)}\right)
 &=& \displaystyle\frac{1}{n_i\Vol(G_0)} \sum_{b \in \Z \smallsetminus \{0\}}\nu(S_0{\vert}_b)\int_{t \geq 1} \Vol\big((tY\mc{B}){\vert}_b\big)t^{-2}d^\times t +O_\epsilon\bigl(X^{\frac{3}{4}+\epsilon}\bigr).
\end{array}
\end{equation}
Note that the upper bound on $t$ in \eqref{eq-toreferlater} can be omitted, since if $t> cY^{1/2}/b^{1/2}$ for some sufficiently large constant $c$, then $(tY\cB)|_b$ is empty.
The $\on{PGL}_2(\Z)$-invariance of $S$ implies that $\nu(S_0|_b) = \nu(S_0|_{|b|})$ for all $b$. Hence, the argument used to deduce~\eqref{eq-firstimebeta} and~\eqref{eq-nomorebeta} yields the following estimate:
\begin{equation} \label{eq-congruent4}
\displaystyle\#\left(\frac{(S_{\on{red}})_X}{\PGL_2(\Z)}\right)
=
C(\cB)\cdot \Bigl(\sum_{b = 1}^\infty\frac{\nu(S_0{\vert}_b)}{b^2}\Bigr)\cdot X^{\frac{5}{6}} + O_\epsilon\big(X^{\frac{3}{4}+\epsilon}\big).
\end{equation}

To evaluate the sum over $b$ on the right-hand side of~\eqref{eq-congruent4}, we use the following property, which is a consequence of the fact that $S$ is a big family: if $p$ is a prime and $b,b' \in \Z_p \smallsetminus \{0\}$ are elements such that ${\vert}b_i{\vert}_p = {\vert}b_i'{\vert}_p$ for each $i$, then $\on{Vol}((S_p)_0{\vert}_b) = \on{Vol}((S_p)_0{\vert}_{b'})$. By repeatedly using this property, we obtain the following chain of equalities:
\begin{equation*} 
\begin{array}{rcl}
   \displaystyle \sum_{b = 1}^\infty\frac{\nu(S_0{\vert}_b)}{b^2} & = &  \displaystyle \prod_p \sum_{i = 0}^\infty
\frac{\on{Vol}((S_p)_0{\vert}_{p^i})}{p^{2i}}\\
  \displaystyle  & = &  \displaystyle \prod_p \Big(1 - \frac{1}{p}\Big)^{-1}\int_{\substack{b \in \Z_p \\ b \neq 0}} |b|_p\on{Vol}\bigl((S_p)_0{\vert}_{b}\bigr) db \\
&=& \displaystyle \prod_p\Big(1-\frac{1}{p}\Big)^{-1} \int_{f\in (S_p)_0} {\vert}b(f){\vert}_p df,
    \end{array}
\end{equation*}
where the second line above follows by partitioning the region of integration $\Z_p \smallsetminus \{0\}$ into level sets for the integrand and summing over all such level sets, and where the last line above follows just as in~\eqref{eq-nomorebeta}.

It remains to handle the case where $S$ is a big family defined by congruence conditions at infinitely many places. This case follows by using an inclusion-exclusion sieve\footnote{We do not flesh out the sieving argument here to avoid being repetitive, because in Section~\ref{sec-sieve} (to follow), we use the same sort of argument to prove Theorem~\ref{thm-acceptcubic}.} in conjunction with the following bound on the number of $\PGL_2(\Z)$-equivalence classes of forms $f$ with large $b(f)$-value:
\begin{theorem} \label{thm-quarticbsw}
Fix a real number $M > 0$. Then the number of $\PGL_2(\Z)$-equivalence classes of $($or equivalently, $P(\Z)$-orbits of$)$ elements of the set $\{f \in V_0(\Z) : H(f) < X,\, {\vert}b(f){\vert}
\geq M\}$ is bounded by $O\big(X^{\frac{5}{6}}/M\big) + O_\epsilon\big(X^{\frac{3}{4}+\epsilon}\big)$, where the implied constant is independent of $M$.
\end{theorem}
\begin{proof}
The required bound follows immediately from the proof of Theorem~\ref{thm-quarticmaincount} by simply summing~\eqref{eq-firstimebeta} over only those $b$ such that ${\vert}b{\vert} \geq M$.
\end{proof}

This concludes the proof of Theorem~\ref{thm-bigquartic}. We finish by noting that Theorem~\ref{thm-needplocalfactsl21} follows from Theorem~\ref{thm-bigquartic} by applying the Jacobian change-of-variables result in Proposition~\ref{prop-jac4} to each $p$-adic integral.


\section{Reduction theory for the action of $\GG(\Z)$ on $W(\R)$}\label{sec-cubicredux}

Fix an integer $n\geq 3$. In this section, we construct finite covers of a fundamental set for the action of $G(\Z)$ on $W(\R)$.
As in \S\ref{sec-sl2redux}, we achieve this in two steps: first, in~Sections~\ref{sec-boxstatement}--\ref{sec-fundgg'}, we choose a certain fundamental domain $\FF$ for the action of $G(\Z)$ on $G(\R)$, and then in~Section~\ref{sec-reduct}, we combine $\FF$ with fundamental sets for the action of $\GG(\R)$ on $W(\R)$ to construct \mbox{our required covers.}

Note that it suffices to construct a fundamental domain for $\on{O}_{\mc{A}}(\Z)$ on $\on{O}_{\mc{A}}(\R)$, as such a domain would also be a fundamental domain for $G(\Z)$ on $G(\R)$. Indeed, when $n$ is odd, this follows because $\on{O}_{\mc{A}}(\Z)$ contains an element of determinant $-1$, namely, the negative of the identity matrix. When $n$ is even, this follows because the $2$-torsion elements of $\on{O}_{\mc{A}}(\R)$ are contained in $\on{O}_{\mc{A}}(\Z)$. Thus, in Sections~\ref{sec-boxstatement}--\ref{sec-fundgg'}, we work with the group $\on{O}_{\mc{A}}$. Furthermore, for our counting purposes, we cannot simply use any fundamental domain $\FF$ for the action of $\on{O}_{\mc{A}}(\Z)$ on $\on{O}_{\mc{A}}(\R)$; rather, we require that $\FF$ be {\it box-shaped at infinity}, and we prove that such a choice of fundamental domain \mbox{exists in Section~\ref{sec-fundgg'}.}

\subsection{A coordinate system for $\on{O}_{\mc{A}}$}\label{sec-boxstatement} \label{sec-coord}

We begin by recalling the Iwasawa decomposition
\begin{equation*}
\on{O}_\AA(\R)=N(\R)TK,
\end{equation*}
where $N$ denotes the (algebraic) group of lower triangular unipotent matrices in $\on{O}_\AA$, $T$ denotes the set of diagonal matrices with positive entries contained in $\on{O}_\AA(\R)$, and $K$ denotes a maximal compact subgroup of $\on{O}_\AA(\R)$. We note that $T$ is a maximal torus of $\on{O}_\AA(\R)$, and that the elements of $T$ normalize $N(\R)$.

A calculation shows that the elements of $N(\R)$ \mbox{are parametrized as in Figure~\ref{fig-bigmats},}
\small
\begin{sidewaysfigure}
\vspace*{6.5in}
\centering
\begin{equation} \label{eq-uform1}
\left[\begin{array}{ccccccccccc} 1 &  & & & & & & &  &  \\
u_{21} & 1 & & & & &  & & &  &  \\
u_{31} & u_{32} & 1 & & & & & & &   &  \\
\vdots & \vdots & \vdots & \ddots &  & & & & & & \\
u_{(\lceil \frac{n}{2} \rceil-1)1} & u_{(\lceil \frac{n}{2} \rceil-1)2} & u_{(\lceil \frac{n}{2} \rceil-1)3} & \cdots & 1 & & & & & &  \\
u_{\lceil \frac{n}{2} \rceil1} & u_{\lceil \frac{n}{2} \rceil2} & u_{\lceil \frac{n}{2} \rceil3} & \cdots & u_{\lceil \frac{n}{2} \rceil \lfloor \frac{n}{2} \rfloor} & 1 & & & & & \\
u_{(\lceil \frac{n}{2} \rceil+1)1} & u_{(\lceil \frac{n}{2} \rceil+1)2} & u_{(\lceil \frac{n}{2} \rceil+1)3} &  \cdots  & -\frac{1}{2} u_{\lceil \frac{n}{2} \rceil \lfloor \frac{n}{2}\rfloor}^2+* & -u_{\lceil \frac{n}{2} \rceil \lfloor \frac{n}{2} \rfloor}  & 1 & &  & & \\
\vdots & \vdots & \vdots & \rddots & \vdots & \vdots & \vdots & \ddots & &  & \\
u_{(n-2)1} & u_{(n-2)2} & -\frac{1}{2}u_{\lceil \frac{n}{2} \rceil3}^2+* & \cdots & * & * & * &  \cdots &  1 & & \\
u_{(n-1)1} & -\frac{1}{2} u_{\lceil \frac{n}{2} \rceil2}^2 + * & * & \cdots &  * & * & * & \cdots & -u_{32} & 1 & \\
- \frac{1}{2}u_{\lceil \frac{n}{2} \rceil1}^2 + * & * & * & \cdots &  * & * & * & \cdots & * & -u_{21} & 1  \end{array}\right]
\end{equation}
\begin{equation} \label{eq-uform2}
\left[\begin{array}{cccccccccccc} 1 &  & & & & & & &  & & \\
u_{21} & 1 & & & & &  & & &  & & \\
u_{31} & u_{32} & 1 & & & & & & & & &  \\
\vdots & \vdots & \vdots & \ddots &  & & & & & & & \\
u_{(\frac{n-2}{2})1} & u_{(\frac{n-2}{2})2} & u_{(\frac{n-2}{2})3} & \cdots & 1 & & & & & & &  \\
u_{\frac{n}{2}1} & u_{\frac{n}{2}2} & u_{\frac{n}{2}3} & \cdots & u_{\frac{n}{2}(\frac{n-2}{2})} & 1 & & & & & & \\
u_{(\frac{n+2}{2})1} & u_{(\frac{n+2}{2})2} & u_{(\frac{n+2}{2})3} &  \cdots  & u_{(\frac{n+2}{2})(\frac{n-2}{2})} & 0 & 1 & & & & & \\
u_{(\frac{n+4}{2})1} & u_{(\frac{n+4}{2})2} & u_{(\frac{n+4}{2})3} &  \cdots  & * & * & -u_{\frac{n}{2}(\frac{n-2}{2})} & 1 & & & & \\
\vdots & \vdots & \vdots & \rddots & \vdots & \vdots & \vdots & \vdots & \ddots & &  & \\
u_{(n-2)1} & u_{(n-2)2} & * & \cdots & * & * & * & * &  \cdots &  1 & & \\
u_{(n-1)1} & * & * & \cdots &  * & * & * & * & \cdots & -u_{32} & 1 & \\
 * & * & * & \cdots &  * & * & * & * & \cdots & * & -u_{21} & 1  \end{array}\right]
\end{equation}
\caption{Parametrization of elements of $N(\R)$.}
 \label{fig-bigmats}
\end{sidewaysfigure}
\normalsize
where the $u_{ij} \in \R$ for $i \in \{2, \dots, n-1\}$ and \mbox{$j \in \{1, \dots, \min\{i-1,n-i\}\}$} are free parameters, and where the symbol ``$*$'' is shorthand and is read as follows: if the ``$*$'' occurs in the row-$i$, column-$j$ entry, then it denotes ``some polynomial of positive degree in the variables $\{u_{i'j'} : i' - j' \leq i-j\}$ with integer coefficients and no constant term'' (the polynomial being abbreviated depends on the matrix entry in which it occurs). We often abbreviate the tuple $$\big(u_{ij} : i \in \{2, \dots, n-1\},\, j \in \{1, \dots, \min\{i-1,n-i\}\}\big)$$ by $u$ and abuse notation by writing $u$ for the \mbox{corresponding element of $N(\R)$.}

Elements of $T$ have the form $s=\diag(t_1,\ldots,t_n)$ with $t_it_{n-i+1}=1$ for $i\in\{1,\ldots,\lceil \frac{n}{2}\rceil\}$. (Note in particular that when $n$ is odd, we have $t_{\frac{n+1}{2}}=1$.) In the sequel, it will be convenient to use the following alternative coordinates for $T$. Define the coordinates $(s_1,\ldots,s_{\lfloor \frac{n}{2}\rfloor})$ to be such that $(t_1,\ldots,t_{\lfloor \frac{n}{2}\rfloor})$ is equal to
\begin{align}
    & \Bigl(\prod_{i = 1}^{\lfloor \frac{n}{2} \rfloor}s_i^{-1},\,\prod_{i = 2}^{\lfloor \frac{n}{2} \rfloor} s_i^{-1}, \dots, \, s_{\lfloor \frac{n}{2} \rfloor}^{-1}\Bigr), \quad & \text{ if $2 \nmid n$,} \label{eq-sodd} \\
    &  \Bigl(\mathfrak{s}_n^{-1}\cdot \prod_{i = 1}^{\frac{n-4}{2}}s_i^{-1},\,\mathfrak{s}_n^{-1} \cdot \prod_{i = 2}^{\frac{n-4}{2}} s_i^{-1}, \dots, \mathfrak{s}_n^{-1} \cdot s_{\frac{n-4}{2}}^{-1},\, \mathfrak{s}_n^{-1}, \, \mathfrak{s}_n^{-1} \cdot s_{\frac{n-2}{2}}\Bigr), \quad & \text{ if $2 \mid n$,} \label{eq-seven}
    \end{align}
where $s_i > 0$ for each $i$ and where $\mathfrak{s}_n \defeq \sqrt{s_{\frac{n-2}{2}}s_{\frac{n}{2}}}$ when $n$ is even.

We denote the elements of $K$ by $\theta$. With respect to the coordinates $u$ on $N(\R)$, $s$ on $T$, and $\theta$ on $K$, the Haar measure $dg$ on $\on{O}_\AA(\R)$ is given by
 \begin{align}
 & dg = d\theta du(\delta(s)d^\times s), \,\, \text{where} \,\, du \defeq \prod_{i = 1}^{\lceil\frac{n}{2} \rceil} \prod_{j = 1}^{i-1} du_{ij},\,\, d^\times s \defeq \prod_{i = 1}^{\lfloor \frac{n}{2} \rfloor} \frac{ds_i}{s_i}, \,\, \text{and} \label{eq-haarcubic} \\
  & \quad\quad\,\,\,\,\,\,\, \delta(s) \defeq \begin{cases} \prod_{i = 1}^{\lfloor \frac{n}{2}\rfloor} s_i^{i^2 - 2i\lfloor \frac{n}{2}\rfloor}, & \text{if $2 \nmid n$,} \\ \big(s_{\frac{n-2}{2}}s_{\frac{n}{2}}\big)^{-\frac{n^2 - 2n}{8}} \prod_{i = 1}^{ \frac{n-4}{2}} s_i^{i^2 - i(n-1)}, & \text{if $2 \mid n$.} \end{cases} \nonumber
\end{align}
Above, $d\theta$ is normalized so that $\int_{\theta \in  \{\pm \on{id}\} \backslash K} d\theta = 1$, and $du$ is normalized so that $N(\Z)$ has covolume $1$ in $N(\R)$.

\subsection{A box-shaped fundamental domain for $\on{O}_\AA(\Z)$ $\curvearrowright$ $\on{O}_\AA(\R)$}\label{sec-fundgg'}

A fundamental domain $\mc{F}$ is said to be \emph{box-shaped at infinity} if it can be sandwiched as $\mc{S}_1 \subset \mc{F} \subset \mc{S}_2$, where $\mc{S}_1 \subset \mc{S}_2$ are nested generalized Siegel sets\footnote{Here, a fundamental domain $\mc{F} \subset \on{O}_{\AA}(\R)$ is defined to be a measurable subset such that there exists a subset $\mc{M} \subset \on{O}_{\AA}(\R)$ of full measure with the property that every $g \in \mc{M}$ is $\on{O}_{\AA}(\Z)$-equivalent to a unique element of $\mc{F}$. By a generalized Siegel set, we mean a finite union of Siegel sets.} satisfying the following conditions.
\begin{itemize}[leftmargin=18pt]
\item[{\rm (a)}] There exists an open subset $\mc{U}_1 \subset \mathfrak{S}_1$ of full measure such that every $\on{O}_\AA(\Z)$-orbit on $\on{O}_\AA(\R)$ meets $\mc{U}_1$ at most once;
\item[{\rm (b)}] Every $\on{O}_\AA(\Z)$-orbit on $\on{O}_\AA(\R)$ meets $\mc{S}_2$ at least once; and
\item[{\rm (c)}] The set $\mc{S}_2 \smallsetminus \mc{S}_1$ is empty ``sufficiently high in the cusp,'' in the sense that, for some $c > 0$, the set $T_c \defeq \{s = (s_1, \dots, s_{\lfloor \frac{n}{2}\rfloor}) \in T : \text{$s_i > c$ for all $i$}\}$ has the property that $\mc{S}_1 \cap NT_cK = \mc{S}_2 \cap NT_cK$.
\end{itemize}
Because they are defined by simple equations in the cusp, box-shaped fundamental domains are particularly amenable to explicit computations. For instance, we use the box-shaped property to evaluate a certain integral that arises in the proof of Theorem~\ref{thm-maincount2} (see~\eqref{eq-switch}). As another example of the utility of box-shaped fundamental domains, see~\cite{MR934172,MR1227503}, where Grenier proves the analogue of Theorem~\ref{thm-boxfunddomain} for the group $\on{SL}_n$ and remarks that his result could be used to compute certain integrals of Eisenstein series that arise when generalizing Selberg's trace formula to the group $\on{SL}_n(\Z)$. Grenier's work has had a number of applications in the literature (see, e.g.,~\cite{MR866106,MR2221138,MR2048520}), and as explained in Section~\ref{sec-fundgg'} (to follow), it plays a central role in our proof of Theorem~\ref{thm-boxfunddomain}.

In this subsection, we construct a box-shaped fundamental domain for the action of $\on{O}_\AA(\Z)$ on $\on{O}_\AA(\R)$. Specifically, we prove the following result:
\begin{theorem} \label{thm-boxfunddomain}
There exists a fundamental domain for the action of $\on{O}_\AA(\Z)$ on $\on{O}_\AA(\bR)$ that is box-shaped at infinity.
\end{theorem}
Our proof of Theorem~\ref{thm-boxfunddomain} occurs over the next five subsubsections and is structured as follows:
\begin{itemize}[leftmargin=20pt]
\item First, in Section~\ref{sec-redux}, we show that it suffices to construct the nested generalized Siegel sets $\mc{S}_1 \subset \mc{S}_2$ satisfying the properties (a)--(c) \mbox{enumerated above.}
\item Next, in Section~\ref{sec-s2}, we construct $\mc{S}_2$ in terms of a certain compact subset $\mc{N} \subset N(\R)$, and in Section~\ref{sec-choosen}, we make a convenient explicit choice for the set $\mc{N}$.
\item It then remains to construct $\mc{S}_1$, which we do using the aforementioned work of Grenier. Specifically, in Section~\ref{sec-grenier}, we recall the construction of Grenier's domain, and in Section~\ref{sec-s1}, we use his result to construct $\mc{S}_1$.
\end{itemize}

\subsubsection{Reduction to constructing $\mc{S}_1$ and $\mc{S}_2$} \label{sec-redux}

The following lemma reduces the problem of constructing the desired fundamental domain $\mc{F}$ into the simpler problem of constructing $\mc{S}_1$ and $\mc{S}_2$:

\begin{lemma}\label{lemmasiegelsand}
Let $\Lambda$ be a discrete subgroup of a Lie group $\mc{G}$ and denote by $\B(\mc{G})$ the Borel $\sigma$-algebra of $\mc{G}$. Suppose that $\mathcal{S}$ and $\mathcal{S}'$ are sets in $\B(\mc{G})$ with the property that the maps $\mathcal{S} \rightarrow \mc{G}/\Lambda$ and $\mathcal{S}' \rightarrow \mc{G}/\Lambda$ induced by $s \mapsto s\Lambda$ are, respectively, injective and surjective. Then there is a fundamental domain $\mathcal{F}$ in $\B(\mc{G})$ for the action of $\Lambda$ on $\mc{G}$ such that $\mathcal{S} \subset \mathcal{F} \subset \mathcal{S}'$.
\end{lemma}
\begin{proof} The argument is analogous to the proof of \cite[Lemma 4.1.1]{MR3307755}. Since $\Lambda$ is discrete, we can find a non-empty open subset $U \subset \mc{G}$ such that $U^{-1}U \cap \Lambda = \{\on{id}\}$. Since $\mc{G}$ is second countable, we can find a sequence of elements $\{g_n\} \subset \mc{G}$ such that $\mc{G} = \bigcup_{n=1}^\infty g_n U$. Let $\mathcal{S}''= \mathcal{S}' \smallsetminus \mathcal{S} \Lambda$ and set:
$$\mathcal{F}' = \bigcup_{n=1}^\infty \left(g_n U \cap \mathcal{S}'' \smallsetminus \bigcup_{i<n} (g_iU \cap \mathcal{S}'')\Lambda \right).$$
Lastly, define $\mathcal{F} = \mathcal{S} \cup \mathcal{F}'$ and note that this union is disjoint. Then $\mathcal{F} \in \B(\mc{G})$ since all the operations used in its construction keep us in the $\sigma$-algebra. It is simple to check that the induced map $\mathcal{F} \rightarrow \mc{G}/ \Lambda$ sending $x \mapsto x\Lambda$ is bijective. We conclude that $\mathcal{F}$ is a fundamental domain.
\end{proof}

\subsubsection{Constructing $\mc{S}_2$} \label{sec-s2}
We now construct $\mc{S}_2$ in terms of a certain compact subset $\mc{N} \subset N(\R)$, to be chosen explicitly in the next subsubsection. This construction is in essence due to Borel and Harish-Chandra (see~\cite[Section 9.2]{MR3156850} and~\cite[Section 4.1]{MR3719247}, which specialize the results of~\cite{MR0148666} and~\cite{MR147566} for semisimple groups to the case of the group $\on{SO}_{\mc{A}}$). By~\cite[Th\'{e}or\`{e}me~2.4 and Exemple~2.5]{MR0148666}, there exists a constant $c_2 > 0$ and a compact set $\mc{N} \subset N(\R)$ such that if we take $T_2 \defeq \{s = (s_1, \dots, s_{\lfloor \frac{n}{2}\rfloor}) \in T : \text{$s_i > c_2$ for all $i$}\}$, then the set
\begin{equation} \label{eq-defsiegel2}
\mc{S}_2 \defeq \mc{N} \, T_2 \, (\{\pm \on{id}\} \backslash K)
\end{equation}
meets every orbit of $\on{O}_{\mc{A}}(\Z)$ on $\on{O}_{\mc{A}}(\R)$ at least once (hence satisfying property (b) above), where $\{\pm \on{id}\} \backslash K$ denotes some strict fundamental domain for the action of the group $\{\pm \on{id}\}$ by left-multiplication on $K$ (where ``strict'' means that every coset of $\{\pm \on{id}\}$ has a unique representative).\footnote{\emph{A priori},~\cite[Th\'{e}or\`{e}me~2.4]{MR0148666} states that a finite number $\sigma$ of translates of Siegel sets can be found such that their union meets every orbit of $\on{O}_{\mc{A}}(\Z)$ on $\on{O}_{\mc{A}}(\R)/K$ at least once; here, $\sigma = \#\big(\on{O}_{\mc{A}}(\Z)\backslash \on{O}_{\mc{A}}(\Q)/\mc{P}(\Q)\big)$. But since the algebraic group $\on{O}_{\mc{A}}$ has class number $1$, it follows from~\cite[Propositions~5.4 and~5.10]{MR1278263} that $\sigma = 1$.}

\subsubsection{Choosing $\mc{N}$} \label{sec-choosen}

Having constructed $\mc{S}_2$ in terms of $\mc{N}$, we now make an explicit choice of $\mc{N}$ that will be convenient in what follows. Let $\ol{\mc{N}} \subset N(\R)$ be the subset defined as follows:
\begin{equation*}
\ol{\mc{N}} \defeq \begin{cases} \{u \in N(\R) : {\vert}u_{ij}{\vert} \leq 1 \text{ for $i = \lceil \tfrac{n}{2} \rceil$, ${\vert}u_{ij}{\vert} \leq \tfrac{1}{2}$ for $i \neq \lceil \tfrac{n}{2}\rceil$}\},\hspace*{0pt} & \text{if $n$ is odd,} \\ \{u \in N(\R) : {\vert}u_{ij}{\vert} \leq \tfrac{1}{2} \text{ for all $i,j$}\},\hspace*{0pt} & \text{if $n$ is even.} \end{cases}
\end{equation*}
The following lemma implies that by chopping $\mc{N}$ into pieces and translating them via elements of $N(\R) \cap \on{O}_{\mc{A}}(\Z)$, we can replace $\mc{N}$ with a subset of $\ol{\mc{N}}$:
\begin{lemma} \label{lem-chopandmove}
Let $u \in N(\R)$. Then there exists $\ol{u} \in N(\R) \cap \on{O}_{\mc{A}}(\Z)$ such that $\ol{u}u \in \ol{\mc{N}}$. Moreover, there exists an open subset $\mc{U}_3 \subset N(\R)$ of full measure such that for any $u \in \mc{U}_3$, there is precisely one element $\ol{u} \in N(\R) \cap \on{O}_{\mc{A}}(\Z)$ such that $\ol{u}u \in \ol{\mc{N}}$.
\end{lemma}
\begin{proof}
We construct $\ol{u}$ inductively. Upon inspecting the coordinate system on $N$ provided in Section~\ref{sec-coord}, we arrive at the following observation: if $k \in \{0, \dots, n-3\}$ is an integer and $u' \in N(\R)$ is an element such that $u'_{ij} = 0$ for all $i,j$ such that $i - j \leq k$, then \mbox{$(u'u)_{ij} = u'_{ij} + u_{ij}$} for all $i,j$ such that $i - j = k+1$. By the observation, we may choose $u_1 \in N(\R) \cap \on{O}_{\mc{A}}(\Z)$ such that $\big{\vert}(u_1)_{i(i-1)} + u_{i(i-1)}\big{\vert} \leq \frac{1}{2}$ for each $i \in \{2, \dots, \lfloor \frac{n}{2} \rfloor\}$ and such that, for $n$ odd, we have \mbox{$\big{\vert}(u_1)_{\lceil \frac{n}{2} \rceil \lfloor \frac{n}{2} \rfloor} + u_{\lceil \frac{n}{2} \rceil \lfloor \frac{n}{2}\rfloor}\big{\vert} \leq 1$.} Suppose for some \mbox{$k \in \{0, \dots, n-4\}$} we have chosen \mbox{$u_\ell \in N(\R) \cap \on{O}_{\mc{A}}(\Z)$} for each $\ell \in \{1, \dots, k+1\}$; then, by the observation, we may choose \mbox{$u_{k+2} \in N(\R) \cap \on{O}_{\mc{A}}(\Z)$} such that $\big{\vert}(u_1)_{i(i-k-2)} + u_{i(i-k-2)}\big{\vert} \leq \frac{1}{2}$ for each $i \in \{k+3, \dots, \lfloor \frac{n+k+2}{2}\rfloor\}$, unless $n$ is odd and $i = \lceil \frac{n}{2} \rceil$, in which case we can only arrange for \mbox{$\big{\vert}(u_1)_{\lceil \frac{n}{2} \rceil (\lceil \frac{n}{2}\rceil-k-2)} + u_{\lceil \frac{n}{2} \rceil (\lceil \frac{n}{2}\rceil-k-2)}\big{\vert} \leq 1$.} Having constructed $u_\ell$ for each $\ell \in \{1, \dots, n-2\}$, we then take $\ol{u} = \prod_{\ell=1}^{n-2}u_{n-2-\ell}$.

As for uniqueness on an open subset of full measure, it suffices to show that $\ol{u}$ is unique when $u$ lies in the interior of $\ol{\mc{N}}$, for then we can take $\mc{U}_3$ to be the union of all $(N(\R) \cap \on{O}_{\mc{A}}(\Z))$-translates of the interior of $\ol{\mc{N}}$. So take $u \in \ol{\mc{N}}$. If $\ol{u}u \in \ol{\mc{N}}$ for some $\ol{u} \in N(\R) \cap \on{O}_{\mc{A}}(\Z)$, then the aforementioned observation, together with the fact that $u$ lies in the interior of $\ol{\mc{N}}$, implies that $\ol{u}_{i(i-1)} = 0$ for each $i$. Proceeding inductively as we did to prove existence, we find that $\ol{u} = \on{id}$.
\end{proof}
 By Lemma~\ref{lem-chopandmove}, we may assume that $\mc{N} \subset \ol{\mc{N}}$. We next show that $\mc{N}$ can be chosen to lie within an even smaller subset of $\ol{\mc{N}}$. Let $\Gamma \subset \on{O}_{\mc{A}}(\Z)$ denote the subgroup of diagonal matrices with integer entries. One readily verifies that $\Gamma$ satisfies the following properties:
\begin{itemize}[leftmargin=20pt]
\item $\Gamma$ is a subgroup of $K$ of order $2^{\lceil \frac{n}{2}\rceil}$ centralizing $T$; \item Conjugation by elements of $\Gamma$ defines a group action on $\ol{\mc{N}}$ with the property that for any $\rho \in \Gamma$ and $u \in \ol{\mc{N}}$, we have ${\vert}(\rho \cdot u)_{ij}{\vert} = {\vert}u_{ij}{\vert}$ \mbox{for all $i,j$; and}
\item The orbit of every element of $\ol{\mc{N}}$ under the action of $\Gamma$ has a representative $u$ such that for every $n$ we have $u_{i(i-1)} \in [0, \tfrac{1}{2}]$ for each $i \in \{2, \dots, \lfloor \frac{n}{2}\rfloor\}$ and such that for odd $n$ we have $u_{\lceil \frac{n}{2} \rceil \lfloor \frac{n}{2} \rfloor} \in [-1,-\tfrac{1}{2}] \cup [0,\tfrac{1}{2}]$. The representative $u$ is unique if each $u_{ij}$ lies in the interior of the corresponding interval or union of intervals.
\end{itemize}
It follows that we can take $\mc{N}$ to lie within the subset
$$\wt{\mc{N}} \defeq \left\{u \in \ol{\mc{N}} : \begin{array}{c} u_{i(i-1)} \in [0,\tfrac{1}{2}] \text{ for each $i \in \{2, \dots, \lfloor \tfrac{n}{2} \rfloor\}$}; \\ \text{$u_{\lceil \frac{n}{2} \rceil \lfloor \frac{n}{2} \rfloor} \in [-1, -\tfrac{1}{2}] \cup [0, \tfrac{1}{2}]$ for odd $n$} \end{array}\right\}$$
By possibly expanding $\mc{N}$, we may in fact choose $\mc{N}$ to be equal to $\wt{\mc{N}}$.

\subsubsection{Grenier's domain} \label{sec-grenier}

 In the following subsubsection, we shall deduce the construction of the set $\mc{S}_1$ from a slight reformulation of Grenier's explicit box-shaped fundamental domain $\mathscr{F}$ for the action of $\on{SL}_n^{\pm}(\Z)$ on $\on{SL}_n^{\pm}(\R)$ (see~\cite{MR1227503}). To state this reformulation, we must introduce some notation. Let $\mathscr{N} \subset \on{SL}_n^{\pm}(\R)$ denote the subgroup of lower-triangular unipotent matrices; for an element $u \in \mathscr{N}$, we denote by $u_{ij}$ the row-$i$, column-$j$ entry of $u$. Let $\ol{\mathscr{N}}$ be the set defined by
 $$\ol{\mathscr{N}} \defeq \{u \in \mathscr{N} : {\vert}u_{ij}{\vert} \leq \tfrac{1}{2} \text{ for all $i,j$}\}$$
 Let $\mathscr{T} \subset \on{SL}_n^{\pm}(\R)$ denote the subgroup of diagonal matrices with positive entries; for an element $s \in \mathscr{T}$, we denote by $s_i$ the quotient of the row-$(i+1)$, column-$(i+1)$ entry of $s$ by the row-$i$, column-$i$ entry. Let $\mathscr{K} \subset \on{SL}_n^{\pm}(\R)$ denote a maximal compact subgroup containing $K$.

 Consider the subset $\mathscr{N}' \subset \ol{\mathscr{N}}$ defined as follows:
 $$\mathscr{N}' \defeq \left\{u \in \ol{\mathscr{N}} : \begin{array}{c} u_{i(i-1)} \in [0,\tfrac{1}{2}] \text{ for each $i \in \{2, \dots, \lceil \tfrac{n}{2} \rceil\}$;}  \\ \text{$u_{i(i-1)} \in [-\tfrac{1}{2},0]$ for each $i \in \{\lceil \tfrac{n}{2} \rceil+2,\dots,n\}$;} \\ \text{$u_{(\lceil \frac{n}{2} \rceil+1)\lceil \frac{n}{2} \rceil} \in [-\tfrac{1}{2},0]$ if $2 \nmid n$, $u_{(\frac{n+2}{2})\frac{n}{2}} \in [-\tfrac{1}{2},\tfrac{1}{2}]$ if $2 \mid n$} \end{array}\right\}$$
 Let $\mathscr{T}' \defeq \{s \in \mathscr{T} : s_i > c_1 \text{ for all $i$}\}$ for a sufficiently large constant $c_1 > c_2$. Let $\varepsilon_n \in \mathscr{K}$ denote the identity matrix when $n$ is odd, and the diagonal matrix whose diagonal entries are given by $\frac{n}{2}$ copies of $-1$ followed by $\frac{n}{2}$ copies of $1$ when $n$ is even, and let $\{\pm \on{id},\pm \varepsilon_n\} \backslash \mathscr{K}$ denote some strict fundamental domain for the action of the group $\{\pm \on{id}, \pm \varepsilon_n\}$ \mbox{by left-multiplication on $K$.}

 We are now in position to state our reformulation of Grenier's result:
 \begin{lemma} \label{thm-whatgreniersays}
 For every sufficiently large $c_1 > 0$, the following property holds: If
 $$us\theta,u's'\theta' \in \mathscr{N}' \, \mathscr{T}' \, (\{\pm \on{id}, \pm \varepsilon_n\} \backslash \mathscr{K})$$
 are $\on{SL}_n^{\pm}(\Z)$-equivalent elements such that $s_i,s_i' > c_1$ for all $i$ and such that $u,u'$ lie in the interior of $\mathscr{N}'$, then $us\theta = u's'\theta'$.
 \end{lemma}
 \begin{proof}
 In~\cite{MR934172}, Grenier constructs a fundamental domain $\ol{\mathscr{F}} \subset \mathscr{N} \mathscr{T}$ for the action of $\on{SL}_n^{\pm}(\Z)$ on $\on{SL}_n^{\pm}(\R)/\mathscr{K}$. The domain $\ol{\mathscr{F}}$ has the property that no two points in its interior are $\on{SL}_n^{\pm}(\Z)$-equivalent. In~\cite[Theorem~1]{MR1227503}, Grenier establishes that, for every sufficiently large $c_1 > 0$, we have $\mathscr{N}''  \mathscr{T}' \subset \ol{\mathscr{F}}$, where $\mathscr{N}'' = \{u \in \mathscr{N}' : u_{(\frac{n+2}{2})\frac{n}{2}} \in [0,\frac{1}{2}]\}$. Consequently, the set $\ol{\mathscr{F}}  (\{\pm \on{id}\} \backslash \mathscr{K})$ is a fundamental domain for the action of $\on{SL}_n^{\pm}(\Z)$ on $\on{SL}_n^{\pm}(\R)$ containing $\mathscr{N}''  \mathscr{T}'  (\{\pm \on{id}\} \backslash \mathscr{K})$. Since $\varepsilon_n  \mathscr{N}''  \varepsilon_n = \mathscr{N}'$, it follows that there is a fundamental domain for the action of $\on{SL}_n^{\pm}(\Z)$ on $\on{SL}_n^{\pm}(\R)$ containing $\mathscr{N}'  \mathscr{T}'  (\{\pm \on{id}, \pm \varepsilon_n\} \backslash \mathscr{K})$. If we have two distinct $\on{SL}_n^{\pm}(\Z)$-equivalent elements $us\theta,u's'\theta' \in \mathscr{N}'  \mathscr{T}'  (\{\pm \on{id}, \pm \varepsilon_n\} \backslash \mathscr{K})$ such that $u,u'$ lie in the interior of $\mathscr{N}'$, then one can find two distinct elements of the interior of $\ol{\mathscr{F}}$ that are $\on{SL}_n^{\pm}(\Z)$-equivalent, which is a contradiction.
 \end{proof}

 \subsubsection{Constructing $\mc{S}_1$} \label{sec-s1}

 Let $T_1 \defeq \{s \in T : s_i > c_1 \text{ for all $i$}\}$, where the constant $c_1 > 0$ is to be chosen shortly. We then take $$\mc{S}_1 \defeq \mc{N} \, T_1 \, (\{\pm \on{id}\} \backslash K).$$
It is evident that $\mc{S}_1$ and $\mc{S}_2$ sastisfy property (c) above. The following lemma states that we can choose $c_1$ so that $\mc{S}_1$ lies within $\mc{S}_2$ and satisfies property (a) above:
\begin{lemma} \label{thm-34cubic}
There exists a constant $c_1 > 0$ for which $(1)$ $\mc{S}_1 \subset \mc{S}_2$ and $(2)$ there exists an open subset $\,\mc{U}_1 \subset \mc{S}_1$ of full measure such that every orbit of $\on{O}_{\mc{A}}(\Z)$ \mbox{on $\on{O}_{\mc{A}}(\R)$ meets $\,\mc{U}_1$ at most once.}
\end{lemma}
\begin{proof}
  Let $\mc{U}_1$ be an open subset of full measure contained in the interior of $\mc{S}_1$ consisting of elements $us\theta$ satisfying the following two properties:
 \begin{itemize}[leftmargin=20pt]
     \item The stabilizer under left-multiplication by $\on{O}_{\mc{A}}(\Z)$ of $us\theta \in \on{O}_{\mc{A}}(\R)/K$ is given by $\{\pm \on{id}\}$; and
     \item There exists a unique element $u_0 \in \mathscr{N} \cap \on{SL}_n^{\pm}(\Z)$ such that $u_0u$ lies in the interior of $\mathscr{N}'$.
 \end{itemize}
   To see why we can arrange for the first property above to hold on a open subset of full measure, we use the following lemma:
   \begin{lemma} \label{lem-stabismall}
 There exists an open subset $\mc{U} \subset \on{O}_{\AA}(\R)/K$ of full measure such that the stabilizer in $\on{O}_{\AA}(\Z)$ of any $g \in \mc{U}$ is given by $K \cap \{\pm \on{id}\}$.
 \end{lemma}
 \begin{proof}[Proof of Lemma~\ref{lem-stabismall}]
 Let $\ol{\mc{F}}$ be any fundamental domain for $\on{O}_{\AA}(\Z)$ on $\on{O}_{\AA}(\R)/K$, and let $g$ be an element of the interior of $\ol{\mc{F}}$. If $\gamma \in \on{O}_{\AA}(\Z)$ stabilizes $g$, then there is an open neighborhood $U \ni g$ contained in the interior of $\ol{\mc{F}}$ such that $\gamma \cdot U$ is contained in the interior of $\ol{\mc{F}}$, implying in fact that $\gamma$ stabilizes every element of $U$. Since left-multiplication by $\gamma$ defines a real-analytic function on $\on{O}_{\AA}(\R)/K$, and since $\on{O}_{\AA}(\R)/K$ is connected (as $K$ meets both of the two connected components of $\on{O}_{\AA}(\R)$), it follows that $\gamma$ stabilizes all of $\on{O}_{\AA}(\R)/K$.

 Now, the stabilizer in $\on{O}_{\AA}(\R)$ of any element $h \in \on{O}_{\AA}(\R)/K$ is given by $hKh^{-1}$. Since $\gamma$ stabilizes all of $\on{O}_{\AA}(\R)/K$, it follows that $\gamma \in \mc{K} \defeq \bigcap_{h \in \on{O}_{\AA}(\R)} hKh^{-1}$. But because $\mc{K}$ is a compact normal subgroup of $\on{O}_{\AA}(\R)$, it follows that $\mc{K}$ is discrete. Let $\mc{K}^+$ denote the intersection of $\mc{K}$ with the identity component $\on{O}_{\AA}(\R)^+$ of $\on{O}_{\AA}(\R)$. Then $\mc{K}^+$ is a discrete normal subgroup of the connected group $\on{O}_{\AA}(\R)^+$, and so $\mc{K}^+$ is central in $\on{O}_{\AA}(\R)^+$. It follows that $\mc{K}^+ \subset \{\pm \on{id}\}$, and hence that $\mc{K}^+$ is central in $\on{O}_{\AA}(\R)$. If $\gamma' \in \mc{K} \smallsetminus \mc{K}^+$, then $\gamma'$ commutes with elements of both connected components of $\on{O}_{\AA}(\R)$, so $\gamma'$ is central in $\on{O}_{\AA}(\R)$ since $\mc{K}$ is normal. Thus, $\mc{K} \subset \{\pm \on{id}\}$, and so $\gamma = \pm \on{id}$.

 Finally, let $\mc{U}$ to be the union of all $\on{O}_{\AA}(\Z)$-translates of the interior of $\ol{\mc{F}}$. Then $\mc{U}$ is an open subset of full measure, and we have shown above that the stabilizer in $\on{O}_{\AA}(\Z)$ of any $g \in \mc{U}$ is contained in $\{\pm \on{id}\}$, as desired. 
 
 This completes the proof of Lemma~\ref{lem-stabismall}.
 \end{proof}
   
   As for the second property, observe that by the definitions of $\mc{N}$ and $\mathscr{N}'$, the desired element $u_0$ must be such that $(u_0)_{i(i-1)} = 0$ for each $i$. Then, by proving the analogue of Lemma~\ref{lem-chopandmove} for the group $\on{SL}_n^{\pm}$, one finds that there exists at least one $u_0 \in \mathscr{N} \cap \on{SL}_n^{\pm}(\Z)$ such that $u_0u \in \mathscr{N}'$. By our explicit characterization of the elements of $\mc{N}$ (see~\eqref{eq-uform1} and~\eqref{eq-uform2}), the row-$i$, column-$j$ entry of $u$ is a non-constant polynomial in the unipotent coordinates for every pair $(i,j)$ with $i > j+1$. Thus, there exists an open subset $\mc{U}_2 \subset \mc{N}$ of full measure such that for any $u \in \mc{U}_2$ and $u_0 \in \mathscr{N} \cap \on{SL}_n^{\pm}(\Z)$, the row-$i$, column-$j$ entry of $u_0u$ is not an integer multiple of $\frac{1}{2}$ for every pair $(i,j)$ with $i > j$, unless $i+1 = j = \frac{n}{2}$, in which case $u_{ij} = 0$. In particular, if for $u \in \mc{U}_2$ and $u_0 \in \mathscr{N} \cap \on{SL}_n^{\pm}(\Z)$ we have $u_0u \in \mathscr{N}'$, then $u_0u$ must in fact lie in the interior of $\mathscr{N}'$, and imitating the proof of uniqueness in Lemma~\ref{lem-chopandmove} yields that $u_0$ must be unique.

   Having defined $\mc{U}_1$, take any $us\theta, u's'\theta' \in \mc{U}_1$ such that $g \cdot us\theta = u's'\theta'$ for some $g \in \on{O}_{\mc{A}}(\Z)$. Let $\varepsilon, \varepsilon' \in \{\pm \on{id}, \pm \varepsilon_n\}$ be such that $\varepsilon\theta,\varepsilon'\theta' \in \{\pm \on{id}, \pm \varepsilon_n\}\backslash \mathscr{K}$. Let $u_0, u_0' \in \mathscr{N} \cap \on{SL}_n^{\pm}(\Z)$ be the unique elements such that $u_0(\varepsilon u \varepsilon),u_0'(\varepsilon' u' \varepsilon')$ lie in the interior of $\mathscr{N}'$. Since $s_i,s_i' > c_1$ for all $i$, it follows from Lemma~\ref{thm-whatgreniersays} that $u_0(\varepsilon u \varepsilon)s(\varepsilon\theta) = u_0'(\varepsilon' u' \varepsilon')s'(\varepsilon'\theta')$. By the uniqueness of the Iwasawa decomposition, we must have
   \begin{equation} \label{eq-equalfacts}
   u_0(\varepsilon u \varepsilon) = u_0'(\varepsilon' u' \varepsilon'),\quad s = s', \quad \varepsilon\theta =\varepsilon'\theta'.
   \end{equation}
   The third equality in~\eqref{eq-equalfacts} implies that $\varepsilon\varepsilon' \in K$, so since $\varepsilon_n \not\in K$ when $n$ is even, it follows that $\varepsilon\varepsilon' = \pm \on{id}$. But $(\{\pm \on{id}\}) \backslash K$ contains either $\theta$ or $-\theta$ and not both, meaning that $\varepsilon\varepsilon' = \on{id}$. Combining this with the first equality in~\eqref{eq-equalfacts} yields that $u$ is a translate of $u'$ by an element of $\mathscr{N} \cap \on{SL}_n^{\pm}(\Z)$, and hence by an element of $\mc{N} \cap \on{O}_{\mc{A}}(\Z)$. The uniqueness statement in Lemma~\ref{lem-chopandmove} then implies that $u = u'$. We conclude that $us\theta = us'\theta'$.

This completes the proof of Lemma~\ref{thm-34cubic}, and hence also that of Theorem~\ref{thm-boxfunddomain}.
\end{proof}

\subsection{Fundamental sets for $\GG(\bR) \curvearrowright W(\bR)$ and $\GG(\Z) \curvearrowright W(\bR)$} \label{sec-reduct}

Let $r$ and $s$ be nonnegative integers with $r+2s=n$. We define $U(\R)^{(r)}$ to be the set of monic polynomial $f(x)$ of degree $n$ with $r$ distinct real roots and $s$ distinct pairs of complex conjugate roots. We let $W(\R)^{(r)}$ denote $\inv^{-1}(U(\R)^{(r)})$, and let $W(\R)^{(r),\dist}$ be the set of elements in $W(\R)^{(r)}$ that are reducible over $\R$. Let $f$ be an element in $U(\R)^{(r)}$.
From \cite[Section 9]{MR3156850} and \cite[Section 4]{MR3719247}, we know that the set $\{B \in W(\R)^{(r),\dist} : \inv(B)=f\}$ consists of a single $\GG(\R)$-orbit, and also that the quantity
\begin{equation} \label{eq-thetarrr}
\stabsize_r \defeq \#\Stab_{G(\R)}(B)
\end{equation}
is independent of the choice of  $B\in W(\R)^{(r)}$.

We next exhibit an explicit representative of the
reducible $\GG(\R)$-orbit of $\inv^{-1}(f)$, for polynomials $f\in U(\R)$. Denote the coefficients of $f$ by $f(x)=x^n+f_1x^{n-1}+\cdots+f_n$,
and define $\sigma_0 \colon U(\R) \to W_0(\R)$ by
\begin{equation*}
\sigma_0(f) \defeq    \left[\begin{array}{ccccccccc}
& & & & & & & 1 & 0 \\
& & & & & & \rddots & 0 & \\
& & & & & 1 & & & \\
& & & & 1 & 0 & & & \\
& & & 1 & f_1 & -\frac{f_2}{2} & & & \\
& & 1 & 0 & - \frac{f_2}{2} & - f_3 & \ddots & & \\
& \rddots & & & & \ddots & \ddots & - \frac{f_{n-3}}{2} & \\
1 & 0 & & & & & - \frac{f_{n-3}}{2} & - f_{n-2} & - \frac{f_{n-1}}{2} \\
0 & & & & & & & - \frac{f_{n-1}}{2} & - f_n
\end{array}\right]
\end{equation*}
if $n$ is odd, and by
\begin{equation*}
\sigma_0(f) \defeq
\left[\begin{array}{cccccccccc} & & & & & & & & 1 & 0 \\
& & & & & & & \rddots & 0 & \\
& & & & & & 1 & & & \\
& & & & & 1 & 0 & & & \\
& & & & 1 & -\frac{f_1}{2} & & & & \\
& & & 1 & -\frac{f_1}{2} & \frac{f_1^2}{4}- f_2 & - \frac{f_3}{2} & & & \\
& & 1 & 0 & & - \frac{f_3}{2} & - f_4 & \ddots & & \\
& \rddots & & & & & \ddots & \ddots & -\frac{f_{n-3}}{2} & \\
1 & 0 & & & & & & -\frac{f_{n-3}}{2} & -f_{n-2} & - \frac{f_{n-1}}{2} \\
0 & & & & & & & & - \frac{f_{n-1}}{2} & -f_n \end{array}\right]
\end{equation*}
if $n$ is even.
These sections are constructed and used in~\cite[Equations (11) and (27)]{sqfrval}. It is easy to check that we have $\on{inv}(\sigma_0(f)) = f$, and so $\sigma_0$ is indeed a section. This explicit section $\sigma_0$ is useful in the proof of Proposition~\ref{prop-jac} (to follow), for its image consists of matrices whose entries are polynomials in the $f_i$. However, in Section~\ref{sec-cubicslice} (also to follow), we shall require a section $\sigma$ with image consisting of elements $B$ whose entries are $O(\on{H}(B))$. To this end, we rescale $\sigma_0$: given $f\in U(\R)$ having height $Y$ and nonzero discriminant, set $f_1(x)\defeq f(x/Y)$. Then $\on{H}(f_1)=1$, and we define the section $\sigma$ by $\sigma(f)\defeq \on{H}(f)\sigma(f_1)$. It is easy to check that $\sigma$ is also a section, and that the coefficients of $B$ in the image of $\sigma$ are bounded by $O(\on{H}(B))$. We now define $\fundset^{(r)}$ to be $\sigma(U(\R)^{(r)})\subset W(\R)^{(r),\dist}$. 

Finally, we have the following result, which follows immediately from an argument parallel to \cite[Section 2.1]{MR3272925}:
\begin{prop}
The multiset $\FF\cdot\fundset^{(r)}$ is a cover of a fundamental domain for the action of $\GG(\Z)$ on $W(\R)^{(r),\dist}$. More precisely, every $\GG(\Z)$-orbit of $B\in W(\R)^{(r),\dist}$ is represented exactly $\theta_r/\#\on{Stab}_{G(\R)}(B)$ times in $\FF\cdot\fundset^{(r)}$. 
\end{prop}

\section{The action of the subgroup $\mc{P}$ on the reducible hyperplane $W_0$} \label{sec-transitivity}


In this section, we examine the action of the lower-triangular (parabolic) subgroup $\mc{P}$ on the reducible hyperplane $W_0$. Specifically, in Section~\ref{sec-preduct}, we show that over many interesting base rings $R$, the action of $\mc{P}(R)$ on the set of elements in $W_0(R)$ lying over a given nondegenerate invariant polynomial is simply transitive. Then, in Section~\ref{sec-jac}, we prove a Jacobian change-of-variables formula relating the Euclidean measure on $W_0(R)$ to the product of the Haar measure on $\mc{P}(R)$ with the Euclidean measure on $U(R)$. This formula will be applied in Section~\ref{sec-archvol}--\ref{sec-congrest} and Section~\ref{sec-sieve}.

\subsection{Reduction theory for the action of $\mc{P}$ on $W_0$} \label{sec-preduct}
Let $R$ be either a field or $\Z_p$ for some prime $p$. Then, by analogy with Proposition~\ref{prop-cantranslate4}, we have the following result, which classifies the orbits and stabilizers of $\mc{P}(R)$ on $W_0(R)$:

\begin{prop} \label{lem-cantranslate}
Let $R$ be as above, and let $f\in U(R)$ be an element with unit discriminant. Then $\inv^{-1}(f)\cap W_0(R)$ consists of a single $\mc{P}(R)$-orbit, and the stabilizer of any element of this unique orbit is trivial.
\end{prop}
\begin{proof}
We first prove the transitivity claim for \emph{any} integral domain $R$ in which $2$ is invertible (this includes every field of characteristic not $2$, as well as $\Z_p$ for odd primes $p$). 
Let $B = [b_{ij}]$ be any element in
$\inv^{-1}(f)\cap W_0(R)$. The idea of the proof is to show that $B$ can be transformed under the group $\mc{P}(R)$ into an element of the image of the section $\sigma_0$, which is defined over $R$ because we assumed that $2 \in R^\times$. Since distinct elements in the image of $\sigma_0$ have distinct invariants, it will follow that $B$ is $\mc{P}(R)$-equivalent to $\sigma_0(f)$, as necessary.

We now translate $B$ into the image of $\sigma_0$. First, notice that each $b_{i(n-i)}$ is a unit in $R$: indeed, the discriminant of $f$ is a unit in $R$, and each $b_{i(n-i)}$ divides the discriminant of $f$. Thus, we may use the action of the diagonal matrices in $\mc{P}(R)$ to transform each $b_{i(n-i)}$ into $1$, and we may assume in what follows that $b_{i(n-i)} = 1$ for each $i$.

Next, using the coefficients of $N$ given in \eqref{eq-uform1} and~\eqref{eq-uform2}, we abuse notation by denoting, for any $v_{ij}\in R$, an element in $N(R)$ called $\wt{u}_{ij}$, whose $u_{ij}$-coefficient is $v_{ij}$, and whose other coefficients are $0$. We successively replace $B$ by $N(R)$-translates of itself by means of the following steps:
\begin{itemize}[leftmargin=.7cm,itemsep=3pt]
\item[{\rm (1)}] Let $v_{21} \in R$ be such that $b_{1n} + \wt{u}_{21} = 0$. In other words, $v_{21}=-b_{1n}$. We redefine $B$ to be $\wt{u}_{21} \cdot B$. We now have $b_{1n} = 0$.
\item[{\rm (2)}] Set $v_{32} \defeq -b_{2(n-1)}$, and redefine $B$ to be $\wt{u}_{32} \cdot B$. Next, set $v_{31}\defeq -b_{2n}$ and redefine $B$ to be $\wt{u}_{31} \cdot B$. We now have $b_{2(n-1)}=b_{2n}=0$.
\item[{\rm ($k$)}] Let $k \in \{3, \dots, n-2\}$, and suppose that we have
transformed the first $k-1$ rows (and hence the first $k-1$ columns) of $B$ into the required form. We now explain how to transform the $k^{\mathrm{th}}$ row (and column) of $B$ into what is needed; i.e., we explain how to clear out the entries $b_{k(n-j)}$ for $j$ in the appropriate range. First set
$v_{(k+1)\min\{k,n-k-1\}}\defeq-b_{k\max\{k+2,n-k+1\}}$ and redefine $B$ to be $\wt{u}_{(k+1)\min\{k,n-k-1\}}\cdot B$. Next set $v_{(k+1)\min\{k-1,n-k-2\}}\defeq-b_{k\max\{k+3,n-k+2\}}$ and redefine $B$ to be
$\wt{u}_{(k+1)\min\{k-1,n-k-2\}}\cdot B$. These two steps clear out $b_{k\max\{k+2,n-k+1\}}$ and $b_{k\max\{k+3,n-k+2\}}$.
Continuing in this manner, let $k' \in \{3, \dots, \min\{k-1,n-k-2\}\}$, and suppose that we have cleared out $b_{kj}$ for $j\in\max\{k+2,n-k+1\},\ldots,\max\{k+k'+1,n-k+k'\}$. Then set $v_{(k+1)\min\{k-k',n-k-k'-1\}}\defeq b_{k\max\{k+k'+2,n-k+k'+1\}}$, and redefine $B$ to be $\wt{u}_{(k+1)\min\{k-k',n-k-k'-1\}}\cdot B$. This process has now cleared out the all of the required entries in the $k^{\mathrm{th}}$ row of $B$.
\end{itemize}
Having transformed the first $n-2$ rows (and therefore columns) of $B$ into the required form, note that the last two rows are already as required.  We have thus replaced $B$ by an $\mc{P}(R)$-translate to ensure that $B$ lies in the image of $\sigma_0$, as desired.

We now handle the case where $R = \Z_2$. In this case, when $n$ is even, the argument given above works without change. But it does not quite work when $n$ is odd, because every element $u \in N(\Z_2)$ has the property that $u_{\lceil \frac{n}{2}\rceil j}$ is divisible by $2$ for each $j \in \{1, \dots, \lfloor \frac{n}{2} \rfloor\}$ (see~\eqref{eq-uform1}). Thus, in step ($\lfloor \frac{n}{2} \rfloor$) of the process outlined above, the action of $N(\Z_2)$ can only be used to make the entries $b_{\lfloor \frac{n}{2} \rfloor j}$ equal to either $0$ or $1$. Consequently, once all the steps have been completed, the resulting matrix $B$ may fail to lie in the image of $\sigma_0(f)$, but only because some of the entries $b_{\lfloor \frac{n}{2} \rfloor j}$ might be equal to $1$. Now, it is easy to verify that such a matrix $B$ has $b_{\lfloor \frac{n}{2} \rfloor j} = 1$ if and only if the $x^{2n-2j+1}$-coefficient of $f$ is odd. Thus, the arrangement of $0$'s and $1$'s in the $\lfloor \frac{n}{2} \rfloor^{\mathrm{th}}$ row of $B$ is uniquely determined by the invariants, and the claim follows.

We finally handle the case where $R$ is a field of characteristic $2$. In this case, when $n$ is even, every element of $W_0(R)$ has discriminant zero, and so the claim is moot. When $n$ is odd, the argument used in the case $R = \Z_2$ works without change.

The claim regarding the stabilizers of elements in $W_0(R)$ follows by inspection. Indeed, any element of $\mc{P}(R)$ stabilizing $B$ must fix each $b_{i(n-i)}$ and hence must lie in the subgroup $N(R) \subset \mc{P}(R)$; moreover, any element of $u \in N(R)$ can be factored as $u = \prod \wt{u}_{ij}$, and it follows from the points enumerated above that $u$ stabilizes $B$ if and only if each $\wt{u}_{ij}$ stabilizes $B$, which happens if and only if each $\wt{u}_{ij}$ is the identity matrix. (Note: this stabilizer computation works over any integral domain $R$.)
\end{proof}

By analogy with Lemma~\ref{lem-cantranslate44}, the above result has the following immediate consequence by specializing to the case $R=\R$:
\begin{lemma} \label{lem-stabisntbig}
Let $f\in U(\R)$ be nondegenerate. Then the set $\{B \in \inv^{-1}(f)_0 :  b_{k(n-k)}(B) >0 \text{ for each $k$}\}$ consists of a single $N(\R)T$-orbit.
\end{lemma}

\subsection{A Jacobian change-of-variables formula} \label{sec-jac}

Proposition~\ref{lem-cantranslate} implies that when $R = \R$ or $\Z_p$ for a prime $p$, the space $W_0(R)$ is a fibration over $U(R)$, where the generic fiber can be identified with $\mc{P}(R)$, so long as it is nonempty. Thus, the Euclidean measure on $W_0(R)$ should be related to the product of the Haar measure on $\mc{P}(R)$ with the Euclidean measure on $U(R)$. By analogy with Proposition~\ref{prop-jac4}, the following result shows that the relationship between these measures is governed by the polynomial function $\lambda$ on $W_0$ defined as in~\eqref{eq-defoflambda}:

\begin{prop} \label{prop-jac}
Let $R=\R$ or $\Z_p$ for some prime $p$. Let $\phi \colon W_0(R) \to \R$ be a measurable function. Then there exists a nonzero rational number $\mc{J} \in \Q^\times$ such that
\begin{equation*}
\int_{B\in W_0(R)}\phi(B){\vert}\lambda(B){\vert}dB=
{\vert}\mc{J}{\vert}\int_{\substack{f\in U(R)\\ \on{disc}(f)\neq 0}}
\Bigl(\sum_{B\in\frac{\inv^{-1}(f)_0}{\mc{P}(R)}}\int_{h\in \mc{P}(R)} \phi(h \cdot B)dh\Bigr)df,
\end{equation*}
where $dB$ and $df$ are Euclidean measures, where $dh$ is the right Haar measure on $\mc{P}$ given by $dh=\delta(s)^{-1}dud^\times s$, and where ${\vert}-{\vert}$ denotes the usual absolute \mbox{value on $R$.}
\end{prop}
\begin{remark}
Unlike in Proposition~\ref{prop-jac4}, which concerns a representation of small dimension, we cannot prove Proposition~\ref{prop-jac} by means of a direct computation. Instead, we follow the general four-step strategy used to prove \cite[Propositions 3.7, 3.10]{MR3272925}, which establish an analogous change-of-variables formula in a simpler setting, where the group $\mc{P}$ is replaced by a semisimple group, and the factor of $|\lambda(B)|$ is not present. 
 The first step is to note that a similar equation holds, where the Jacobian constant $\cJ$ is replaced by a Jacobian function that \emph{a priori} depends on the section $\sigma$, the group element $h$, and the invariant $f$. Second, the Jacobian function is shown to be independent of the group element $h$ using left-invariance of the Haar measure (which does not hold in our case). Third, the Jacobian function is shown to be independent of the section $\sigma$ using right-invariance of Haar measure (which does hold in our case, and the independence on the section follows in exactly the same way). Finally, once independence of $\sigma$ has been established, the section can be chosen to be a polynomial map, from which independence on the invariant can be easily deduced by comparing degrees and dimensions of the invariants and spaces involved. This final step also goes through without change for us. 

Thus, the main difference in our case is that the Haar measure $dh$ is not left-invariant. 
As we demonstrate below, the extra factor of $|\lambda(B)|$ captures how the volumes of sets in $W_0(R)$ transform under left-translation by group elements and therefore compensates for the failure of the measure $dh$ to be left-invariant.
\end{remark}

\begin{proof}[Proof of Proposition~\ref{prop-jac}]
Let $\mc{U}\subset U(\R)$ be an open set, and let $\sigma\colon \mc{U}\to W_0(\R)$ be a continuous section (such as the section $\sigma_0$ constructed in \S\ref{sec-reduct}) with respect to $\inv$. We first claim that we have
\begin{equation}\label{eq-jac1}
\int_{B \in \mc{P}(\R) \cdot \sigma(\mc{U})} \phi(B){\vert}\lambda(B){\vert} dB  = {\vert}\mc{J}{\vert}\int_{f \in \mc{U}} \int_{h\in \mc{P}(\R)} \phi(h \cdot \sigma(f)) dhdf,
\end{equation}
for some nonzero rational constant $\mc{J} \in \Q^\times$. We prove~\eqref{eq-jac1} in a series of steps.
First, by the Stone--Weierstrass theorem, we may assume that $\sigma$ is piecewise analytic, in which case we have
\begin{equation*}
\int_{B \in \mc{P}(\R) \cdot \sigma(\mc{U})} \phi(B){\vert}\lambda(B){\vert} dB  = \int_{f \in \mc{U}} \int_{h\in \mc{P}(\R)}{\vert}\mc{J_\sigma}(h,f){\vert} \phi(h\cdot \sigma(f))dhdf,
\end{equation*}
where $\mc{J}_{\sigma}(h,f)$ denotes the determinant of the Jacobian matrix arising from the change-of-variables taking the measure $\lambda(B)dB$ on to the product measure $dhdf$. Note in particular that the function $\mc{J}_{\sigma}$ is piecewise continuous in both arguments $h$ and $f$.

Second, we show that
$\mc{J}_{\sigma}(h,f)$ is independent of $h$. To do this, fix $\gamma \in \mc{P}(\R)$, and consider the transformation on $W_0(\R)$ sending $B \mapsto \gamma \cdot B$. Then there is a character $\chi_\lambda \colon \mc{P}(\R) \to \R_{> 0}$ such that $\lambda(\gamma \cdot B) d(\gamma \cdot B) = \chi_\lambda(\gamma) \lambda(B) dB$; note that $\chi_\lambda$ exists because $\lambda(B)$ is a product of coefficients that are unchanged by the action of $N(\R)$. In fact, writing $\gamma=tu=su$, the elementary computation \begin{align*}
\lambda(\gamma \cdot B)d(\gamma \cdot B) = \prod_{i = 1}^{\lfloor \frac{n}{2}\rfloor} \left(t_{i}/t_{i+1}\right)^{1-2i} \lambda(B) t_1^{n-1}\prod_{i = 2}^{\lfloor \frac{n}{2}\rfloor} t_i^{n-2i+2} dB &= t_1^{n-2}\prod_{i = 2}^{\lfloor \frac{n}{2}\rfloor} t_{i}^{n-2i} \lambda(B) dB\\
&= \delta(s)^{-1} \lambda(B) dB
\end{align*} for odd $n$ (and a similar one for even $n$), reveals that $\chi_\lambda(\gamma)=\delta(s)^{-1}$ for $\gamma=su$.
On the other hand, the transformation $B \mapsto \gamma \cdot B$ acts on $\mc{P}(\R) \times \mc{U}$ by sending $(h,f) \mapsto (\gamma h,f)$. Letting $\rho \colon \mc{P}(\R) \to \R_{>0}$ denote the modulus for the left action of $\mc{P}$ on the right Haar measure $dh$, we see that
\begin{equation} \label{eq-algtwo}
\mc{J}_{\sigma}(\gamma  h,f) d(\gamma  h) df = \rho(\gamma)\mc{J}_{\sigma}(\gamma  h,f) dhdf.
\end{equation}
By definition, $\rho(\gamma)=\delta(s)^{-1}$ for $\gamma=su$, and so we have $\rho(\gamma)=\chi_\lambda(\gamma)$. Therefore, we have
\begin{equation} \label{eq-algthree}
\mc{J}_\sigma(\gamma h,f)d(\gamma h)df=\lambda(\gamma B)dB=\chi_\lambda(\gamma)\lambda(B)dB=\chi_\lambda(\gamma)\mc{J}_\sigma(h,f)dhdf.
\end{equation}
Comparing~\eqref{eq-algtwo} and~\eqref{eq-algthree}, we see that $\mc{J}_\sigma(h,f)=\mc{J}(f)$ is independent of $h$.

Third, that $\mc{J}_{\sigma}(h,f)$ is independent of $\sigma$ follows from an argument identical to Step 2 in the proof of~\cite[Proposition 3.10]{MR3272925}; this step relies crucially on the fact that the measure $dh$ is right-invariant. Thus, we can take $\sigma$ to be the polynomial section $\sigma_0$ defined in Section~\ref{sec-reduct}. Having made this choice of section, that $\mc{J}_{\sigma_0}(h,f)$ is independent of $f$ and given by a nonzero rational constant follows from an argument identical to Steps 3 and 4 in the proof of~\cite[Proposition 3.10]{MR3272925}.

We have therefore proven \eqref{eq-jac1}. Proposition \ref{prop-jac} now follows from \eqref{eq-jac1} and the principle of permanence of identities in a manner identical to how \cite[Proposition 3.7]{MR3272925}
is deduced from \cite[Proposition 3.10]{MR3272925}.
\end{proof}

We conclude this section by computing the value of the Jacobian constant ${\vert}\mc{J}{\vert} \in \Q^\times$ that arises in Proposition~\ref{prop-jac}:
\begin{prop} \label{prop-jacobian}
The value of ${\vert}\mc{J}{\vert}$ is $1$ when $n$ is odd and $2^{-\frac{n}{2}}$ when $n$ is even.
\end{prop}
\begin{proof}
To compute ${\vert}\mc{J}{\vert}$, it suffices to compute ${\vert}\cJ{\vert}_p$ for each $p$ since $\mc{J} \in \Q^\times$. To do this, we construct convenient sets in $W_0(\Z_p)$ whose volumes are computed in two different ways: first, using Proposition \ref{prop-jac}, and second, via an $\F_p$-point count. Equating the two answers yields \mbox{the value of ${\vert}\mc{J}{\vert}_p$.}

\vspace*{0.1cm}
\noindent \emph{Case 1: $p > 2$}: Fix a nondegenerate polynomial $f \in U(\mathbb{F}_p)$, and let $\phi_p \colon W_0(\Z_p) \to \R$ be the indicator function of the set
$$\Sigma \defeq \left\{B \in W_0(\Z_p) : \on{inv}(B) \equiv f\,(\on{mod}p)\right\}.$$
By Proposition~\ref{lem-cantranslate}, the group $\mc{P}(\Z_p)$ acts simply transitively on the set of elements in $\Sigma$ having any fixed invariant polynomial. Hence, from Proposition \ref{prop-jac}, we obtain
\begin{align} \label{eq-prop28right}
\Vol(\Sigma)=
    & {\vert}\mc{J}{\vert}_p \cdot \on{Vol}(\mc{P}(\Z_p)) \int_{\substack{g \in U(\Z_p) \\ g\equiv f\,(\on{mod}p)}}  dg
    ={\vert}\mc{J}{\vert}_p \cdot \on{Vol}(\mc{P}(\Z_p)) \cdot p^{-\dim U}.
\end{align}
On the other hand, Proposition~\ref{lem-cantranslate} also implies that the group $\mc{P}(\F_p)$ acts simply transitively on the mod-$p$ reduction $\ol{\Sigma}$ of $\Sigma$. Thus, we have
\begin{equation} \label{eq-sigmafsize}
\#\ol{\Sigma} = \#\mc{P}(\mathbb{F}_p).
\end{equation}
Since $\Vol(\Sigma)=p^{-\dim W_0}\cdot \#\ol{\Sigma}$, $\Vol(\mc{P}(\Z_p))=p^{-\dim \mc{P}}\cdot \#\mc{P}(\mathbb{F}_p)$, and $\dim \mc{P} +\dim U =\dim W_0$, we obtain from \eqref{eq-prop28right} and \eqref{eq-sigmafsize} that ${\vert}\mc{J}{\vert}_p=1$ for all odd primes $p$.

\vspace*{0.1cm}
\noindent \emph{Case 2: $p = 2$}: The proof here is similar to Case 1, so we highlight the differences. Pick an integer $m\gg 1$, and set $q=2^m$. This time, we pick the polynomial $f(x)=x^n\in U(\Z/q\Z)$, and \mbox{we define the set}
$$\Sigma \defeq \left\{B \in W_0(\Z_2) : \begin{array}{c}{\vert}b_{i(n-i)}(B){\vert}_2 = 1 \text{  $\forall i \in \{1, \dots, \lfloor \tfrac{n}{2} \rfloor\}$, } \\ \on{inv}(B) \equiv f\,(\on{mod}q) \end{array}\right\}$$
As before, we obtain
\begin{equation*}
\Vol(\Sigma)
    ={\vert}\mc{J}{\vert}_2 \cdot \on{Vol}(\mc{P}(\Z_2)) \cdot q^{-\dim U}.
\end{equation*}
However, the situation over $\Z/q\Z$ is more complicated. Here, the mod-$q$ reduction $\ol{\Sigma}$ of $\Sigma$ breaks up into $2^{\lfloor\frac{n}2\rfloor}$ different $\mc{P}(\Z/q\Z)$-orbits. Indeed, the $\lfloor\frac{n}2\rfloor$ different coefficients labelled $-\frac{f_i}{2}$ in the image of $\sigma_0(f)$ in Section~\ref{sec-reduct} can be taken to be either $0$ or $\frac{q}{2}$, and this gives exactly $2^{\lfloor\frac{n}2\rfloor}$ different elements that are inequivalent under the action of $\mc{P}(\Z/q\Z)$. Therefore, this time we have
\begin{equation*}
\#\ol{\Sigma} =2^{\lfloor\frac{n}2\rfloor} \cdot \#\mc{P}(\Z/q\Z).
\end{equation*}
As before, we have $\on{Vol}(\Sigma) = q^{-\dim W_0} \cdot \#\ol{\Sigma}$, and it is easy to check that we have
$\Vol(\mc{P}(\Z_2)) = 2^{\lfloor\frac{n}2\rfloor}q^{-\dim \mc{P} } \cdot \#\mc{P}(\Z/q\Z)$ when $n$ is odd and $\Vol(\mc{P}(\Z_2))=q^{-\dim \mc{P}} \cdot \#\mc{P}(\Z/q\Z)$ when $n$ is even. It follows that ${\vert}\mc{J}{\vert}_2=1$ when $n$ is odd and ${\vert}\mc{J}{\vert}_2=2^{\frac{n}{2}}$ when $n$ is even.
\end{proof}

\section{Counting reducible $\GG(\Z)$-orbits on $W(\Z)$} \label{sec-2tors}

Let $n \geq 3$, $r$, and $s$ be nonnegative integers with $r+2s=n$.\footnote{Note that definitions of quantities introduced in what follows may implicitly depend on $r$.} In this section, we obtain asymptotics for the number of reducible orbits of $\GG(\Z)$ on $W(\Z)^{(r)}$ of bounded height, thereby proving Theorems \ref{thm-maincount2}, \ref{thm-big}, and~\ref{thm-acceptcubic} using Method I. The proofs using Method II are given in the next section. 

To simplify the exposition in the rest of this section, we introduce the following notation:
\begin{itemize}[leftmargin=13pt]
\item For any set $S \subset W(\Z)$, let $S_{\on{red}} \subset S$ be the subset of reducible elements of $S$; for $X > 0$, let $S_X \defeq \{B \in S : \on{H}(B) < X\}$; and as before, let $S_0 \defeq S\cap W_0(\Z)$ be the set of elements of $S$ that lie on the reducible hyperplane.
\item Let $G_0 \subset \GG(\bR)$ be a fixed nonempty open bounded set such that  $G_0^{-1}=G_0$ and $G_0$ is left- and right-$K$-invariant. As explained in \S\ref{sec-countpgl}, such a set can be constructed by starting with a nonempty open bounded set $G_0'$ and taking $G_0=K(G_0'\cup G_0'^{-1}) K$.
\item Define the multiset $\cB_\infty$ by
\begin{equation*}
\cB_\infty \defeq G_0\cdot \fundset^{(r)}\cap W_0(\R).
\end{equation*}
Set $\cB \defeq (\cB_\infty)_1$, and 
note that by the construction of $\fundset$, we have \mbox{$(\cB_\infty)_X=X\cB$.}
\item We define the quantity $C(\cB)$ by
\begin{equation*}
    C(\cB)\defeq\frac{1}{\wt\theta_r\on{Vol}(G_0)} \cdot
\int_{B\in \cB}{\vert}\lambda(B){\vert}dB,
\end{equation*}
where the volume of $G_0$ is computed using the Haar measure $dg$, and where $\wt\theta_r \defeq 2^{\lceil \frac{n}{2}\rceil}\theta_r$, with $\theta_r$ as defined in~\eqref{eq-thetarrr}.
\end{itemize}
\begin{itemize}[
leftmargin=13pt]
\item For a finite set $\Sigma$ of $\GG(\Z)$-orbits on $W(\Z)$, let $\#'\Sigma$ be the number of elements of $\Sigma$, where each $B \in \Sigma$ is counted with weight $1/\#\on{Stab}_{\GG(\Z)}(B)$.
\end{itemize}

This section is organized as follows. After setting up Bhargava's averaging method in \S6.1, we reduce to problem of counting reducible $G(\Z)$-orbits on $W(\Z)$ to a question of counting integer points in various regions of $W_0(\Z)$. This is accomplished in Proposition \ref{prop:redorbitinW0} by combining previously obtained ``main ball'' counting estimates. The advantage of this result over simply applying the averaging method over the nonreductive group $\mc{P}$ is that the integral in the right hand side of \eqref{eq-avgNT} goes over $T_1$ instead of $T$. This comes at a cost of a fairly large (but sufficient for our purposes) error term.
The regions in $W_0(\R)$ that we need to count in are very skewed. In \S6.2, we use a slicing method to express the point count in terms of certain constants that are expressed as products of local integrals. In \S6.3, we use our Jacobian change-of-variables from the previous section to evaluate the contribution to these constants from the infinite place. Finally, in \S6.4, we evaluate the contribution from finite places, and also carry out a squarefree sieve proving Theorems \ref{thm-maincount2}, \ref{thm-big}, and \ref{thm-acceptcubic}.

\subsection{Averaging over fundamental domains} \label{sec-cubicave}

As in \S\ref{sec-sl2avg}, we begin by applying Bhargava's averaging technique, developed in \cite{MR2183288,MR2745272}.
Let $\mc{F}$ be a fundamental domain for the action of $\GG(\Z)$ on $\GG(\R)$ that is box-shaped at infinity (recall that such an $\mc{F}$ exists by Theorem~\ref{thm-boxfunddomain}). Then, by analogy with~\eqref{eq-sl2pmmult}, we obtain the following:
\begin{equation} \label{eq-switch}
\#'\left(\frac{(W(\Z)^{(r)}_{\on{red}})_X}{\GG(\Z)}\right) = \frac{1}{\theta_r\Vol(G_0)} \int_{g \in \FF} \#\big(g G_0 \cdot \fundset^{(r)}_X \cap W(\Z)_\red\big) dg.
\end{equation}
Since $\FF$ is a box-shaped fundamental domain, it follows that we can write, up to a measure-0 set, $\FF$ as the disjoint union $\FF'\cup \mc{N}T_1(\{\pm\on{id}\} \backslash K)$ where $\mc{N}$ is the compact subset of $N(\R)$ determined in Section~\ref{sec-choosen}, $T_1 \defeq \{s = (s_1, \dots, s_{\lfloor \frac{n}{2}\rfloor}) \in T : \text{$s_i > c_1$ for all $i$}\}$ is a subset of $T$, and $\FF'$ is a subset of $\mc{N}\{s = (s_1, \dots, s_{\lfloor \frac{n}{2}\rfloor}) \in T : \text{$s_i \le c_1$ for some $i$}\} (\{\pm\on{id}\} \backslash K)$.
By combining counting results from \cite{MR3156850}, \cite{MR3719247}, and \cite{sqfrval}, we prove the following result, in partial analogy with Proposition~\ref{prop-afterave4}:
\begin{prop}\label{prop:redorbitinW0}
We have
\begin{equation}\label{eq-avgNT}
\displaystyle\#'\left(\frac{(W(\Z)^{(r)}_{\on{red}})_X}{\GG(\Z)}\right) = \displaystyle\frac{1}{\wt\theta_r\Vol(G_0)} \int_{us \in \ol{\mc{N}}T_1} \#\big(usX\mc{B} \cap W_0(\Z)\big)\delta(s)dud^\times s +O_\epsilon\bigl(X^{\frac{n^2+n-0.4}{2}+\epsilon}\bigr),
\end{equation}
where $\mc{B}$ is the multiset $\mc{B} \defeq (\mc{B}_\infty)_1 \defeq G_0\cdot\fundset^{(r)}_1\cap W_0(\R)$, and $\ol{\mc{N}}$ is defined in Section~\ref{sec-choosen}.
\end{prop}
\begin{proof}
Recall that $B = [b_{ij}]\in W_0(\R)$ if and only if $b_{ij}=0$ for all $i+j < n$. Borrowing terminology from \cite{sqfrval2}, we break up the integral on the right hand side of \eqref{eq-switch} into three regions, the main body, the shallow cusp, and the deep cusp, and we estimate the contributions from each region separately.

The main body is the region of the integral over $\FF$ consisting of $g\in\FF$ such that $g G_0\cdot\RR_X^{(r)}$ contains integer points $B\in W(\Z)$ with $b_{11} \neq 0$. The number of {\it reducible} elements in the main ball has been shown to be negligible in \cite[Proposition 10.7]{MR3156850} (for odd $n$) and \cite[Proposition 23]{MR3719247} (for even $n$). These estimates have been improved to a power-saving bound of $O_\epsilon(X^{\frac{n^2+n-0.4}{2}+\epsilon})$ in \cite[Proposition 2.6]{sqfrval} and \cite[Proposition 3.5]{sqfrval}, respectively.

Next, the shallow cusp is the region of the integral over $\FF$ consisting of $g\in\FF$ such that every $B\in g G_0\cdot\RR_X^{(r)}$ satisfies $|b_{11}|<1$, and such that $g G_0\cdot\RR_X^{(r)}$ contains integer points $B\in W(\Z)$ with $b_{\lfloor \frac{n-1}{2}\rfloor\lfloor\frac{n-1}{2}\rfloor} \neq 0$. The proofs of \cite[Proposition 10.5]{MR3156850} (for odd $n$) and \cite[Proposition 21]{MR3719247} (for even $n$) prove that the number of elements in the shallow cusp is bounded by $O(X^{\frac{n^2+n-2}{2}})$. Together, these bounds imply that we have
\begin{equation*}
\#'\left(\frac{(W(\Z)^{(r)}_{\on{red}})_X}{\GG(\Z)}\right) = \frac{1}{\theta_r\Vol(G_0)} \int_{g \in \FF} \#\big(g G_0 \cdot \fundset^{(r)}_X \cap W_0(\Z)\big) dg+O_\epsilon\bigl(X^{\frac{n^2+n-0.4}{2}+\epsilon}\bigr).
\end{equation*}

We now claim that the above estimate also holds when 
the region of integration $\FF$ is replaced by the region $\mc{N}T_1$. To prove this claim, we show that the above integral is negligible when $\FF$ is replaced by $\FF'$. However, this also follows from the previous bounds since $\FF'$ lies within the main body and the shallow cusp. Indeed, note that for an element $ntk\in\FF$ to lie within the deep cusp (i.e., in the cusp but not the shallow cusp), we must have $\log t_{\lfloor\frac{n-1}{2}\rfloor} \gg \log X$. Moreover, since $ntk\in\FF$, the condition $\log t_{\lfloor \frac{n-1}{2}\rfloor}\gg \log X$ automatically implies that $\log s_i\gg \log X$ for every $i$, and thus $ntk\not\in\FF'$. Finally, we replace $\mc{N}$ with $\ol{\mc{N}}$ (see Section~\ref{sec-choosen} for the definition), and to compensate, we divide by the order of the subgroup $\Gamma \subset K$, which is $2^{\lceil \frac{n}{2}\rceil}$. The result now follows from the definitions of $\theta_r$ and $\wt{\theta}_r$.
\end{proof}

\subsection{Slicing} \label{sec-cubicslice}
Just like in \S\ref{sec-sl2pmslicing}, Proposition \ref{prop-davenport} is not by itself sufficient to estimate the number of integral points in the region $usX\mc{B}$, which is typically quite skewed. 
Instead, we fiber the region $usX\mc{B}$ by the coefficients $b_{1(n-1)},\ldots,b_{\lfloor \frac{n}{2}\rfloor\lceil \frac{n}{2}\rceil}$. For any $b = (b_1, \dots, b_{\lfloor\frac{n}{2}\rfloor}) \in (\R \smallsetminus \{0\})^{\lfloor \frac{n}{2}\rfloor}$, and any $\mathcal{S}\subset W(\R)$, let $\mathcal{S}{\vert}_b$ denote the {\it slice of $\mathcal{S}$ at $b$}, i.e.,
\begin{equation*}
\mathcal{S}{\vert}_b\defeq
\bigl\{B\in\mathcal{S} \cap W_0(\R):b_{k(n-k)}(B)=b_k,\, \forall k\,\in \{1, \dots, \lfloor \tfrac{n}{2}\rfloor\}\bigr\}.
\end{equation*}
We can express the integrand of the right-hand side of~\eqref{eq-avgNT} as
\begin{equation} \label{eq-slicedint}
\#(usX\cB \cap W(\Z)) = \sum_{\substack{b \in \Z^{\lfloor \frac{n}{2} \rfloor} \\ b_i \neq 0 \, \forall i}} \#\bigl((usX\cB){\vert}_b \cap W(\Z)\bigr).
\end{equation}
By examining the action of $s$ on an element $B = [b_{ij}] \in W(\bR)$, we define the {\it weight} $w_{ij}=w(b_{ij})$ to be the quantity by which $s$ scales the matrix entry $b_{ij}$ for each pair $(i,j) \in \{1, \dots, n\}^2$.
For any subset $S$ of coefficients $b_{ij}$ of $W(\R)$, we let $w(S)$ denote the product of the weights of all the elements in $S$.
From Proposition~\ref{prop-davenport}, it follows that we have
\begin{equation}\label{eq-dave22}
\#\big((usX\cB){\vert}_b \cap W(\Z)\big) = \on{Vol}\big((usX\cB){\vert}_b\big)(1+O(X^{-1})),
\end{equation}
where the error term is seen to be $X^{-1}$ times the main term as follows. The weight of every coefficient in $W_0(\R)$ not being sliced over is $\gg 1$: indeed, note that $w(b_{k(n+1-k)})=1$ for each $k$, and that the remaining weights are all at least as big. As a consequence, the range of each coefficient varying in $(usX\cB)_b$ is $\gg X$, and the volume of $(usX\cB){\vert}_b$ is asymptotic to the product of the ranges of these varying coefficients. Proposition~\ref{prop-davenport} then yields a saving of size $X$, as necessary.

Now, since unipotent transformations preserve both the value of $b$ and the volume, we have
$\on{Vol}\big((usX\cB){\vert}_b\big) = \Vol\big((sX\cB){\vert}_b\big).$
Recall that we have normalized measures to ensure $\on{Vol}(\ol{\mc{N}})=1$. Therefore, \eqref{eq-avgNT}, \eqref{eq-slicedint}, and \eqref{eq-dave22} yield the following:
\begin{equation}\label{eq-afterslice}
\begin{array}{rcl}
\displaystyle\#'\left(\frac{(W(\Z)^{(r)}_{\on{red}})_X}{\GG(\Z)}\right)
 &=& \displaystyle\frac{1}{\wt\theta_r\Vol(G_0)} \sum_{\substack{b \in \Z^{\lfloor \frac{n}{2} \rfloor} \\ 
 b_i \neq 0 \, \forall i}}\int_{s \in T_1} \Vol\big((sX\mc{B}){\vert}_b\big)\delta(s)d^\times s \\
& & \qquad\qquad\qquad\qquad\qquad\qquad +O_\epsilon\bigl(X^{\frac{n^2+n-0.4}{2}+\epsilon}\bigr).
\end{array}
\end{equation}

Define an action of $T$ on $(\R \smallsetminus \{0\})^{\lfloor\frac{n}{2}\rfloor}$ by setting $s((b_i)_i)\defeq(w(b_{i(n-i)})b_i)_i$.  For fixed $X$ and $b=(b_i)_i\in (\R \smallsetminus \{0\})^{\lfloor\frac{n}2\rfloor}$, we write $X^{-1}s^{-1}(b)\eqdef\beta \eqdef(\beta_i)_i$. Let $S$ denote the set of coefficients of $W_0$, and write $S=S_0\sqcup S^\flat$, where $S_0$ is the set of all $b_{i(n-i)}$, and $S^\flat\defeq S\smallsetminus S_0$. Since $(sX\cB){\vert}_b=sX(\cB{\vert}_{\beta})$, it follows that we have
\begin{equation}\label{eq-renormvol}
\Vol\bigl((sX\cB){\vert}_b\bigr)=
\Vol\bigl(sX(\cB{\vert}_{\beta})\bigr)=
X^{\dim(S^\flat)}w(S^\flat)\Vol(\cB{\vert}_\beta).
\end{equation}
Consider the change-of-variables $s_i\mapsto\beta_i$, and note that $d^\times s=d^\times\beta \defeq \prod_i (d\beta_i/\beta_i)$.
Write $n=2g+1$ when $n$ is odd, and $n=2g+2$ when $n$ is even. We have $s_k=X\beta_kb_k^{-1}$ for all $1\leq k\leq g$. When $n$ is even, we further have $s_{g+1}=X^2\beta_g\beta_{g+1}b_g^{-1}b_{g+1}^{-1}$.
A direct computation now yields that
\begin{equation*}
\begin{array}{rccl}
X^{\dim S^\flat}w(S^\flat)\delta(s)&=&
X^{\dim S^\flat}\prod_{k=1}^{g}s_k^{2k}
&\mbox{for $n$ odd,}
\\[.1in]
X^{\dim S^\flat}w(S^\flat)\delta(s)&=&
X^{\dim S^\flat}(s_g^{g-1}s_{g+1}^{g+1})\prod_{k=1}^{g-1}s_k^{2k}
&\mbox{for $n$ even.}
\end{array}
\end{equation*}
Therefore, defining $\mc{Z}(a_1,\ldots,a_g)\defeq \prod_{k=1}^g a_k^{2k}$ when $n=2g+1$ is odd and $\mc{Z}(a_1,\ldots,a_{g+1})\defeq a_{g+1}^{g+1}\prod_{k=1}^g a_k^{2k}$ when $n=2g+2$ is even, we have
\begin{equation}\label{eq-stobeta}
X^{\dim S^\flat}w(S^\flat)\delta(s)=
X^{\frac{n^2+n}{2}}\frac{\mc{Z}(\beta)}{\mc{Z}(b)}.
\end{equation}

We now examine each individual summand on the right-hand side of \eqref{eq-afterslice}. For $b=(b_i)_i\in (\R \smallsetminus \{0\})^{\lfloor\frac{n}2\rfloor}$, let ${\vert}b{\vert}$ denote $({\vert}b_i{\vert})_i$. Recall that $G_0$ is $K$-invariant, and recall from Section~\ref{sec-choosen} that $K$ contains every diagonal matrix in $G(\Z)$ with each entry $\pm 1$. It follows that we have $\Vol(\cB{\vert}_b)=\Vol(\cB{\vert}_{{\vert}b{\vert}})$. Therefore, from \eqref{eq-renormvol} and \eqref{eq-stobeta}, we deduce that
\begin{equation*}
\int_{s\in T_1}\Vol\big((sX\mc{B}){\vert}_b\big)\delta(s)d^\times s=\frac{X^{\frac{n^2+n}{2}}}{\mc{Z}({\vert}b{\vert})}
\int_{\beta \in \R_{\geq 0}^{\lfloor\frac{n}{2}\rfloor}\smallsetminus T_1'}\mc{Z}(\beta)\Vol(\cB{\vert}_\beta)d^\times\beta,
\end{equation*}
for each $b=(b_i)_i\in (\Z \smallsetminus \{0\})^{\lfloor\frac{n}2\rfloor}$,
where $T_1'$ is a region contained in the set of those $\beta$ for which $\beta_i\ll X^{-1}$ for at least one $i$. Since $\cB$ is a bounded set and the integral of $\mc{Z}(\beta)d^\times\beta$ over $T_1'$ is clearly bounded by $O(X^{-1})$, we find that
\begin{equation}\label{eq-slicefin}
\begin{array}{rcl}
\displaystyle\int_{s\in T_1}\Vol\big((sX\mc{B}){\vert}_b\big)\delta(s)d^\times s&=&
\displaystyle\frac{X^{\frac{n^2+n}{2}}}{\mc{Z}({\vert}b{\vert})}
\int_{\beta\in\R_{>0}^{\lfloor\frac{n}{2}\rfloor}}\mc{Z}(\beta)\Vol(\cB{\vert}_\beta)d^\times\beta \\
& & \qquad\qquad\qquad\qquad\qquad\, +\,O\big(X^{\frac{n^2+n-2}{2}}\big)
\\[.2in]&=&
\displaystyle\frac{X^{\frac{n^2+n}{2}}}{\mc{Z}({\vert}b{\vert})}\int_{B\in\cB^+}\lambda(B)dB+O\big(X^{\frac{n^2+n-2}{2}}\big),
\end{array}
\end{equation}
where $\cB^+ \defeq \{B \in \cB : b_{i(n-i)}(B) > 0, \,\forall i \in \{1, \dots, \lfloor \frac{n}{2}\rfloor\}\}$. Substituting~\eqref{eq-slicefin} into \eqref{eq-afterslice} and summing over $b$, we immediately obtain the following:
\begin{equation}\label{eq:almostT1}
\#'\left(\frac{(W(\Z)^{(r)}_{\on{red}})_X}{\GG(\Z)}\right)=
C(\mc{B})\cdot C_n^{{\rm fin}}\cdot
X^{\frac{n^2+n}{2}}+O_\epsilon\big(X^{\frac{n^2+n-0.4}{2}+\epsilon}\big).
\end{equation}

To recover Theorem \ref{thm-maincount2} from~\eqref{eq:almostT1}, it remains to prove two facts: first, that the constant $C(\mc{B})$ is equal to $C_{n,r}^{{\rm inf}}$, and second, that the $\#'$-count and the $\#$-count differ by only a negligible amount.
We verify these facts in the next two subsections.

\subsection{Computing the constant} \label{sec-archvol}

In this subsection, we compute the value of $C(\cB)$. Recall that we defined $\cB$ to be the multiset $\cB \defeq G_0\cdot \fundset_1^{(r)}\cap W_0(\R)$. Since $G_0$ is right-$K$-invariant, we may write $G_0=\mc{S}K$ for some $\mc{S}\subset N(\R)T$. Hence, we have that $\cB = \mc{S}K \cdot \mc{R}^{(r)}_1 \cap W_0(\R) = \mc{S} \cdot (K\mc{R}_1^{(r)})_0$. 
Then, by analogy with Lemma~\ref{lem-whatsthefibersize}, we have the following lemma concerning the multiplicity of the fiber of $(K\fundset^{(r)})_0$ over $U(\R)^{(r)}$:
\begin{lemma}\label{lem-cover}
The map $\inv\colon (K\fundset^{(r)})_0\to U(\R)^{(r)}$
is $\wt{\theta}_r$ to $1$.
\end{lemma}
\begin{proof}
It suffices to prove that the map $\on{inv} \colon \{B \in (K\fundset^{(r)})_0 : b_{i(n-i)}(B) > 0,\, \forall i \in \{1, \dots, \lfloor \frac{n}{2}\rfloor\}\} \to U(\R)^{(r)}$ is $\theta_r$ to $1$. This is a consequence of the following two facts: first, the stabilizer in $G(\R)$ of any element of $W(\R)^{(r)}$ has size $\theta_r$, and second, the group $N(\R)T$ acts simply transitively on $\inv^{-1}(f)\cap W_0(\R)$ for any $f\in U(\R)^{(r)}$ by Lemma~\ref{lem-stabisntbig}. Indeed, given $B\in \RR^{(r)}$ having invariant polynomial $f$, and $pk\in\Stab_{\GG(\R)}(B)$ with $p\in N(\R)T$ and $\theta \in K$, the element $\theta B=p^{-1}B$ belongs to $(K\RR^{(r)})_0$ and has invariant polynomial $f$. This association yields the result.
\end{proof}

Now, by analogy with Proposition~\ref{prop-cbi}, we are in position to compute the constant $C(\cB)$:

\begin{prop}\label{prop:infC}
We have that $C(\mc{B})=C_{n,r}^{{\rm inf}}$.
\end{prop}
\begin{proof}
We have
\begin{equation} \label{eq-jacsimplify}
\begin{array}{rcl}
C(\cB)&=&\displaystyle\frac{1}{\theta_r\Vol(G_0)}
\int_{B\in\mc{S}\cdot (K\fundset^{(r)}_1)_0}
{\vert}\lambda(B){\vert}dB
\\[.2in]&=&\displaystyle
\frac{{\vert}\mc{J}{\vert}\Vol_{\on{right}}(\mc{S})\Vol\bigl(\{f \in U(\R)^{(r)} : \on{H}(f) < 1\}\bigr)}{\Vol(\mc{S}K)}
\end{array}
\end{equation}
where the second equality follows by applying the Jacobian change-of-variables established in Proposition~\ref{prop-jac} along with Lemma \ref{lem-cover}, and where $\on{Vol}_{\on{right}}(\mc{S})$ denotes the volume of $\mc{S}$ with respect to the right Haar measure on $\mc{P}(\R)$. But since $\on{Vol}(K)$ is normalized to be equal to $1$, we have that $\on{Vol}_{\on{right}}(\mc{S}) = \on{Vol}(K\mc{S})$. In complete analogy with Lemma~\ref{lem-volswitch}, we have that $\on{Vol}(K\mc{S}) = \on{Vol}(\mc{S}K)$. Combining this with~\eqref{eq-jacsimplify} and the computation of ${\vert}\mc{J}{\vert}$ performed in Proposition~\ref{prop-jacobian} completes the proof of Proposition~\ref{prop:infC}.
\end{proof}

For the last missing ingredient in the proof of Theorem \ref{thm-maincount2}, we prove that there is no asymptotic difference between the $\#'$-count and the $\#$-count. To this end, let $W(\Z)_\bs \defeq \{B \in W(\Z)_{\on{red}}^{(r)}: \on{Stab}_{\GG(\Z)}(B) \neq 1\}$. Then we have the \mbox{following result:}
\begin{prop} \label{prop-bs}
We have that
\begin{equation*}
\displaystyle\#\left(\frac{(W(\Z)_\bs)_X}{\GG(\Z)}\right)
=O_\epsilon\big(X^{\frac{n^2+n-0.4}{2}+\epsilon}\big)
\end{equation*}
\end{prop}
\begin{proof}
Following \eqref{eq-avgNT} and \eqref{eq-afterslice}, we have that
\begin{equation*}
\begin{array}{rcl}
\displaystyle\#\left(\frac{(W(\Z)_\bs)_X}{\GG(\Z)}\right) &\ll& \displaystyle\int_{us \in \mc{N}T_1} \#\big(usX\mc{B} \cap W(\Z)_\bs\big)\delta(s)dud^\times s
\\[.2in]&\ll&\displaystyle
\sum_{\substack{b \in \Z^{\lfloor \frac{n}{2} \rfloor} \\ b_i \neq 0 \, \forall i}}\int_{s\in T_1}\#\big(sX\mc{B}{\vert}_b \cap (W(\Z)_\bs)\big)\delta(s)dud^\times s,
\end{array}
\end{equation*}
up to an error of $X^{\frac{n^2+n-0.4}{2}+\epsilon}$.

Next, note that if $B\in W(\Z)_\bs$, then the reduction of $B$ mod $p$ has a nontrivial stabilizer for every prime $p$. Fix $b=(b_i)_i \in (\Z \smallsetminus \{0\})^{\lfloor \frac{n}{2} \rfloor}$, and let $p$ be a prime such that $p\nmid\prod b_i$. We claim that a positive proportion of elements in $W_0(\F_p){\vert}_b\defeq \bigl\{B\in W_0(\F_p):b_{k(n-k)}(B)=b_k,\, \forall k\,\in \{1, \dots, \lfloor \tfrac{n}{2}\rfloor\}\bigr\}$ have trivial stabilizer in $\GG(\F_p)$. Indeed, this is true for every $B\in W(\F_p)$ whose invariant polynomial is irreducible over $\F_p$ when $n$ is odd (see~\cite[Section 10.7]{MR3156850}), and for every $B \in W(\F_p)$ whose invariant polynomial is a linear polynomial times an irreducible polynomial over $\F_p$ when $n$ is even (see~\cite[proof of Proposition~23]{MR3719247}). A positive proportion of invariant polynomials satisfy these splitting criteria, and the fiber in $W_0(\F_p)|_b$ over each polynomial has the same size (in fact, this size is $\#N(\F_p)$). A power saving estimate for each summand in the above equation now follows by using the Selberg sieve analogously to the argument in \cite{MR3264252}. (Indeed, a power saving bound from the Selberg sieve only requires the ability to count these integer points with a power saving error term, and requires a positive proportion of mod $p$ residue classes to be avoided. We omit the details of the computation of the precise power saving exponent since the argument closely follows that in \cite{MR3264252}.)
The proposition now follows since the sum over $b$ converges absolutely.
\end{proof}

Theorem \ref{thm-maincount2} now follows from \eqref{eq:almostT1} and Propositions \ref{prop:infC} and \ref{prop-bs}.

\subsection{Congruence conditions}\label{sec-congrest}

We now prove Theorem~\ref{thm-big}. Let $S\subset W(\Z)^{(r)}$ be a big family, and suppose for now that $S$ is defined by congruence conditions at finitely many places (i.e., suppose that $S_p = W(\Z_p)$ for all primes $p \gg 1$). 
For each $b \in (\Z_p \smallsetminus \{0\})^{\lfloor \frac{n}{2} \rfloor}$, let $$(S_p)_0{\vert}_{b} \defeq \{B \in (S_p)_0 : b_{i(n-i)}(B) = b_i,\, \forall i \in \{1, \dots, \lfloor \tfrac{n}{2}\rfloor\}\},$$ and for each $b \in (\Z \smallsetminus \{0\})^{\lfloor \frac{n}{2} \rfloor}$, let $\nu(S_0{\vert}_b) \defeq \prod_p \on{Vol}((S_p)_0{\vert}_b)$ denote the density of the slice $S_0{\vert}_b$ in $W_0(\Z){\vert}_b$; here, each $p$-adic volume $\on{Vol}((S_p)_0{\vert}_b)$ is computed with respect to the Euclidean measure on $W_0(\Z_p){\vert}_b$, normalized so that $W_0(\Z_p){\vert}_b$ has volume $1$.
Then an argument identical to the one used to obtain \eqref{eq-afterslice} yields the following asymptotic formula:
\begin{equation}\label{eq-ascong}
\begin{array}{rcl}
\displaystyle\#\left(\frac{(S_{\on{red}})_X}{\GG(\Z)}\right)
 &=& \displaystyle\frac{1}{\theta_r\Vol(G_0)} \sum_{\substack{b \in \Z^{\lfloor \frac{n}{2} \rfloor} \\ b_i \neq 0 \, \forall i}}\nu(S_0{\vert}_b)\int_{s \in T_1} \Vol\big((sX\mc{B}){\vert}_b\big)\delta(s)d^\times s \\
 & & \qquad\qquad\qquad\qquad\qquad\qquad\qquad 
+\,O_\epsilon\bigl(X^{\frac{n^2+n-0.4}{2}+\epsilon}\bigr).
\end{array}
\end{equation}
The $\GG(\Z)$-invariance of $S$ implies that $\nu(S_0{\vert}_b)=\nu(S_0{\vert}_{{\vert}b{\vert}})$ for all $b$. Hence,  \eqref{eq-slicefin} and \eqref{eq-ascong} yield the following estimate:
\begin{equation} \label{eq-congruent}
\displaystyle\#\left(\frac{(S_{\on{red}})_X}{\GG(\Z)}\right)
=
C(\cB) \cdot \Bigl(\sum_{\substack{b \in \Z^{\lfloor \frac{n}{2} \rfloor} \\ b_i > 0 \, \forall i}}\frac{\nu(S_0{\vert}_b)}{\mc{Z}(b)}\Bigr)\cdot X^{\frac{n^2+n}{2}} + O_\epsilon\big(X^{\frac{n^2+n-0.4}{2}+\epsilon}\big).
\end{equation}

To evaluate the sum over $b$ on the right-hand side of~\eqref{eq-congruent}, we use the following property, which is a consequence of the fact that $S$ is a big family: if $p$ is a prime and $b,b' \in (\Z_p \smallsetminus \{0\})^{\lfloor \frac{n}{2} \rfloor}$ are elements such that ${\vert}b_i{\vert}_p = {\vert}b_i'{\vert}_p$ for each $i$, then $\on{Vol}((S_p)_0{\vert}_b) = \on{Vol}((S_p)_0{\vert}_{b'})$. By repeatedly using this property, we obtain the following chain of equalities:
\begin{equation*} 
\begin{array}{rcl}
   \displaystyle \sum_{\substack{b \in \Z^{\lfloor \frac{n}{2} \rfloor} \\ b_i > 0 \, \forall i}}\frac{\nu(S_0{\vert}_b)}{\mc{Z}(b)} & = &  \displaystyle \prod_p \sum_{\substack{(i) \in \Z^{\lfloor \frac{n}{2} \rfloor} \\ i_j \geq 0\, \forall j}}
\frac{\on{Vol}\left((S_p)_0{\vert}_{\big(p^{i_1},\dots, p^{i_{\lfloor \frac{n}{2}\rfloor}}\big)}\right)}{\mc{Z}\big(p^{i_1},\dots, p^{i_{\lfloor \frac{n}{2}\rfloor}}\big)}\\
  \displaystyle  & = &  \displaystyle \prod_p \Big(1 - \frac{1}{p}\Big)^{-\lfloor \frac{n}{2} \rfloor}\int_{\substack{b \in \Z_p^{\lfloor \frac{n}{2}\rfloor} \\ b_i \neq 0\,\forall i}} \frac{\on{Vol}\bigl((S_p)_0{\vert}_{({\vert}b_1{\vert}_p^{-1},\dots,{\vert}b_{\lfloor \frac{n}{2}\rfloor}{\vert}_p^{-1})}\bigr)}{\mc{Z}\big({\vert}b_1{\vert}_p^{-1},\dots,{\vert}b_{\lfloor \frac{n}{2}\rfloor}{\vert}_p^{-1}\big)\prod_i {\vert}b_i{\vert}_p} db \\
  &=& \displaystyle \prod_p\Big(1-\frac{1}{p}\Big)^{-\lfloor\frac{n}{2}\rfloor} \int_{\substack{b \in \Z_p^{\lfloor \frac{n}{2} \rfloor} \\ b_i \neq 0\, \forall i}} \left\vert \frac{\mc{Z}(b)}{\prod_i b_i}\right\vert_p \Vol((S_p)_0{\vert}_b) db \\
&=& \displaystyle \prod_p\Big(1-\frac{1}{p}\Big)^{-\lfloor\frac{n}{2}\rfloor} \int_{B\in (S_p)_0} {\vert}\lambda(B){\vert}_p dB,
    \end{array}
\end{equation*}
where the second line above follows by partitioning the region of integration $(\Z_p \smallsetminus \{0\})^{\lfloor \frac{n}{2}\rfloor}$ into level sets for the integrand and summing over all such level sets, and where the last line above follows just as in~\eqref{eq-slicefin}.

It remains to handle the case where $S$ is a big family defined by congruence conditions at infinitely many places. By abuse of notation, let $\mc{Z}$ be the polynomial on $W_0$ defined by $\mc{Z}(B) \defeq \mc{Z}\big(b_{1(n-1)}(B),\dots,b_{\lfloor \frac{n}{2}\rfloor \lceil \frac{n}{2}\rceil}(B)\big)$. Then the case of infinitely many places follows from the case of finitely many places by using the following bound on the number of $G(\Z)$-equivalence classes of elements with large $\mc{Z}$-value, in conjunction with an inclusion-exclusion sieve:\footnote{Just as in \S\ref{sec-sl2conditions}, we do not flesh out the sieving argument here to avoid being repetitive, because in Section~\ref{sec-sieve} (to follow), we use the same sort of argument to prove Theorem~\ref{thm-acceptcubic}.}
\begin{theorem} \label{thm-bsw}
Fix a real number $M > 0$. Then the number of $G(\Z)$-equivalence classes of $($or equivalently, $\mc{P}(\Z)$-orbits of$)$ elements of the set $\{B \in W_0(\Z) : \on{H}(B) < X,\, {\vert}\mc{Z}(B){\vert}
\geq M^2\}$ is bounded by $O\big(X^{\frac{n^2+n}{2}}/M\big) + O_\epsilon\big(X^{\frac{n^2+n-0.4}{2}+\epsilon}\big)$, where the implied constant is independent of $M$.
\end{theorem}
\begin{proof}
The required bound follows immediately from the proof of Theorem~\ref{thm-maincount2} by simply summing~\eqref{eq-slicefin} over only those $b$ such that ${\vert}\mc{Z}(b){\vert} \geq M^2$.
\end{proof}
\begin{remark}
Theorem~\ref{thm-bsw} constitutes a slight strengthening of a result previously proven in \cite[Sections 2.3--2.4,\, 3.3--3.4]{sqfrval}, where Bhargava, Shankar, and Wang used the averaging method to obtain an upper bound of $O_\epsilon\big(X^{\frac{n^2+n}{2}+\epsilon}/M\big) + O_\epsilon\big(X^{\frac{n^2+n-0.4}{2}+\epsilon}\big)$. 
\end{remark}

This concludes the proof of Theorem \ref{thm-big}. We finish by noting that Theorem~\ref{thm-acceptcubic}  follows from Theorem~\ref{thm-big} by applying the Jacobian change-of-variables result in Proposition~\ref{prop-jac} to each $p$-adic integral, with $\phi$ taken to be the characteristic function of $(S_p)_0$. Indeed, Propositions~\ref{prop-jac} and~\ref{prop-jacobian} together imply the equality of each factor at the odd primes; at $p=2$, there is an extra factor of $2^{n/2}$ when $n$ is even, \mbox{which perfectly accounts for the corresponding factor at infinity.}

\section{A local-to-global principle for the action of $\mc{P}(\Z)$ on $W_0(\Z)$} \label{sec-locglob}

In this section, we develop an alternate method for counting reducible $\GG(\Z)$-orbits on $W(\Z)$. This method, which we called ``Method II'' in Section~\ref{sec-whatsnew}, consists of the following steps:
\begin{itemize}[leftmargin=20pt]
\item First, in Section~\ref{sec-stabtriv}, we consider a general representation with trivial generic stabilizer, and we prove that the integral orbits of such a representation satisfy a strong local-to-global principle, which is stated precisely in Theorem~\ref{thm-genlocglob}.
\item As shown in Proposition~\ref{lem-cantranslate}, the action of the group $\mc{P}$ on the reducible hyperplane $W_0$ is a representation with trivial generic stabilizer. In Section~\ref{sec-deducecase}, we deduce Theorem~\ref{thm-stronglocglob} by applying Theorem~\ref{thm-genlocglob} to the action of $\mc{P}$ on $W_0$.
\item Finally, in Section~\ref{sec-sieve}, we explain how to use Theorem~\ref{thm-stronglocglob} to deduce asymptotics for the number of $\mc{P}(\Z)$-orbits on $W_0(\Z)$---and hence also for the number of reducible $G(\Z)$-orbits on $W(\Z)$. The main analytic ingredient is an upper bound obtained by Bhargava, Shankar, and Wang on the number of $\mc{P}(\Z)$-orbits of elements $B \in W_0(\Z)$ with the \mbox{property that $\lambda(B)$ is large.}
\end{itemize}

\subsection{Group actions with trivial stabilizers} \label{sec-stabtriv}

In this section, we work with the following data: an algebraic group $H$, finite-dimensional $H$-representations $V$ and $I$, and an $H$-equivariant morphism $\phi \colon V \to I$, all defined over $\Z$. Note that we do not require $\phi$ to be a linear map. Any polynomial map will suffice. Suppose that the group $H$ has class number $1$ over $\Q$, meaning that $H(\mathbb{A}_{\Q})$ is the product of its subgroups $H(\mathbb{A}_{\Z})$ and $H(\Q)$ (i.e., for every $h \in H(\mathbb{A}_{\Q})$, there exist $h' \in H(\mathbb{A}_{\Z})$ and $h'' \in H(\Q)$ such that $h = h'h''$). Here $\mathbb{A}_\Q$ denotes the ring of adeles, and $\mathbb{A}_\Z$ is the ring of everywhere integral adeles. Suppose further that, for some nonempty open subscheme $\mc{I} \subset I$, defined over $\Z$, the following two assumptions hold:
\begin{enumerate}[leftmargin=20pt]
\item For every $i \in \mc{I}(\C)$, the set $\{v \in V(\C) : \phi(v) = i\}$ forms a single nonempty $H(\C)$-orbit. 
\item For every $i \in \mc{I}(\C)$ and every (or equivalently, any) $v\in V(\C)$ with $\phi(v) = i$, we have $\Stab_{H(\C)}(v) = 1$.
\end{enumerate}
In this setting, we prove the following strong local-to-global principle for the action of $H$ on $V$:
\begin{theorem} \label{thm-genlocglob}
Suppose that $i \in \mc{I}(\Z)$ is an element contained in the image $\phi(V(\Z)$. For each prime $p$, let $v_p \in V(\Z_p)$ be such that $\phi(v_p) = i$. Then there exists $v \in V(\Z)$, unique up to the action of $H(\Z)$, such that $v$ is $H(\Z_p)$-equivalent to $v_p$ for each prime $p$.
\end{theorem}
Before we give the proof of Theorem~\ref{thm-genlocglob}, we first use assumptions (1) and (2) enumerated above to deduce that an analogue of assumption (1) holds for \emph{any} subfield $K \subset \C$:
\begin{lemma} \label{lem-cohom}
Let $K \subset \C$ be a subfield. For every $i \in \mc{I}(K)$, the set $\{v \in V(K) : \phi(v) = i\}$ consists of single $H(K)$-orbit.
\end{lemma}
\begin{proof}
Let $i \in \mc{I}(K)$, and let $v,v' \in V(K)$ with $\phi(v) = \phi(v') = i$. By assumption (1), there exists $h \in H(\C)$ such that $v' = h \cdot v$. Then for any $\sigma \in \on{Gal}(\C/K)$, we have $v' = {}^\sigma h \cdot v$, so $h^{-1}{}^\sigma h \in \on{Stab}_{H(\C)}(v) = 1$ by assumption (2). It follows that $h$ is fixed by $\on{Gal}(\C/K)$, and so $h \in H(K)$, as necessary.
\end{proof}

We now use Lemma~\ref{lem-cohom}, assumption (2), and the fact that $H$ has class number $1$ over $\Q$ to complete the proof of the theorem:
\begin{proof}[Proof of Theorem~\ref{thm-genlocglob}]
Let us temporarily drop assumption (2). For a principal ideal domain $R$ with fraction field $K$ and an element $v \in V(R)$, denote by $H(K)_{v} \defeq \{h \in H(K) : h \cdot v \in H(R)\}$. Then it is clear that the set of $H(R)$-orbits contained in the $H(K)$-orbit of an element $v \in V(R)$ is in natural bijection with the double coset space $H(R) \backslash H(K)_{v}/\on{Stab}_{H(K)}(v)$. With assumption (2) reinstated, this double coset space is simply given by $H(R) \backslash H(K)_{v}$ as long as the $H$-invariant of $v$ lies in the subset $\mc{I}(R) \subset I(R)$.

Now, using the fact that $H$ has class number $1$ over $\Q$, it is proven in~\cite[proof of Proposition~3.6]{MR3272925}\footnote{To be clear,~\cite[Proposition~3.6]{MR3272925} concerns the case $H  = \on{PGL}_2$, but it is evident that the same argument goes through for any group $H$ of class number $1$ over $\Q$.} that the diagonal embedding $H(\Q) \hookrightarrow \prod_p H(\Q_p)$ induces a bijection
\begin{equation} \label{eq-locglobbij}
H(\Z) \backslash H(\Q)_{v_0} \longrightarrow \prod_p H(\Z_p) \backslash H(\Q_p)_{v_0}.
\end{equation}
Let $v_0\in V(\Z)$ be an element such that $\phi(v_0)=i$.
By the result of the previous paragraph, the bijection in~\eqref{eq-locglobbij} may be regarded as identifying the set of $H(\Z)$-orbits contained in the $H(\Q)$-orbit of $v_0$ with the product over all primes $p$ of the set of $H(\Z_p)$-orbits contained in the $H(\Q_p)$-orbit of $v_0$.

By Lemma~\ref{lem-cohom}, which implies that the $H(\Q_p)$-orbit of $v_0$ is equal to that of $v_p$ for each prime $p$, we may view the tuple $(v_p)_p$ as an element of the right-hand side of~\eqref{eq-locglobbij}. Then, under the bijection, the tuple $(v_p)_p$ corresponds to the $H(\Z)$-orbit of the desired element $v \in V(\Z)$.
\end{proof}

\subsection{Proof of Theorem~\ref{thm-stronglocglob}} \label{sec-deducecase}

We now deduce Theorem~\ref{thm-stronglocglob} from Theorem~\ref{thm-genlocglob}. To do so, we take $H = \mc{P}$, $V = W_0$, $I = U$, $\phi = \on{inv}$, and $\mc{I} \subset U$ to be the open subscheme consisting of nondegenerate polynomials. Note that the group $\mc{P}$ clearly has class number $1$ over $\Q$, and note that assumptions (1) and (2) follow immediately from Proposition~\ref{lem-cantranslate}. Applying Theorem~\ref{thm-genlocglob} then yields the following: for any $f \in \mc{I}(\Z)$, \emph{if} there exists $B_0 \in W_0(\Z)$ such that $\on{inv}(B_0) = f$ and \emph{if} there exists $B_p \in W_0(\Z_p)$ such that $\on{inv}(B_p) = f$ for each prime $p$, then there exists $B \in W_0(\Z)$, unique up to the action of $\mc{P}(\Z)$, such that $B$ is $\mc{P}(\Z_p)$-equivalent to $B_p$ for each prime $p$. To prove Theorem~\ref{thm-stronglocglob}, it now remains to verify the existence of the elements $B_0$ and $B_p$ for each prime $p$.

When $n$ is even, the existence of the elements $B_0$ and $B_p$ is implied by the existence of the integral section $\sigma_0$ given in Section~\ref{sec-reduct}---indeed, when $n$ is even, we have restricted our consideration to those monic polynomials whose $x^i$-coefficients are divisible by $2$ for each odd $i$, so the section $\sigma_0$ is defined over $\Z$ in this case. On the other hand, when $n$ is odd, the section $\sigma_0$ is not defined over $\Z$. Instead, it is shown in~\cite[Section 4.1]{MR3782066} that for any principal ideal domain $R$ and for each $f \in U(R)$, there exists a pair $(A,B)$ of $n \times n$ symmetric matrices with entries in $R$ such that: $A$ is split over $K = \on{Frac}(R)$; $A$ and $B$ share a maximal isotropic space over $K$; and $\on{det}(xA + yB) = (-1)^{\lfloor\frac{n}{2}\rfloor}f(x,y)$. By the classification of unimodular symmetric bilinear forms (see~\cite[Chapter V]{MR0344216}), $A$ is $\on{GL}_n(R)$-equivalent to $\mc{A}$. By translating $B$ with the same element of $\on{GL}_n(R)$, we obtain $B' \in W_0(R)$ satisfying $\on{inv}(B') = f$.

\subsection{Proof of Theorem~\ref{thm-acceptcubic}} \label{sec-sieve}

We now use Theorem~\ref{thm-stronglocglob} to give a second proof of Theorem~\ref{thm-acceptcubic}. We call a subset $\mathfrak{S} \subset W_0(\Z)$ a \emph{big family} if $\mathfrak{S} = W(\Z)^{(r)}_0 \cap \bigcap_p \mathfrak{S}_p$, where the sets $\mathfrak{S}_p \subset W_0(\Z_p)$ satisfy the \mbox{following properties:}
\begin{enumerate}[leftmargin=20pt]
\item[(1)] $\mathfrak{S}_p$ is $\mc{P}(\Z_p)$-invariant and is the preimage under reduction modulo $p^j$ of a nonempty subset of $W_0(\Z/p^j\Z)$ for some $j > 0$ for each $p$; and
\item[(2)] $\mathfrak{S}_p$ contains all elements $B \in W_0(\Z_p)$ such that, for all $p \gg 1$, we have that $b_{i(n-i)}(B)$ is a $p$-adic unit for some $i$.
\end{enumerate}
We then have the following variant of Theorem~\ref{thm-acceptcubic}, from which Theorem~\ref{thm-acceptcubic} readily follows by taking $\mathfrak{S}_p = (S_p)_0$ for each $p$. Indeed, given this choice of $\mathfrak{S}$, the asymptotics for reducible $\GG(\Z)$-orbits on $S$ are the same as those for $\mc{P}(\Z)$-orbits on $\mathfrak{S}$, as explained in Section~\ref{sec-whatsnew}.
\begin{theorem} \label{thm-acceptcubic2}
Let $\mathfrak{S} \subset W_0(\Z)$ be a big family. Then the number of $\mc{P}(\Z)$-orbits on $\mathfrak{S}$ of height up to $X$ is given by
\begin{equation} \label{eq-midacceptcubic2}
 \Bigl(\prod_p \int_{f \in U(\Z_p) } \#\left(\frac{\on{inv}^{-1}(f)\cap \mathfrak{S}_p}{\mc{P}(\Z_p)}\right) df\Bigr)\cdot N^{(r)}(X) + o\big(X^{\frac{n^2+n}{2}}\big),
\end{equation}
 where $df$ denotes the Euclidean measure on $U(\Z_p)$, normalized so that $U(\Z_p)$ has volume $1$.
\end{theorem}
\begin{proof}
To start, fix an integer $\mathfrak{b} \geq 1$, and suppose its prime factorization is given by $\mathfrak{b} = \prod_p p^{e_p}$. We first prove an analogue of Theorem~\ref{thm-acceptcubic} for the subfamily $\mathfrak{S}(\mathfrak{b}) \defeq \{B \in \mathfrak{S} : {\vert}\mc{Z}(B){\vert} = \mathfrak{b}\}$. Note that the subfamily $\mathfrak{S}(\mathfrak{b})$ is itself a big family, where $\mathfrak{S}(\mathfrak{b})_{p} = \{B \in \mathfrak{S}_p : {\vert}\mc{Z}(B){\vert}_p = {\vert}\mathfrak{b}{\vert}_p\}$.

If $\mathfrak{S}(\mathfrak{b}) = \varnothing$, then the result is tautologically true, so we may assume that $\mathfrak{S}(\mathfrak{b}) \neq \varnothing$. For each prime $p \mid \mathfrak{b}$, we partition $U(\Z_p)$ as $U(\Z_p) = \bigsqcup_{j = 1}^{m_p} U_{p,j}$, where each $U_{p,j}$ is a level set for the function that sends $f \in U(\Z_p)$ to $\#(\mc{P}(\Z_p) \backslash (\on{inv}^{-1}(f) \cap \mathfrak{S}(\mathfrak{b})_p))$. Write ``$E(m)$'' to mean ``a power of $m$ that depends only on $n$ and $\mathfrak{S}$.'' Then set $U_{p,j}$ is defined by congruence conditions modulo $E(p^{e_p})$. The quantity $\#(\mc{P}(\Z_p) \backslash (\on{inv}^{-1}(f) \cap \mathfrak{S}(\mathfrak{b})_p))$ is independent of the choice of $f \in U_{p,j}$ (by the definition of a level set) and is evidently bounded by $E(p^{e_p})$. 

Now for each prime $p \nmid \mathfrak{b}$, the proof of Proposition~\ref{lem-cantranslate} implies that $\#(\mc{P}(\Z_p) \backslash (\on{inv}^{-1}(f) \cap \mathfrak{S}(\mathfrak{b})_p)) = 1$ for each $f \in \on{inv}(\mathfrak{S}(\mathfrak{b})_p)$. It then follows from Theorem~\ref{thm-stronglocglob} that the quantity
\begin{equation} 
\label{eq-thistheprod}
\#\left(\frac{\on{inv}^{-1}(f) \cap \mathfrak{S}(\mathfrak{b})}{\mc{P}(\Z)}\right) = \prod_p \#\left(\frac{\on{inv}^{-1}(f) \cap \mathfrak{S}(\mathfrak{b})_p}{\mc{P}(\Z_p)}\right) =\prod_{p \mid \mathfrak{b}} \#\left(\frac{\on{inv}^{-1}(f) \cap \mathfrak{S}(\mathfrak{b})_p}{\mc{P}(\Z_p)}\right)
\end{equation}
is independent of the choice of $f \in \on{inv}(\mathfrak{S}(\mathfrak{b})) \cap \bigcap_p U_{p,j_p}$. Therefore, for each tuple $(j_p)_{p \mid \mathfrak{b}} \in \prod_{p \mid \mathfrak{b}} \{1,\dots, m_p\}$, we have
\begin{equation} \label{eq-fixedfquad}
\sum_{\substack{f \in U(\Z) \,\cap\, \bigcap_p U_{p,j_p} \\ \on{H}(f) < X }} \#\left(\frac{\on{inv}^{-1}(f) \cap \mathfrak{S}(\mathfrak{b})}{\mc{P}(\Z)}\right) = \#\left(\frac{\on{inv}^{-1}(f^*) \cap \mathfrak{S}(\mathfrak{b})}{\mc{P}(\Z)}\right) \cdot \sum_{\substack{f \in \on{inv}(\mathfrak{S}(\mathfrak{b})) \,\cap\, \bigcap_p U_{p,j_p} \\ \on{H}(f) < X }} 1
\end{equation}
where $f^* \in \on{inv}(\mathfrak{S}(\mathfrak{b})) \,\cap\, \bigcap_{p \mid \mathfrak{b}} U_{p,j_p}$ is any fixed element. Since $\mathfrak{S}$ is a big family, it follows that $\on{inv}(\mathfrak{S}(\mathfrak{b})_p) = U(\Z_p)$ for every $p \gg 1$ that does not divide $\mathfrak{b}$. 
As the set $\on{inv}(\mathfrak{S}(\mathfrak{b})) \,\cap\, \bigcap_p U_{p,j_p}$ is defined by congruence conditions modulo $E(\mathfrak{b})$, since $\on{inv}(\mathfrak{S}(\mathfrak{b})_p) \cap U_{p,j_p}$ is defined by congruence conditions modulo $E(p^{e_p})$ for each $p$, we obtain the following asymptotic:
\begin{equation} \label{eq-fixedfquad2}
\begin{array}{rcl}
\displaystyle \sum_{\substack{f \in \on{inv}(\mathfrak{S}(\mathfrak{b})) \,\cap\, \bigcap_p U_{p,j_p} \\ \on{H}(f) < X }} 1 & =& \displaystyle N^{(r)}(X) \cdot \prod_{p \mid \mathfrak{b}} \int_{f \in \on{inv}(\mathfrak{S}(\mathfrak{b})_p) \cap U_{p,j_p}} df \cdot \prod_{p \nmid \mathfrak{b}} \int_{f \in \on{inv}(\mathfrak{S}(\mathfrak{b})_p)} df \\
& & \qquad\qquad\qquad\qquad\qquad \,\,\,\,+ \,O_\epsilon\big(E(\mathfrak{b})X^{\frac{n^2+n-0.4}{2}+\epsilon}\big).
\end{array}
\end{equation}
 Substituting the asymptotic~\eqref{eq-fixedfquad2} into the right-hand side of~\eqref{eq-fixedfquad}, applying~\eqref{eq-thistheprod} to the resulting expression, and summing that over tuples $(j_p)_{p \mid \mathfrak{b}} \in \prod_{p \mid \mathfrak{b}} \{1,\dots, m_p\}$ yields the following:
\begin{equation}\label{eq-fixedconductoresult}
\begin{array}{rcl}
\displaystyle \sum_{\substack{f \in U(\Z)\\ \on{H}(f) < X }}  \#\left(\frac{\on{inv}^{-1}(f) \cap \mathfrak{S}(\mathfrak{b})}{\mc{P}(\Z)}\right) &= & \displaystyle N^{(r)}(X) \cdot \prod_{p} \int_{f \in U(\Z_p)} \#\left(\frac{\on{inv}^{-1}(f) \cap \mathfrak{S}(\mathfrak{b})_p}{\mc{P}(\Z_p)}\right) df \\
& &\qquad\qquad\qquad\qquad\quad  +\,O_\epsilon\big(E(\mathfrak{b})X^{\frac{n^2+n-0.4}{2}+\epsilon}\big).
\end{array}
\end{equation}

Next, we prove that the theorem holds with ``$=$'' replaced by ``$\geq$.'' For any real number $M > 1$, let $\mathfrak{S}(M) \defeq \{B \in \mathfrak{S} : {\vert}\mc{Z}(B){\vert} < M^2\}$. Summing~\eqref{eq-fixedconductoresult} over $\mathfrak{b} < M$ yields the following:
\begin{equation} \label{eq-conductorboundquad}
\begin{array}{rcl}
\displaystyle \sum_{\substack{f \in U(\Z)\\ \on{H}(f) < X }}  \#\left(\frac{\on{inv}^{-1}(f) \cap \mathfrak{S}(M)}{\mc{P}(\Z)}\right) & = & \displaystyle N^{(r)}(X) \cdot \sum_{\mathfrak{b} < M} \prod_{p} \int_{f \in U(\Z_p)} \#\left(\frac{\on{inv}^{-1}(f) \cap \mathfrak{S}(\mathfrak{b})_p}{\mc{P}(\Z_p)}\right) df  \\
& &\qquad\qquad\qquad\quad\,\,\, +\, O_\epsilon\big(E(M)X^{\frac{n^2+n-0.4}{2}+\epsilon}\big). 
\end{array}
\end{equation}
Dividing~\eqref{eq-conductorboundquad} through by $N^{(r)}(X)$ and letting $X \to \infty$, we find that
\begin{equation} \label{eq-conductorboundquad2}
\liminf_{X \to \infty} \frac{\sum_{\substack{f \in U(\Z) \\ \on{H}(f) < X }}  \displaystyle \#\left(\frac{\on{inv}^{-1}(f) \cap \mathfrak{S}}{\mc{P}(\Z)}\right) }{N^{(r)}(X)} \geq \sum_{\mathfrak{b} < M}\prod_{p } \int_{f \in U(\Z_p)} \#\left(\frac{\on{inv}^{-1}(f) \cap \mathfrak{S}(\mathfrak{b})_p}{\mc{P}(\Z_p)}\right)  df.
\end{equation}
Now, letting $M \to \infty$ on the right-hand side of~\eqref{eq-conductorboundquad2} and factoring the sum into an Euler product, we obtain the following:
\begin{align} 
     \sum_{\mathfrak{b} = 1}^\infty\prod_{p } \int_{f \in U(\Z_p)} \#\left(\frac{\on{inv}^{-1}(f) \cap \mathfrak{S}(\mathfrak{b})_p}{\mc{P}(\Z_p)}\right)  df & = \prod_p \sum_{e = 0}^\infty \int_{f \in U(\Z_p)} \#\left(\frac{\on{inv}^{-1}(f) \cap \mathfrak{S}(p^e)_p}{\mc{P}(\Z_p)}\right)  df \nonumber \\ & = \prod_p  \int_{f \in U(\Z_p)} \#\left(\frac{\on{inv}^{-1}(f) \cap \mathfrak{S}_{p}}{\mc{P}(\Z_p)}\right)  df. \label{eq-interchange}
\end{align}
Combining~\eqref{eq-conductorboundquad2} with~\eqref{eq-interchange}, we find that Theorem~\ref{thm-acceptcubic} holds with ``$=$'' replaced by ``$\geq$.'' Note that it is not \emph{a priori} clear whether the infinite sum on the left-hand side of~\eqref{eq-interchange} converges, but even if it were to diverge, it would still be equal to the product on the right-hand side! To show that the sum does indeed converge, one can apply the Jacobian change-of-variables result in Proposition~\ref{prop-jac} to each $p$-adic integral; it is then clear that the summand at $\mathfrak{b}$ is $O\big(\mathfrak{b}^{-1}\prod_{p \mid \mathfrak{b}} p^{-1}\big)$, which is sufficient, as the sum of the reciprocals of the powerful numbers converges (see~\cite{MR266878}).

It thus remains to prove the theorem with ``$=$'' replaced by ``$\leq$.'' Let $\mathfrak{S}(M)' \defeq \mathfrak{S} \smallsetminus \mathfrak{S}(M)$. Then for each $B \in \mathfrak{S}(M)'$, we have that ${\vert}\mc{Z}(B){\vert} \geq M^2$. 
From Theorem~\ref{thm-bsw}, it follows that
\begin{align}
& \sum_{\substack{f \in U(\Z) \\ \on{H}(f) < X }}  \#\left(\frac{\on{inv}^{-1}(f) \cap \mathfrak{S}(M)'}{\mc{P}(\Z)}\right) =O\big(X^{\frac{n^2+n}{2}}/M\big) + O_\epsilon\big(X^{\frac{n^2+n-0.4}{2}+\epsilon}\big). \label{eq-discboundquad}
\end{align}
On the other hand, it follows from~\eqref{eq-conductorboundquad} that
\begin{equation} \label{eq-conductorboundquad3}
\begin{array}{rcl}
\displaystyle 
\sum_{\substack{f \in U(\Z) \\ \on{H}(f) < X }}  \#\left(\frac{\on{inv}^{-1}(f) \cap \mathfrak{S}(M)}{\mc{P}(\Z)}\right) &\leq& \displaystyle N^{(r)}(X) \cdot \prod_{p} \int_{f \in U(\Z_p)} \#\left(\frac{\on{inv}^{-1}(f) \cap \mc{S}_{p}}{\mc{P}(\Z_p)}\right) df \\
& & \qquad\qquad\qquad\quad\, +\, O_\epsilon\big(E(M)X^{\frac{n^2+n-0.4}{2}+\epsilon}\big).
\end{array}
\end{equation}
Taking $M$ to grow as a sufficiently small power of $X$ and combining~\eqref{eq-discboundquad} with~\eqref{eq-conductorboundquad3} yields Theorem~\ref{thm-acceptcubic2}, and hence also Theorem~\ref{thm-acceptcubic}.
\end{proof}
We finish by noting that Theorems~\ref{thm-maincount2} and~\ref{thm-big} follow from Theorem~\ref{thm-acceptcubic} by applying the Jacobian change-of-variables result in Proposition~\ref{prop-jac} to each $p$-adic integral.

\section*{Acknowledgments}

 \noindent We thank Manjul Bhargava for several enlightening conversations, for providing numerous helpful comments and suggestions, and for offering feedback on earlier drafts of this paper. We thank the anonymous referee for giving us detailed feedback that helped us improve the exposition and readability of the paper. We are also grateful to Alex Bartel, Noam D.~Elkies, Andrew Granville, Jef Laga, Aaron Landesman, Aurel Page, Peter Sarnak, and Melanie Matchett Wood for helpful discussions, and to Will Sawin for suggesting the proof of Lemma~\ref{lem-stabismall}.

 Shankar was supported by an National Sciences and Engineering Research Council of Canada discovery grant and a Sloan fellowship. Siad was supported by a Queen Elizabeth II/Steve Halperin Scholarship in Science and Technology at the University of Toronto during the initial phase of this project and is grateful to Princeton University and the Institute for Advanced Study for providing excellent working conditions during its final stage. Swaminathan was supported by the National Science Foundation, under the Graduate Research Fellowship, as well as Award No.~2202839.

	\bibliographystyle{abbrv}
	\bibliography{references}

 \end{document}